\documentclass[10pt,reqno]{amsart}
\usepackage{amsmath,amsthm,amssymb,mathrsfs,graphicx,extpfeil}
\usepackage{epsfig}
\usepackage{indentfirst, latexsym, amssymb, enumerate,amsmath,graphicx}
\usepackage{float} 
\usepackage{colortbl}
\usepackage{epsfig,subfigure}
\usepackage{bm}
\usepackage{empheq}
\usepackage{color}
\usepackage{setspace}
\usepackage{bbm}
\usepackage{lipsum}

\usepackage{amsmath}
\numberwithin{equation}{section}

\newtheorem{theorem}{Theorem}[section]
\newtheorem{lemma}{Lemma}[section]

\newtheorem{proposition}{Proposition}[section]

\newtheorem{remark}{Remark}[section]

\newtheorem{definition}{Definition}[section]

\newcommand{\bu}{\mbox{\boldmath $u$}}
\newcommand{\bw}{\mbox{\boldmath $w$}}
\newcommand{\bb}{\mbox{\boldmath $b$}}
\newcommand{\by}{\mathbf{y}}
\newcommand{\bx}{\mathbf {x}}
\newcommand{\bB}{\mbox{\boldmath $B$}}
\newcommand{\bU}{\mbox{\boldmath $U$}}
\newcommand{\bV}{\mbox{\boldmath $V$}}
\newcommand{\bv}{\mbox{\boldmath $v$}}
\newcommand{\bo}{\boldsymbol{\omega}}
\newcommand{\bxi}{\boldsymbol{\xi}}
\newcommand{\ddx}{\, {\rm d} \bx}
\newcommand{\ddS}{\, {\rm d} S_{\bx}}
\newcommand{\dt}{\, {\rm d} t}
\newcommand{\dxdt}{\, {\rm d}\bx{\rm d}t}

\newcommand{\bq}{\mathbf{q}}
\newcommand{\bn}{\mathbf{n}}
\newcommand{\diver}{\mathrm{div}}

\def\charf {\mbox{{\text 1}\kern-.30em {\text l}}}

\newcommand{\norm}[1]{\big\Vert#1\big\Vert}

\begin{document}
\date{\today}
\title[\hfilneg \hfil ]
{On the Motion of a Body with a Cavity Filled with Magnetohydrodynamic Fluid }

\author[Huang]{Bingkang Huang}
\address[Bingkang Huang]{\newline School of Mathematics, Hefei University of Technology, Hefei 230009, China}
\email{bkhuang92@hotmail.com bkhuang@whu.edu.cn}

\author[M\'acha]{V\'aclav M\'acha}
 \address[V\'aclav M\'acha]{\newline Institute of Mathematics
 Czech Academy of Sciences,
 \v{Z}itn\'{a} 25,
 Czech Republic}
 \email{macha@math.cas.cz}

 \author[Ne\v{c}asov\'{a}]{\v{S}\'{a}rka Ne\v{c}asov\'{a}}

 \address[\v{S}\'{a}rka Ne\v{c}asov\'{a}]{\newline Institute of Mathematics
	 Czech Academy of Sciences,
	 \v{Z}itn\'{a} 25,
	 Czech Republic}
 \email{matus@math.cas.cz}

\thanks{\textbf{Acknowledgment.} V. M. and \v S. N. would like to thank Prof. G. P. Galdi for introducing them this type of FSI problems. 
The research of B.-K. Huang is supported by the grant from National Natural Science Foundation of China (grant No.11901148), the Fundamental Research Funds for the Central Universities (grant No.JZ2022HGTB0257). The research of V. M\'acha and \v S. Ne\v casov\'a is supported by Praemium Academiæ of \v S. Ne\v casov\' a. Moreover, \v S. N. is supported by GC22-08633J. Finally, the Institute of Mathematics, CAS is supported by RVO:67985840.}


\begin{abstract}
We study the dynamics of a coupled system, formed by a rigid body with a cavity entirely filled with magnetohydrodynamic  compressible fluid. Our aim is to derive the global existence of the unique classical solutions and weak solutions to this system. Moreover, we show the weak-strong uniqueness principle which means that a weak solution coincides with a strong solution on the time existence of a strong solution, provided they emanate from the same initial data. 
\\[2mm]
\\[2mm]
\end{abstract}

\maketitle \centerline{\date}



\section{Introduction and Main Results}
We consider an insulating rigid body  $\mathcal{B}$ with a simply connected cavity $\mathcal{C}$ (both in the physical space $\mathbb R^3$) entirely filled with a  magnetohydrodynamic compressible fluid $\mathcal{F}$ and we denote  $\mathcal{S}:=\mathcal{B}\cup\mathcal{C}$. The motion of this coupled system is governed by the following system
\begin{equation}\label{mhd-y}
\left.
\begin{split}
 r_t+ \mathrm{div}_{\by} (r\bw)=0&\\
(r\bw)_t+\mathrm{div}_{\by}(r\bw\otimes \bw) =\mathrm{div}_{\by}T(\bw, p(r)) +\mathrm{div}_{\by}\left(\bB\otimes\bB-\frac{1}{2}|\bB|^2\mathbb{I}\right)&\\
\bB_t+\bw\cdot\nabla_{\by}\bB+\bB\mathrm{div}_{\by}\bw-\bB\cdot\nabla_{\by}\bw=-\nabla_\by\times(\nabla_\by\times\bB), \quad \mathrm{div}_{\by}\bB=0&
\end{split}
\right\}
\end{equation}
\hspace{9cm}$\forall (t,\by) \in \cup_{t>0}\{t\}\times\mathcal{C}(t)$,
\begin{equation}\label{boundary-y}
\left.
\begin{split}
\bw={\bar{\bw}}(t)\times(\by-\by_{c})+{\boldsymbol\eta}(t)&\\
 \quad \bB\cdot \boldsymbol{N}=0,\quad (\nabla_\by \times \bB) \times \boldsymbol{N} = 0 &
 \end{split}
 \right\}\forall (t,\by) \in \cup_{t>0}\{t\}\times\partial\mathcal{C}(t),
\end{equation}

\begin{equation}\label{com}
\left.
\begin{split}
 \frac{\mathrm{d}}{\mathrm{d}t}(\boldsymbol{J}_{c}\cdot{\bar{\bw}})=-\int_{\partial\mathcal{C}(t)}(\by-{\by}_{c})\times \left(T(\bw, p(r)) + \bB\otimes \bB - \frac 12 |\bB|^2\mathbb I\right)\cdot \boldsymbol{N}&\\
m_{\mathcal{B}}\boldsymbol{\eta}(t)=-\int_{\mathcal{C}(t)}r\bw&
\end{split}
\right\}  \forall t\in(0,\infty),
\end{equation}
where $r$, $\bw$, $\bB$ represent the density of the fluid,  the velocity, and the magnetic fields. Next, $\bar{\bw}$ is the angular velocity of $\mathcal{B}$, and $\boldsymbol{\eta}$ denotes the velocity of its center of mass $c$. Furthermore, $\by_c$ is the position of $c$, $m_{\mathcal{B}}$ and $\boldsymbol{J}_{c}$ are (respectively) the mass of the rigid body and its inertia tensor with respect to $\by_c$. Further, $\boldsymbol N$ is a unit outer normal vector on the boundary $\partial \mathcal{C}$. The stress tensor  has a form
\begin{equation*}
T(\bw, p(r)):=S(\nabla_{\by}\bw)-p(r)\mathbb{I},
\end{equation*}
with 
\begin{equation}\label{S}
\begin{split}
S(\nabla_{\by}\bw)&:=\mu\mathcal{D}_\by(\bw)+\left(\lambda-\frac{2}{3}\mu\right)\mathbb{I}\mathrm{div}_{\by}\bw,\\
\mathcal{D}_\by(\bw)&:=\nabla_{\by}\bw+\nabla_{\by}\bw^{\top},\quad \mu>0,\quad \lambda\geq 0,
\end{split}
\end{equation}
where $\mathbb I$ denotes an identity $3\times 3$ matrix.
The pressure $p(r)$ is given by 
\begin{equation*}
p(r):=ar^{\gamma}, 
\end{equation*}
where $\gamma >1$ and $a>0$ are given. We also assume that the density of $B$ is $1$.

We refer $\eqref{mhd-y}_{1}$ as the mass balance equation and $\eqref{mhd-y}_{2}$ as the conservation law of momentum. Next, $\eqref{mhd-y}_{3}$ is called the induction equation. We assume that the body is insulated and that the fluid sticks to the boundary -- this is expressed by the boundary conditions $\eqref{boundary-y}_1$. 
The movement of the body $\mathcal B$ is governed by the Newton laws and we assume there is no outer force. This is expressed by the boundary condition \eqref{com} -- we assume that the center of mass of $\mathcal S$ cannot be moved by the interior forced and it is originally in rest $\eqref{com}_2$ and the balance of angular momentum  $\eqref{com}_1$ holds.

In order to reformulate the problem in the time-independent fluid domain, we proceed with the following transformation. Let $\boldsymbol{Q}(t)$ be the family of proper orthogonal transformations
\begin{equation*}
\frac{\mathrm{d}}{\mathrm{d}t}\boldsymbol{Q}(t)=\mathbb{S}({\bar{\bw}})\cdot\boldsymbol{Q}(t), \quad \boldsymbol{Q}(0)=\mathbb I,
\end{equation*}
where
\begin{equation*}
  \mathbb{S}({\bar{\bw}}):=\left(  \begin{array}{ccc}
        0 & -\bar{\bw}^3 & \bar{\bw}^2\\
        \bar{\bw}^3 & 0 & -\bar{\bw}^1\\
        -\bar{\bw}^2 & \bar{\bw}^1 & 0\\
    \end{array}\right).
\end{equation*}

We set
\begin{equation}\label{coordinate-change}
\bx=\boldsymbol{Q}^\top(t)\cdot(\by-\by_{c}), 
\end{equation}
and we introduce new notations for the transformed quantities:
\begin{equation}\label{new-sym}
\begin{split}
&\rho(t, \bx):=r(t, \boldsymbol{Q}\cdot \bx+\by_{c}), \quad \bu(t, \bx):=\boldsymbol{Q}^\top\cdot\bw(t, \boldsymbol{Q}\cdot \bx+\by_{c}),\\
&\boldsymbol{\omega}(t):=\boldsymbol{Q}^{\top}\cdot\bar{\bw}(t), \quad \boldsymbol{\xi}(t):=\boldsymbol{Q}^{\top}\cdot\boldsymbol{\eta}(t),\\
&\boldsymbol{I}_{c}:=\boldsymbol{Q}^{\top}\cdot\boldsymbol{J}_{c}(t)\cdot\boldsymbol{Q},\quad \boldsymbol{n}:=\boldsymbol{Q}^{\top}\cdot\boldsymbol{N}(t),\quad  \bb(t, \bx):=\boldsymbol{Q}^\top\cdot\bB(t, \boldsymbol{Q}\cdot \bx+\by_{c}).
\end{split}
\end{equation}

Based on \eqref{coordinate-change} and \eqref{new-sym} we rewrite the original system \eqref{mhd-y}--\eqref{boundary-y} as follows
\begin{equation}\label{mhd-x}
\left.
\begin{split}
 \rho_t+ \mathrm{div}_{\bx} (\rho\bv)=0&\\
(\rho\bu)_t+\mathrm{div}_{\bx}(\rho\bv\otimes \bu)+\rho\boldsymbol{\omega}\times \bu+\nabla_{\bx}p(\rho) =\mathrm{div}_{\bx}S(\nabla_{\bx}\bu) +\mathrm{div}_{\bx}\left(\bb\otimes\bb-\frac{1}{2}|\bb|^2\mathbb{I}\right)&\\
\bb_t+(\bo\times \bb)-(\bo\times \bx+\bxi)\cdot\nabla_{\bx}\bb+\bu\cdot \nabla_{\bx}\bb+\bb\mathrm{div}_{\bx}\bu-\bb\cdot\nabla_{\bx}\bu=-\nabla_\bx \times(\nabla_\bx \times \bb), \quad \mathrm{div}_{\bx}\bb=0&
\end{split}
\right\}
\end{equation}
\hspace{9cm}$\forall (t,\bx) \in (0,\infty)\times \mathcal{C},$
\begin{equation*}
\left.
\begin{split}
\bu={\bo}(t)\times\bx+{\bxi}(t)&\\
\bb\cdot \bn=0,\quad (\nabla_{\bx}\times \bb)\times\bn=0&
\end{split}
\right\}\forall(t,\bx) \in (0,\infty)\times\partial\mathcal{C},
\end{equation*}
\begin{equation}\label{com-x}
\left.
\begin{split}
 \boldsymbol{I}_c\cdot \frac{\mathrm{d}}{\mathrm{d}t}\bo+\bo\times(\boldsymbol{I}_c\cdot\bo)=-\int_{\partial\mathcal{C}}\bx\times\left( T(\bu, p(\rho)) + \bb\otimes\bb - \frac 12 |\bb|^2\mathbb I\right)\cdot \boldsymbol{n}\ddx &\\
m_{\mathcal{B}}\boldsymbol{\xi}(t)=-\int_{\mathcal{C}}\rho\bu\ddx &
\end{split}
\right\}  \forall t\in(0,\infty),
\end{equation}
where 
 $\bv:=\bu-\bo\times\bx-\bxi$ denotes the relative velocity of the fluid velocity with respect to the rigid body $\mathcal{B}$.


We proceed similarly as in \cite{DGMZ-2016-ARMA,G-2002-Elsevier} to deduce the equivalent system
\begin{equation}\label{mhd-x-new}
\left.
\begin{split}
 \rho_t+ \mathrm{div}_{\bx} (\rho\bv)=0&\\
(\rho\bu)_t+\mathrm{div}_{\bx}(\rho\bv\otimes \bu)+\rho\boldsymbol{\omega}\times \bu+\nabla_{\bx}p(\rho) =\mathrm{div}_{\bx}S(\nabla_{\bx}\bu) +\mathrm{div}_{\bx}\left(\bb\otimes\bb-\frac{1}{2}|\bb|^2\mathbb{I}\right)&\\
\bb_t+(\bo\times \bb)-(\bo\times \bx+\bxi)\cdot\nabla_{\bx}\bb+\bu\cdot \nabla_{\bx}\bb+\bb\mathrm{div}_{\bx}\bu-\bb\cdot\nabla_{\bx}\bu=-\nabla_\bx \times(\nabla_\bx \times \bb), \quad \mathrm{div}_{\bx}\bb=0&
\end{split}
\right\}
\end{equation}

$\hspace{9cm}\forall (t,\bx) \in (0,\infty)\times \mathcal{C},$
\begin{equation*}
\left.
\begin{split}
\bu={\bo}(t)\times\bx+{\bxi}(t)&\\
\bb\cdot \bn=0,\quad (\nabla_\bx \times\bb)\times\bn=0&
\end{split}
\right\} \forall (t,\bx) \in (0,\infty)\times \partial\mathcal{C},
\end{equation*}

\begin{equation*}
\left.
\begin{split}
\frac{\mathrm{d}}{\mathrm{d}t}\boldsymbol{M}+\bo\times\boldsymbol{M}=0, \quad \boldsymbol{M}=\boldsymbol{I}_C\cdot\bo+\int_{\mathcal{C}}\rho\bx\times\bu\ddx&\\
m_{\mathcal{B}}\boldsymbol{\xi}(t)=-\int_{\mathcal{C}}\rho\bu\ddx&
\end{split}
\right\}  \forall t\in(0,\infty).
\end{equation*}

In this paper we consider the coupled system \eqref{mhd-x}--\eqref{com-x} with an initial condition
\begin{equation}\label{initial-con}
(\rho, \bu, \bb)|_{t=0}=(\rho_0, \bu_0, \bb_0).
\end{equation}

We formally integrate $\eqref{mhd-x}_1$ over $\mathcal{C}\times[0, t]$ to deduce that 
\begin{equation*}
\int_{\mathcal{C}}\rho(t, \bx)\ddx=\int_{\mathcal{C}}\rho_0(\bx)\ddx=m_{\mathcal{F}}, \quad   \forall t\geq 0, 
\end{equation*}
where $m_\mathcal{F}$ denotes the mass of the fluid's part.

Many researchers are committed to the mathematical analysis of fluid-structure interaction (FSI) system with the moving rigid body containing a cavity filled by the fluid. This kind of system not only plays an important role in mathematics but also in engineering. Early results on the motion of a rigid body coupled with a liquid, which is a specific problem in the FSI system, could go back to   \cite{C-1972-NASA, MR-1968-Springer, S-1954-Russian} and \cite{Z-1885-RJPCS}. There is an enormous amount of literature devoted to this coupled system. Broadly speaking, these types of FSI systems could be classified into two types: either the fluid is incompressible or the fluid is compressible. For the first case,  here we review some interesting literature. The nonlinear instability of a body coupled with the incompressible Euler equations was given in \cite{FL-1998-SIAM}. The time-periodic weak solutions to a rigid body filled with an incompressible Newtonian fluid were constructed in \cite{GMM-2016-Springer}. Moreover, the large-time behavior of weak solutions was obtained in \cite{GM-2016-Cambridge}. In \cite{DGMZ-2016-ARMA, G-2017-Springer}, the authors considered the stability  of this coupled system without external forces. The finite energy weak solutions to a deformable body in the incompressible Newtonian fluid were established in \cite{NTT-2011-AAM}. We emphasize that the weak solutions to the incompressible viscous fluid-solid system equipped with a Navier condition were constructed in \cite{MPS}.
Specifically, the solvability of the weak solution and the ultimate motion of the FSI system have been shown in \cite{GM-2018-QAM, MPS-2018-JMFM, ST-2013-PRSE}. Some significant results have been also achieved for the FSI system with compressible fluid. The existence and large-time behavior of strong solutions to a rigid body filled with a compressible Newtonian fluid was given in \cite{GMN-2019-ARMA}, and the existence of weak solutions was investigated in \cite{GMN-2020-IJNM}. The weak solution to the system governing the flow of compressible viscous fluids around a rotating obstacle was established in \cite{KNN-2014-AUFS}. The finite energy weak solutions to a deformable body in a viscous compressible fluid which lies in bounded and unbounded domains were presented in  \cite{ MN-2016-M3AS, MN-2016-PRSES}. The existence of a weak solution to a nonlinear fluid-elastic structure interaction problem with heat exchange was considered in \cite{MMNRT-2022}.
The interested reader can find more in \cite{ CAL-2017-Berlin, L-1992-SAACM, L-1993-JJIAA, RoyT} and in the references quoted there.
Let us mention that several results in the case of the motion of rigid bodies inside of the bounded domain filled by incompressible/compressible fluids were investigated by several authors.  We can refer to e.g.\cite{DEES2, GLSE, CST, GH-2014-CPAM, F4, NRRS, KrNePi2}.  

Let us also mention recent results on the uniqueness and regularity of the FSI problem in the incompressible case with one rigid body under the assumption that the velocity field belongs to the Prodi-Serrin regularity class, see \cite{M1, M2}.

 The case when a rigid body is moving in a bounded domain filled by incompressible or compressible fluids with the electromagnetic field was investigated, see \cite{BNSS,S}.
To our knowledge, no results are available on the dynamics of a rigid body with a cavity filled with compressible magnetohydrodynamic equations. Therefore the following problems are treated in this paper: {\it is it possible to obtain the global existence of the strong and weak solutions, and the weak-strong uniqueness principle to a rigid body with a cavity filled with a magnetohydrodynamic fluid?} 
We adopt the following strategies to answer the previous questions:
\begin{itemize}
	\item The global existence of the strong solution is achieved by the method of A. Valli \cite{V-1983-ASNSP} modified so that we handle also the Newton laws \eqref{com-x} and the induction equation \eqref{mhd-x}$_3$.
	\item To show the existence of the  weak solutions to system \eqref{mhd-x}--\eqref{com-x}, we introduce new artificial viscosity and pressure terms to solve an approximation system. With uniform bounds in hand, we use compactness arguments to perform limits and to construct the finite energy weak solutions.
    \item  The weak-strong uniqueness principle for Navier-Stokes (N-S) equations is widely recognized. 
    We derive this property for system \eqref{mhd-x}--\eqref{com-x} to assert that a weak solution is the same as a strong solution, provided they emanate from the same initial data. 
\end{itemize}

We give the global-in-time existence of strong solutions to $\eqref{mhd-x}$-$\eqref{com-x}$.
\begin{theorem}\label{global-existence-strong-solution}
Let $\mathcal{C}$ be of class $C^4$, and let 
	\begin{equation*}
	\begin{split}
	&\bu_0 \in W^{1,2}_{(\mathcal{S})}, \hspace{0.2cm}\bu_0|_{\mathcal{C}}\in H^2_{(\mathcal{C})}, \hspace{0.2cm}\bb_0\in H^2_{(\mathcal{C})},\\
	&\rho_0 \in H^2_{(\mathcal{C})}\cap L^\gamma_{(\mathcal{C})}, \hspace{0.2cm}\gamma>1.
	\end{split}
	\end{equation*}
	Assume, moreover, that there exists a sufficiently small positive constant $\epsilon$ such that $\mathsf{E}$ and $\mathsf{F}$ given by $\eqref{define-e}$ and $\eqref{F}$ satisfying $\mathsf{E}(0)+\mathsf{F}^{\frac{1}{2}}(0)\leq \epsilon$. Then there is a unique solution $(\rho, \bu, \bb)$ to $\eqref{mhd-x-new}$--$\eqref{initial-con}$ which satisfies 
	\begin{equation}\label{regularity}
	\begin{split}
	&\rho \in C(0,T;H^2_{(\mathcal{C})}),  \hspace{0.2cm} \rho_t\in C(0,T;H^1_{(\mathcal{C})}),\\
	&\bu \in C(0,T;W^{1,2}_{(\mathcal{S})}),  \hspace{0.2cm}\bv \in C(0,T;H^2_{(\mathcal{C})})\cap L^2(0,T;H^3_{(\mathcal{C})}),\\
	&\bu_t\in C(0,T;L^2_{(\mathcal{S})})\cap L^2(0,T;W^{1,2}_{(\mathcal{S})}),\\
	&\bb \in C(0,T;H^2_{(\mathcal{C})})\cap L^2(0,T;H^3_{(\mathcal{C})}),  \\
	&\bb_t\in C(0,T;L^2_{(\mathcal{C})})\cap L^2(0,T;H^{1}_{(\mathcal{C})}),\hspace{0.2cm}\forall T \in (0,\infty).
	\end{split}
	\end{equation}
\end{theorem}
 We also prove the existence of finite energy weak solutions to $\eqref{mhd-x}$-$\eqref{com-x}$.
\begin{theorem}\label{weak-solution-existence}
	Assume that $\mathcal{S}$ is a bounded domain and $\partial\mathcal{C}$ is in class of $C^{2+\nu}$ with $\nu>0$. Let $\gamma>\frac{3}{2}$, $\rho_0\geq 0$, $\rho_0\in L^\gamma_{(\mathcal{S})}$, $(\rho\bu)_0=\bq\in L^1_{(\mathcal{S})}$, where $\bq=0$ on the set $\rho_0=0$ and $\int_{\mathcal{S}}\bq\ddx=0$. Moreover, $\bb_0\in L^2_{(\mathcal{C})}$, $\mathrm{div}_{\bx}\bb_0=0$ in $\mathcal{D}'_{(\mathcal{C})}$. Then there exists a weak solution to coupled system \eqref{mhd-x}--\eqref{com-x} in the sense of  Definition $\ref{weak-solution}$.
\end{theorem}
Finally, the weak-strong uniqueness principle is established. 
\begin{theorem}\label{w.s.strong.thm}
	Assume that  $(\rho, \bu, \bb)$ is a weak solution to the system $\eqref{mhd-x}$-$\eqref{com-x}$ established in Theorem $\ref{weak-solution-existence}$. Let $(r, \bU, \bB)$ be a strong solution to the system $\eqref{mhd-x}$-$\eqref{com-x}$  which emanates from the same initial data, and satisfies 
	\begin{equation*}
		\begin{split}
	\left(r,\frac{1}{r},\nabla_{\bx}r,  \bU,\nabla_{\bx}\bU,\bU_t, \bB,\nabla_{\bx}\bB,\bB_t \right)\in L^\infty(0,T; \mathcal{C}). 
\end{split}
\end{equation*}
Then it holds that 
\begin{equation*}
\begin{split}
\rho=r,\ \mbox{a.e. in }\mathcal{\mathcal{C}},\ \bu=\bU\ \mbox{a.e. in }\mathcal{S}, \ \bb=\bB\ \mbox{a.e. in }\mathcal{C}.
\end{split}
\end{equation*}
\end{theorem}
\begin{remark}\ \\  
\begin{itemize}
\item Under the assumption of the smallness of the initial data, we can construct the global existence of strong solutions to $\eqref{mhd-x}$-$\eqref{com-x}$ for any  $p(\rho)$ increasing and $C^1$.
\end{itemize}
\begin{itemize}
	\item Condition $\gamma>\frac{3}{2}$ is necessary to construct a weak solution, we also mention that the existence of finite energy weak solution is valid for general initial data. 
\end{itemize}
\begin{itemize}
	\item We cannot guarantee the uniqueness of the weak solutions. However, we utilize the relative entropy inequality to show that a weak solution coincides with a strong solution as long as they have the same initial data.
\end{itemize}
\end{remark}
\begin{remark}
Let us stress that since we consider that the magnetic field on the boundary $\cup_{t>0}\{t\}\times\partial\mathcal{C}(t)$  satisfies $\bB\cdot \boldsymbol{N}=0,\quad (\nabla_\by \times \bB) \times \boldsymbol{N} = 0$ it gives us contribution in the conservation of the angular momentum, see (\ref{com}).
\end{remark}
\subsection{Notations}
 Throughout the rest of this paper, $C$ will be used to denote a generic positive constant which may change line by line but it is always independent of the solution. We use $\epsilon$ to denote the small positive constant. Furthermore, $A \sim B$ means that there exist positive constants $C_1$ and $C_2$ such that $C_1A\leq B\leq C_2A$. The Lebesque spaces are denoted by $L^p_{(\mathcal C)}$, $W^{n, p}_{(\mathcal{C})}$ denotes usual Sobolev spaces with standard norm $\|\cdot\|_{W^{n, p}_{(\mathcal{C})}}$, $H^n_{(\mathcal{C})}:=W^{n,2}_{(\mathcal{C})}$. We further define function spaces
\begin{equation*}
W^{n, p}_{(\mathcal{S})}=\{\boldsymbol{\phi}:\mathcal{S}\to\mathbb{R}^3, \quad \|\boldsymbol{\phi}\|_{W^{n, p}_{(\mathcal{S})}}\leq \infty, \quad \boldsymbol{\phi}|_{\mathcal{B}}=\boldsymbol{l}\times \bx+\boldsymbol{k},  \quad \boldsymbol{l},\boldsymbol{k}\in \mathbb{R}^3\}.
\end{equation*}

Particularly, $L^{p}_{(\mathcal{S})}=W^{0, p}_{(\mathcal{S})}$. 
For any $\boldsymbol{\phi}\in W^{n, p}_{(\mathcal{S})}$, we use    $\boldsymbol{l}_{\boldsymbol{\phi}}$ and $\boldsymbol{k}_{\boldsymbol{\phi}}$ to emphasize that these vectors are associated with $\boldsymbol{\phi}$. 
We also define
\begin{equation*}
\mathcal{V}_{(\mathcal{S})}=\{\boldsymbol{\varphi}:\mathcal{S}\to\mathbb{R}^3, \quad \|\boldsymbol{\varphi}\|_{C^{\infty}_{(\mathcal{S})}}\leq \infty, \quad \boldsymbol{\varphi}|_{\mathcal{B}}=\boldsymbol{l}_{\boldsymbol{\varphi}}\times \bx+\boldsymbol{k}_{\boldsymbol{\varphi}},  \quad \boldsymbol{l}_{\boldsymbol{\varphi}},\boldsymbol{k}_{\boldsymbol{\varphi}}\in \mathbb{R}^3\}.
\end{equation*}

Especially, due to the symmetry of matrix $S$ in $\eqref{S}$,  we have
\begin{equation}\label{symmetry}
\begin{split}
&S(\nabla_{\bx}\boldsymbol{\psi}):\nabla_{\bx}\boldsymbol{\phi}=S(\nabla_{\bx}\boldsymbol{\psi}):\nabla_{\bx}(\boldsymbol{\phi}-\boldsymbol{l}_{\boldsymbol{\phi}}\times \bx),\\
&S(\nabla_{\bx}\boldsymbol{\phi}):\nabla_{\bx}\boldsymbol{\phi}=S(\nabla_{\bx}(\boldsymbol{\phi}-\boldsymbol{l}_{\boldsymbol{\phi}}\times \bx)):\nabla_{\bx}(\boldsymbol{\phi}-\boldsymbol{l}_{\boldsymbol{\phi}}\times \bx),
\end{split}
\end{equation}
for any $\boldsymbol{\psi}, \boldsymbol{\phi} \in W^{n, p}_{(\mathcal{S})}$.

The contents of this paper are organized into five sections. After this introduction and main results, which constitute Section 1, we give the local existence of strong solution in Section 2. Section 3 mainly concerns the global existence of strong solutions for small initial data. Section 4 provides the existence of the finite energy weak solution. Furthermore, Section 5 is devoted to the weak-strong uniqueness principle.

\section{Local Existence of Strong Solutions to Coupled System}

We introduce new quantities 
\begin{equation*}
\begin{split}
\bar{\rho}=\frac{1}{\mathrm{vol}(\mathcal{C})}\int_{\mathcal{C}}\rho\mathrm{d}{\bx}, \quad \varrho=\rho-\bar{\rho}, 
\end{split}
\end{equation*}
and we derive an equivalent system of equations that has a form
\begin{equation}\label{re-system-mass}
\left.
\begin{split}
&\varrho_t+\bv\cdot\nabla_{\bx}\varrho+\varrho\mathrm{div}_{\bx}\bv+\bar{\rho}\mathrm{div}_{\bx}\bv=0, \quad \forall (t,\bx)\in (0,\infty)\times\mathcal{C},\\
&\int_{\mathcal{C}}\varrho(t, \bx)\ddx=0, \quad \forall t\in(0, \infty),
\end{split}\right\}
\end{equation}
\begin{equation}\label{re-system-momentum}
\left.
\begin{split}
&\int_{\mathcal{S}}\rho_{\mathcal{S}}(\bu_t\cdot \boldsymbol{\phi}+\bv\cdot\nabla_{\bx}\bu\cdot\boldsymbol{\phi}+\bo_{\bu}\times \bu\cdot\boldsymbol{\phi})\ddx-\int_{\mathcal{S}}p(\rho)\mathrm{div}_{\bx}\boldsymbol{\phi}\ddx\\
&+\int_{\mathcal{S}}S(\nabla_{\bx}\bu):\nabla_{\bx}\boldsymbol{\phi}\ddx+\int_{\mathcal{S}}\left(\bb\otimes\bb-\frac{1}{2}|\bb|^2\mathbb{I}\right):\nabla_{\bx}\boldsymbol{\phi}\ddx=0
\end{split}\right\}\forall \boldsymbol{\phi}\in W^{1,2}(\mathcal{S}), \forall t\in (0,\infty),
\end{equation}

\begin{equation}\label{re-system-induction}
\left.
\begin{split}
\bb_t+(\bo_{\bu}\times \bb)-(\bo_{\bu}\times \bx+\bxi_{\bu})\cdot\nabla_{\bx}\bb+\bu\cdot\nabla_{\bx}\bb+\bb\mathrm{div}_{\bx}\bu-\bb\cdot\nabla_{\bx}\bu=-\nabla_\bx \times(\nabla_\bx \times \bb)&\\
\mathrm{div}_{\bx}\bb=0&\\
 \quad \bb\cdot \bn=0,\quad (\nabla_\bx\times \bb)\times \bn = 0 \quad on\quad  (0,\infty)\times\partial\mathcal{C}&
\end{split}\right\}  
\end{equation}
\begin{equation}\label{re-system}
\begin{split}
&m_{\mathcal{B}}\bxi_{\bu}=-\int_{\mathcal{C}}\rho\bu\ddx, \quad \forall t\in(0, \infty), 
\end{split}
\end{equation}
where 
\begin{equation*}
\rho_{\mathcal{S}}=\left\{
\begin{split}
&\rho(t, \bx), \quad \forall \bx \in \mathcal{C},\\
&1, \quad \forall \bx \in \mathcal{B}.
\end{split}
\right.
\end{equation*}
For the compatibility between the notions of weak and strong solutions we refer to \cite{GMN-2019-ARMA}.

We state the local existence of strong solutions to the coupled system $\eqref{re-system-mass}-\eqref{re-system}$.
\begin{theorem}\label{local-existence}
Let $\mathcal{C}$ be of class $C^4$ and the initial data satisfy
\begin{equation*}
\begin{split}
&\rho_0\in W^{2,2}_{(\mathcal{C})}, \quad 0< \underline{m}\leq \rho_0\leq \bar{m}, \\
&\bu_0\in W^{1,2}_{(\mathcal{S})}, \quad \bu_0|_{\mathcal{C}}\in W^{2,2}_{(\mathcal{C})}, \quad \bb_0|_{\mathcal{C}}\in W^{2,2}_{(\mathcal{C})}.
\end{split}
\end{equation*}
Then there exists a small time $T^*>0$, such that $(\rho, \bu, \bb )$ is a unique solution to $\eqref{mhd-x}$-$\eqref{com-x}$.
Moreover, $(\rho, \bu, \bb )$ satisfies that 
\begin{equation*}
\begin{split}
&\rho(t, \bx) >0,\quad \forall(t, \bx)\in [0,T^*]\times \mathcal{C},\quad \rho\in C(0,T^*; H^{2}_{(\mathcal{C})}), \quad \rho_t\in C(0, T^*; H^{1}_{(\mathcal{C})}), \\
&\bu \in C(0,T^*; W^{1,2}_{(\mathcal{S})}), \quad \bu|_{\mathcal{C}}\in L^2(0,T^*; H^{3}_{(\mathcal{C})})\cap C(0,T^*;H^{2}_{(\mathcal{C})}),\\
& \bu_t\in L^2(0,T^*; W^{1,2}_{(\mathcal{S})})\cap C(0,T^*; L^2_{(\mathcal{S})}),\\
&\bb\in L^2(0,T^*;H^{3}_{(\mathcal{C})})\cap C(0,T^*; H^{2}_{(\mathcal{C})}), \quad\bb_t\in L^2(0,T^*; H^{1}_{(\mathcal{C})})\cap C(0,T^*; L^2_{(\mathcal{C})}).
\end{split}
\end{equation*}
\end{theorem}

The proof of the above theorem is established by considering three linear equations and Schauder fixed-point argument. We just sketch the outline of the proof and present the result, interested readers can refer the Section 4.3 in \cite{V-1983-ASNSP}.

\subsection{Continuity Equation}
For given $\bar{\bv}$ and $\varrho_0$, we treat the linear continuity equation
\begin{equation}\label{linear-mass}
\begin{split}
\varrho_t+\bar{\bv}\cdot \nabla_{\bx}\varrho+\varrho\mathrm{div}_{\bx}\bar{\bv}+\bar{\rho}\mathrm{div}_{\bx}\bar{\bv}&=0,\quad \forall(t, \bx)\in (0, T)\times\mathcal{C},\\
\varrho|_{t=0}&=\varrho_0.
\end{split}
\end{equation}

The following can be found as \cite[Lemma 2.3]{V-1983-ASNSP}.
\begin{lemma} Let $\partial \mathcal{C}$ be of class $C^1$ and let us assume that 
\begin{equation*}
\begin{split}
&\bar{\bv}\in L^1(0,T;H^{3}_{(\mathcal{C})}), \quad   \bar{\bv}\cdot \boldsymbol{n}=0,\\
&\varrho_0\in H^{2}_{(\mathcal{C})}, \quad \int_{\mathcal{C}}\varrho_0\mathrm{d}{\bx}=0.
\end{split}
\end{equation*}
Then there exists a unique solution to $\eqref{linear-mass}$ such that 
\begin{equation*}
\begin{split}
\varrho\in C(0, T; H^{2}_{(\mathcal{C})}), \quad \int_{\mathcal{C}}\varrho(t, \bx)\mathrm{d}{\bx}=0,
\end{split}
\end{equation*}
and 
\begin{equation*}
\begin{split}
\|\varrho\|_{L^\infty(0,T;H^{2}_{(\mathcal{C})})}\leq C(\|\varrho_0\|_{H^{2}_{(\mathcal{C})}}+1)\exp(\|\bar{\bv}\|_{L^1(0,T;H^{3}_{(\mathcal{C})})}).
\end{split}
\end{equation*}
Furthermore, if $\bar{\bv}\in C([0,T], H^{2}_{(\mathcal{C})})$, we could obtain  $\varrho_t\in C([0,T],H^{1}_{(\mathcal{C})})$ and 
\begin{equation*}
\begin{split}
\|\varrho_t\|_{L^\infty(0,T;H^{1}_{(\mathcal{C})})}\leq C\|\bar{\bv}\|_{L^\infty(0,T;H^2_{(\mathcal{C})})}(\|\varrho_0\|_{H^{2}_{(\mathcal{C})}}+1)\exp(\|\bar{\bv}\|_{L^1(0,T;H^{3}_{(\mathcal{C})})}).
\end{split}
\end{equation*}
\end{lemma}

\subsection{Linear Magnetic Equation}
When $\bar{\bv}\in W^{1,2}_{(\mathcal{S})}$ is given, we solve the magnetic equation with the slip boundary. (For more information on 3-D MHD with slip boundary, please check \cite{CHPS-23}.)
\begin{equation}\label{linear-magnetic-equation}
\begin{split}
\bb^j_t+(\bo_{\bar{\bv}}\times \bb)^j-(\bo_{\bar{\bv}}\times \bx+{\boldsymbol{\xi}_{\bar{\bv}}})\cdot\nabla_{\bx}\bb^j+\mathrm{div}_{\bx}(\bb^j\bar{\bv}-\bar{\bv}^j\bb)&=-\left(\nabla_\bx \times(\nabla_\bx \times \bb)\right)^j,\hspace{0.2cm} j=1,2,3,\\
\mathrm{div}_{\bx}\bb&=0,\quad \forall (t, \bx)\in (0, T)\times\mathcal{C}, \\
{\bb\cdot\bn}|_{\partial \mathcal C} = 0,\quad {(\nabla_\bx\times \bb)\times \bn}|_{\partial \mathcal C} = 0,\quad \bb|_{t=0}&=\bb_0.
\end{split}
\end{equation}

Since $\eqref{linear-magnetic-equation}$ is a linear parabolic system (see also \cite{KoYa}) one can deduce energy estimates which yield the following lemma.
\begin{lemma}  
Let
\begin{equation*}
\begin{split}
\bar{\bv}|_{\mathcal{C}}\in L^2(0,T; H^{3}_{(\mathcal{C})})\cap C(0,T; H^{2}_{(\mathcal{C})}),\quad \bb_0 \in H^{2}_{(\mathcal{C})},\quad \diver_\bx \bb_0 = 0.
\end{split}
\end{equation*}
Then there exists a unique solution $\bb(t, \bx)$ to $\eqref{linear-magnetic-equation}$ such that 
\begin{equation*}
\begin{split}
&\bb\in L^2(0,T; H^{3}_{(\mathcal{C})})\cap C(0,T; H^{2}_{(\mathcal{C})}),\\
&\bb_t\in L^2(0,T; H^{1}_{(\mathcal{C})})\cap C(0,T; L^2_{(\mathcal{C})}).
\end{split}
\end{equation*}
\end{lemma}

\begin{proof}
Existence and the regularity of $\bb_t$ can be found as \cite[Lemma 2.2]{FaYu}. Since the time derivative can be then switched to the right hand side, the regularity of $\bb$  can be deduced from the regularity of solutions to the stationary problem -- we refer to \cite{georgescu} (see also \cite{KoYa}).
\end{proof}

\subsection{Linear Momentum System}
For given $\tilde{\rho}(t, \bx)>0$, $\bm{F}\in (W^{1,2}_{(\mathcal{S})})^*$, we define
\begin{equation*}
\tilde{\rho}_{\mathcal{S}}:=\tilde{\rho}_{\mathcal{S}}(t, \bx)=\left\{
\begin{split}
&\tilde{\rho}(t, \bx), \quad \forall \bx \in \mathcal{C},\\
&1, \quad \forall \bx \in \mathcal{B}.
\end{split}
\right.
\end{equation*}
We investigate 
\begin{equation}\label{linear-momentum-system}
\begin{split}
\int_{\mathcal{S}}\tilde{\rho}_{\mathcal{S}}{\bu}_t\cdot\boldsymbol{\phi}\ddx+\int_{\mathcal{S}}S(\nabla_{\bx}\bu):\nabla_{\bx}\boldsymbol{\phi}\ddx&=\int_{\mathcal{S}}\bm{F}\cdot\boldsymbol{\phi}\ddx,\\
\bu|_{t=0}&=\bu_0,\\
\boldsymbol{\xi}_{\bu}&=-\frac{1}{m_{\mathcal{B}}}\int_{\mathcal{C}}\tilde{\rho}\bu\ddx,
\end{split}
\end{equation}
where $\boldsymbol{\phi}\in W^{1,2}_{(\mathcal{S})},$ and $\bu=\bv+\bo_{\bu}\times \bx+\boldsymbol{\xi}_{\bu}$.

The following can be found in \cite[Section 4.2]{V-1983-ASNSP}.
\begin{lemma}
Let $\partial \mathcal{C}$ be of class $C^2$.
Assume the existence of $\bar m>0$ such that
\begin{equation*}
\begin{split}
&\tilde{\rho}\in L^\infty(0, T;L^\infty_{(\mathcal{C})}), \quad \frac{\bar{m}}{2}\leq \tilde {\rho}(t, \bx)\leq 2\bar{m},\quad \forall(t, \bx)\in [0, T]\times\mathcal{C},  \\
&\bm{F}\in L^2(0, T; L^2_{(\mathcal{S})}), \quad \bu_0\in W^{1,2}_{(\mathcal{S})}.
\end{split}
\end{equation*}
(1) There exists a unique solution $\bu(t, \bx)$ to $\eqref{linear-momentum-system}$ such that 
\begin{equation*}
\begin{split}
\bu\in C(0, T; W^{1,2}_{(\mathcal{S})}), \quad \bu|_{\mathcal{C}}\in L^2(0,T; H^{2}_{(\mathcal{C})}), \quad \bu_t\in L^2(0, T; L^{2}_{(\mathcal{S})}).
\end{split}
\end{equation*}
(2) Moreover, we suppose in addition  $\partial \mathcal{C}$ be of class $C^3$ and that
\begin{equation*}
\begin{split}
\nabla_{\bx}\tilde\rho \in L^4(0,T;L^6_{(\mathcal{C})}), \hspace{0.2cm}\tilde\rho_t\in L^2(0,T; L^3_{(\mathcal{C})}). 
\end{split}
\end{equation*}
Assume that
\begin{equation*}
	\begin{split}
		\int_{\mathcal{S}}\bm{F}\cdot \bm{\phi}\ddx=\int_{\mathcal{S}}\bm{F}_1\cdot \bm{\phi}\ddx+\int_{\mathcal{S}}\bm{F}_2:\nabla \bm{\phi}\ddx,
\end{split}
\end{equation*}
where
\begin{equation*}
	\begin{split}
		\bm{F}_1\in L^2(0,T;W^{1,2}_{(\mathcal{S})})\cap L^\infty(0,T;L^2_{(\mathcal{S})}), \quad \bm{F}_2\in L^2(0,T;W^{2,2}_{(\mathcal{S})})\cap L^\infty(0,T;W^{1,2}_{(\mathcal{S}})).
\end{split}
\end{equation*}
Furthermore, we assume that $\bm{F}_t$ can be written as
\begin{equation*}
	\begin{split}
		\int_{\mathcal{S}}\bm{F}_t\cdot \bm{\phi}\ddx=\int_{\mathcal{S}}(\bm{F}_3)_t\cdot \bm{\phi}\ddx+\int_{\mathcal{S}}(\bm{F}_4)_t:\nabla \bm{\phi}\ddx,
\end{split}
\end{equation*}
where
\begin{equation*}
	\begin{split}
		(\bm{F}_3)_t\in L^2(0,T;L^{\frac{6}{5}}_{(\mathcal{S})}), \quad (\bm{F}_4)_t\in L^2(0,T;L^2_{(\mathcal{S})}).
\end{split}
\end{equation*}
Then it holds that 
\begin{equation*}
\begin{split}
& \bu|_{\mathcal{C}}\in L^2(0,T; H^{3}_{(\mathcal{C})})\cap C(0,T; H^{2}_{(\mathcal{C})}),\hspace{0.2cm}\bu_t\in L^2(0,T; W^{1,2}_{(\mathcal{S})})\cap C(0,T; L^2_{(\mathcal{S})}).
\end{split}
\end{equation*}
\end{lemma}
With the above properties in hand, we use Schauder fixed-point argument to prove the Theorem $\ref{local-existence}$. Its proof is similar to section 4.3 in \cite{GMN-2019-ARMA}, and therefore we omit it.

\section{Global Existence of Strong Solutions to the Coupled System}

The following section is divided into several parts where we derive estimates needed for the global existence of strong solutions.

Assume that 
\begin{equation}\label{bound-density-1}
\frac{\bar{\rho}}{4}\leq \rho\leq 3\bar{\rho}.
\end{equation}

We multiply $\eqref{re-system-mass}_1$ by $\frac{p'(\bar{\rho})}{\bar{\rho}}\varrho$ and integrate the resulting equation over $\mathcal{C}$. Consequently,
\begin{equation}\label{l-2-rho}
\begin{split}
\frac{1}{2}\frac{p'(\bar{\rho})}{\bar{\rho}}\frac{\mathrm{d}}{\mathrm{d}t}\int_{\mathcal{C}}\varrho^2\ddx +p'(\bar{\rho})\int_{\mathcal{C}}\varrho \mathrm{div}_{\bx}\bv\ddx=-\frac{1}{2}\frac{p'(\bar{\rho})}{\bar{\rho}}\int_{\mathcal{C}}\varrho^2\mathrm{div}_{\bx}\bv\ddx.
\end{split}
\end{equation}

We choose $\boldsymbol{\phi}=\bu$ in \eqref{re-system-momentum} to derive that 
\begin{multline}\label{l-2-u}
\int_{\mathcal{S}}\rho_{\mathcal{S}}\bu_t\cdot \bu\ddx +\int_{\mathcal{C}}S(\nabla_{\bx}\bv):\nabla_{\bx}\bv\ddx-\int_{\mathcal{C}}p'(\bar{\rho})\rho\mathrm{div}_{\bx}\bv\ddx\\
=-\int_{\mathcal{S}}\rho_{\mathcal{S}}\bv \cdot\nabla_{\bx}\bu\cdot \bu\ddx+\int_{\mathcal{C}}(p'(\bar{\rho})-p'({\rho}))\nabla_{\bx}\varrho\cdot\bu\ddx-\int_{\mathcal{C}}\left(\bb\otimes\bb-\frac{1}{2}|\bb|^2\mathbb{I}\right):\nabla_{\bx}\bu\ddx.
\end{multline}

We multiply $\eqref{re-system-induction}_1$ by $\bb$, integrate the resulting equation over $\mathcal{C}$ then use integration by parts and definition of relative velocity $\bv$ to deduce that 
\begin{equation}\label{l-2-b}
\begin{split}
\frac{1}{2}\frac{\mathrm{d}}{{\mathrm{d}t}}\int_{\mathcal{C}}|\bb|^2\ddx+\int_{\mathcal{C}}|\nabla_{\bx}\times \bb|^2\ddx=-\int_{\mathcal{C}}(\bv\cdot\nabla_{\bx}\bb+\bb\mathrm{div}_{\bx}\bu-\bb\cdot\nabla_{\bx}\bu)\cdot\bb\ddx.
\end{split}
\end{equation}

Further, \eqref{l-2-rho}, \eqref{l-2-u} and \eqref{l-2-b}, integration by parts, and continuity equation \eqref{mhd-x}$_1$ yield
\begin{multline}\label{zero-order-1}
\frac{1}{2}\frac{p'(\bar{\rho})}{\bar{\rho}}\frac{\mathrm{d}}{\mathrm{d}t}\int_{\mathcal{C}}\varrho^2\ddx+
\frac{1}{2}\frac{\mathrm{d}}{\mathrm{d}t}\int_{\mathcal{S}}\rho_{\mathcal{S}}|\bu|^2\ddx+
\frac{1}{2}\frac{\mathrm{d}}{{\mathrm{d}t}}\int_{\mathcal{C}}|\bb|^2\ddx+\int_{\mathcal{C}}S(\nabla_{\bx}\bv):\nabla_{\bx}\bv\ddx\\
+\int_{\mathcal{C}}|\nabla_{\bx}\times\bb|^2\ddx
=-\frac{1}{2}\frac{p'(\bar{\rho})}{\bar{\rho}}\int_{\mathcal{C}}\mathrm{div}_{\bx}\bv\varrho^2\ddx+\int_{\mathcal{C}}(p'(\bar{\rho})-p'({\rho}))\nabla_{\bx}\varrho\cdot\bu\ddx\\
 -\int_{\mathcal{C}}\left(\bb\otimes\bb-\frac{1}{2}|\bb|^2\mathbb{I}\right):\nabla_{\bx}\bu\ddx 
 -\int_{\mathcal{C}}(\bv\cdot\nabla_{\bx}\bb+\bb\mathrm{div}_{\bx}\bu-\bb\cdot\nabla_{\bx}\bu)\cdot\bb\ddx
=:\sum_{i=1}^4\mathcal{I}_{1,i}.
\end{multline}

By Sobolev's inequality, it holds 
\begin{equation*}
\mathcal{I}_{1,1}\leq  C\norm{\mathrm{div}_{\bx}\bv}_{L^\infty_{(\mathcal{C})}}\norm{\varrho}^2_{L^2_{(\mathcal{C})}} \leq \epsilon\norm{\bv}_{H^3_{(\mathcal{C})}}^2+C\norm{\varrho}^4_{L^2_{(\mathcal{C})}}.
\end{equation*}

It follows from $\eqref{re-system}$, H\"older's inequality and $\eqref{bound-density-1}$ that 
\begin{equation}\label{xi-u}
\begin{split}
|\boldsymbol{\xi}_{\bu}|\leq C\norm{\bu}_{L^2_{(\mathcal{C})}}.
\end{split}
\end{equation}

Based on  $\bu=\bv+\bo_{\bu}\times \bx +\boldsymbol{\xi}_{\bu}$ and the boundedness of $\mathcal{C}$, we further have 
\begin{multline}\label{bo}
|\bo_{\bu}|\sim \norm{\bo_{\bu}\times \bx}_{L^2_{(\mathcal{C})}}\leq C(\norm{\bu}_{L^2_{(\mathcal{C})}}+\norm{\bv}_{L^2_{(\mathcal{C})}}+\norm{\boldsymbol{\xi}_{\bu}}_{L^2_{(\mathcal{C})}})\\
\leq C(\norm{\bu}_{L^2_{(\mathcal{C})}}+\norm{\bv}_{L^2_{(\mathcal{C})}}+|\boldsymbol{\xi}_{\bu}|)\leq C(\norm{\bu}_{L^2_{(\mathcal{C})}}+\norm{\bv}_{L^2_{(\mathcal{C})}}).
\end{multline}

 H\"older's inequality  and $\eqref{bound-density-1}$ give
\begin{equation*}
\begin{split}
\mathcal{I}_{1,2}&\leq  \epsilon \norm{{\bu}}^2_{L^2_{(\mathcal{C})}}
+C\norm{{\nabla_{\bx}\varrho}}^2_{L^2_{(\mathcal{C})}}\norm{{\varrho}}^2_{L^\infty_{(\mathcal{C})}}.
\end{split}
\end{equation*}

By direct calculation and $\eqref{symmetry}$, then integration by parts gives that 
\begin{multline*}
\mathcal{I}_{1,3}+\mathcal{I}_{1,4}\\
=-\int_{\mathcal{C}}\left(\bb\otimes\bb-\frac{1}{2}|\bb|^2\mathbb{I}\right):\nabla_{\bx}\bu\ddx-\int_{\mathcal{C}}\bv\cdot\nabla_{\bx}\bb\cdot\bb\ddx+\int_{\mathcal{C}}(\bb\otimes\bb-|\bb|^2\mathbb{I}):\nabla_{\bx}\bu\ddx\\
=-\int_{\mathcal{C}}\bv\cdot\nabla_{\bx}\bb\cdot\bb\ddx-\frac{1}{2}\int_{\mathcal{C}}(|\bb|^2\mathbb{I}):\nabla_{\bx}\bv\ddx=0.
\end{multline*}

We plug the above estimates into $\eqref{zero-order-1}$. Since $\|\nabla_\bx\times \bb\|_{L^2_{(\mathcal{C})}}\sim \|\nabla_\bx \bb\|_{L^2_{(\mathcal{C})}}$ (\cite{KoYa}), we deduce that 
\begin{multline}\label{zero-order}
\frac{1}{2}\frac{p'(\bar{\rho})}{\bar{\rho}}\frac{\mathrm{d}}{\mathrm{d}t}\int_{\mathcal{C}}\varrho^2\ddx+
\frac{1}{2}\frac{\mathrm{d}}{\mathrm{d}t}\int_{\mathcal{S}}\rho_{\mathcal{S}}|\bu|^2\ddx+
\frac{1}{2}\frac{\mathrm{d}}{{\mathrm{d}t}}\int_{\mathcal{C}}|\bb|^2\ddx\\
+\int_{\mathcal{C}}S(\nabla_{\bx}\bv):\nabla_{\bx}\bv\ddx
+\int_{\mathcal{C}}|\nabla_{\bx} \bb|^2\ddx
\\
\leq  \epsilon (\norm{{{\bv}}}^2_{H^3_{(\mathcal{C})}}+\norm{{{\bu}}}^2_{L^2_{(\mathcal{C})}})
+C(\norm{\varrho}^4_{L^2_{(\mathcal{C})}}+\norm{{\nabla_{\bx}\varrho}}^2_{L^2_{(\mathcal{C})}}\norm{{\varrho}}^2_{L^\infty_{(\mathcal{C})}}).
\end{multline}

We apply $\partial_t$ on $\eqref{re-system-mass}_1$, multiply by $\frac{p'(\bar{\rho})}{\bar{\rho}}\varrho_t$, then integrate the resulting equation over $\mathcal{C}$ to get
\begin{multline}\label{varrho-l-2-t}
\frac{1}{2}\frac{p'(\bar{\rho})}{\bar{\rho}}\frac{\mathrm{d}}{\mathrm{d}t}\int_{\mathcal{C}}\varrho_t^2\ddx +p'(\bar{\rho})\int_{\mathcal{C}}\varrho_t \mathrm{div}_{\bx}\bv_t\ddx\\
=-\frac{p'(\bar{\rho})}{\bar{\rho}}\int_{\mathcal{C}}\varrho_t\bv_t\cdot\nabla_{\bx}\varrho\ddx-\frac{1}{2}\frac{p'(\bar{\rho})}{\bar{\rho}}\int_{\mathcal{C}}\varrho_t^2\mathrm{div}_{\bx}\bv\ddx-\frac{p'(\bar{\rho})}{\bar{\rho}}\int_{\mathcal{C}}\varrho_t\varrho\mathrm{div}_{\bx}\bv_t\ddx
:=\sum_{i=1}^{3}\mathcal{I}_{2,i}.
\end{multline}

Using H\"older's inequality and Sobolev's inequality, it gives  
\begin{equation*}
\mathcal{I}_{2,1}\leq \norm{\bv_t}_{L^6_{(\mathcal{C})}}\norm{\varrho_t}_{L^2_{(\mathcal{C})}}\norm{\nabla_{\bx}\varrho}_{L^3_{(\mathcal{C})}}\leq \epsilon\norm{\nabla_{\bx}\bv_t}^2_{L^2_{(\mathcal{C})}}+C(\norm{\varrho_t}^4_{L^2_{(\mathcal{C})}}+\norm{\nabla_{\bx}\varrho}^4_{H^1_{(\mathcal{C})}}).
\end{equation*}
Similarly
\begin{equation*}
\mathcal{I}_{2,2}\leq \norm{\varrho_t}^2_{L^2_{(\mathcal{C})}}\norm{\mathrm{div}_{\bx}\bv}_{L^\infty_{(\mathcal{C})}}\leq \norm{\varrho_t}^2_{L^2_{(\mathcal{C})}}\norm{\mathrm{div}_{\bx}\bv}_{H^2_{(\mathcal{C})}}
\leq \epsilon \norm{\bv}^2_{H^3_{(\mathcal{C})}}+C\norm{\varrho_t}^4_{L^2_{(\mathcal{C})}}, 
\end{equation*}
and 
\begin{equation*}
\mathcal{I}_{2,3}\leq \norm{\varrho_t}_{L^2_{(\mathcal{C})}}\norm{\varrho}_{L^\infty_{(\mathcal{C})}}\norm{\mathrm{div}_{\bx}\bv_t}_{L^2_{(\mathcal{C})}}\leq \epsilon\norm{\varrho_t}^2_{L^2_{(\mathcal{C})}}+C\norm{\varrho}^2_{H^2_{(\mathcal{C})}}\norm{\bv_t}^2_{H^1_{(\mathcal{C})}}.
\end{equation*}

We put the above estimates into $\eqref{varrho-l-2-t}$ to get
\begin{multline}\label{varrho-l-2-t-r}
\frac{1}{2}\frac{p'(\bar{\rho})}{\bar{\rho}}\frac{\mathrm{d}}{\mathrm{d}t}\int_{\mathcal{C}}\varrho_t^2\ddx+p'(\bar{\rho})\int_{\mathcal{C}}\varrho_t \mathrm{div}_{\bx}\bv_t\ddx\\
\leq \epsilon(\norm{\nabla_{\bx}\bv_t}^2_{L^2_{(\mathcal{C})}}+\norm{\bv}^2_{H^3_{(\mathcal{C})}}+\norm{\varrho_t}^2_{L^2_{(\mathcal{C})}}) +C(\norm{\nabla_{\bx}\varrho}^4_{H^1_{(\mathcal{C})}}+\norm{\varrho_t}^4_{L^2_{(\mathcal{C})}}+\norm{\varrho}^2_{H^2_{(\mathcal{C})}}\norm{\bv_t}^2_{H^1_{(\mathcal{C})}}).
\end{multline}

We differentiate  \eqref{re-system-momentum} with respect to $t$  and then multiply it by  $\boldsymbol{\phi}=\bu_t$. The resulting equations is 
\begin{multline}\label{l-2-u-t}
\frac{1}{2}\frac{\mathrm{d}}{\mathrm{d}t}\int_{\mathcal{S}}\rho_{\mathcal{S}}|\bu_t|^2\ddx+\int_{\mathcal{C}}S(\nabla_{\bx}\bv_t):\nabla_{\bx}\bv_t\ddx-p'(\bar{\rho})\int_{\mathcal{C}}\rho_t\mathrm{div}_{\bx}\bv_t\ddx\\
=-\int_{\mathcal{C}}\rho_t\bo_{\bu}\times\bu\cdot\bu_t\ddx-\int_{\mathcal{S}}\rho_{\mathcal{S}}(\bo_{\bu})_t\times\bu\cdot\bu_t\ddx-\int_{\mathcal{C}}\rho_t\bv \cdot\nabla_{\bx}\bu\cdot \bu_t\ddx\\
 -\int_{\mathcal{C}}\rho\bv_t \cdot\nabla_{\bx}\bu\cdot \bu_t\ddx-2\int_{\mathcal{C}}\rho\bv \cdot\nabla_{\bx}\bu_t\cdot \bu_t\ddx-\int_{\mathcal{C}}p'({\rho})_t\nabla_{\bx}\varrho\cdot\bu_t\ddx\\
+\int_{\mathcal{C}}(p'(\bar{\rho})-p'({\rho}))\nabla_{\bx}\varrho_t\cdot\bu_t\ddx-\int_{\mathcal{C}}\left(\bb\otimes\bb-\frac{1}{2}|\bb|^2\mathbb{I}\right)_t:\nabla_{\bx}\bu_t\ddx
=:\sum_{i=1}^{8}\mathcal{I}_{3,i}.
\end{multline}

Because  $|\bo_{\bu}|\sim\norm{\bo_{\bu}\times \bx}_{L^\infty_{(\mathcal{C})}}\sim \norm{\bo_{\bu}\times \bx}_{L^2_{(\mathcal{C})}}$ and $\bu=\bv+\bo_{\bu}\times\bx+\bxi_{\bu}$, we use \eqref{xi-u} and \eqref{bo} to get 
\begin{multline}\label{u-l-infty-2}
\norm{{\bu}}_{L^\infty_{(\mathcal{C})}}\leq C( \norm{{\bv}}_{L^\infty_{(\mathcal{C})}}+\norm{\bo_{\bu}\times \bx}_{L^\infty_{(\mathcal{C})}}+\norm{{\boldsymbol{\xi}_{\bu}}}_{L^\infty_{(\mathcal{C})}})\\
\leq C(\norm{{\bv}}_{L^\infty_{(\mathcal{C})}}+|\bo_{\bu}|+|{{\boldsymbol{\xi}_{\bu}}}|)\leq C(\norm{{\bv}}_{L^\infty_{(\mathcal{C})}}+\norm{{\bu}}_{L^2_{(\mathcal{C})}}+\norm{{\bv}}_{L^2_{(\mathcal{C})}}).
\end{multline}

On account of  $\eqref{bo}$ and H\"older's inequality, it holds that 
\begin{multline*}
\mathcal{I}_{3,1}\leq \epsilon \norm{\bu_t}^2_{L^2_{(\mathcal{C})}}+C\norm{\varrho_t}^2_{L^2_{(\mathcal{C})}}\norm{\bu}^2_{L^\infty_{(\mathcal{C})}}\norm{\bo_{\bu}}^2_{L^\infty_{(\mathcal{C})}}\leq  \epsilon \norm{\bu_t}^2_{L^2_{(\mathcal{C})}}+C\norm{\varrho_t}^2_{L^2_{(\mathcal{C})}}\norm{\bu}^2_{L^\infty_{(\mathcal{C})}}(\norm{{\bu}}^2_{L^2_{(\mathcal{C})}}+\norm{{\bv}}^2_{L^2_{(\mathcal{C})}})\\
\leq  \epsilon \norm{\bu_t}^2_{L^2_{(\mathcal{C})}}+C\norm{\varrho_t}^2_{L^2_{(\mathcal{C})}}(\norm{\bv}^2_{L^\infty_{(\mathcal{C})}}+\norm{\bu}^2_{L^2_{(\mathcal{C})}}+\norm{\bv}^2_{L^2_{(\mathcal{C})}})(\norm{{\bu}}^2_{L^2_{(\mathcal{C})}}+\norm{{\bv}}^2_{L^2_{(\mathcal{C})}}).
\end{multline*}

 It holds that $\bu|_{\mathcal{B}}=(\bo_{\bu}\times\bx)|_{\mathcal{B}}+\boldsymbol{\xi}_{\bu}|_{\mathcal{B}}$, and thus
\begin{equation}\label{u-B-C}
\norm{\bu}^2_{L^2_{(\mathcal{B})}}\leq \norm{\bo_{\bu}\times\bx}^2_{L^2_{(\mathcal{B})}}+|\mathrm{vol}(\mathcal{B})|^2|\boldsymbol{\xi_{\bu}}|^2
\leq C(|\bo_{\bu}|^2+|\boldsymbol{\xi_{\bu}}|^2)\leq C(\norm{\bu}^2_{L^2_{(\mathcal{C})}}+\norm{\bv}^2_{L^2_{(\mathcal{C})}}).
\end{equation}

Next, $\eqref{re-system-mass}$ and $\eqref{re-system}$ yield
\begin{multline*}
\frac{\mathrm{d}}{\mathrm{d}t}\boldsymbol{\xi_{\bu}}=-\frac{1}{m_{\mathcal{B}}}\int_{\mathcal{C}}(\rho_t\bu-\rho\bu_t)\ddx=-\frac{1}{m_{\mathcal{B}}}\int_{\mathcal{C}}(\rho\bv\cdot\nabla_{\bx}\bu-\rho\bu_t)\ddx\\
=-\frac{1}{m_{\mathcal{B}}}\int_{\mathcal{C}}(\rho\bv\cdot\nabla_{\bx}(\bv+\bo_{\bu}\times\bx+\boldsymbol{\xi}_{\bu})-\rho\bu_t)\ddx.
\end{multline*}

Due to H\"older's inequality, \eqref{xi-u} and \eqref{bo}, it also holds that 
\begin{equation}\label{xi-u-t-1}
\begin{split}
|(\boldsymbol{\xi_{\bu}})_t|\leq C\norm{\bv}_{L^2_{(\mathcal{C})}}(\norm{\nabla_{\bx}\bv}_{L^2_{(\mathcal{C})}}+\norm{\bu}_{L^2_{(\mathcal{C})}}+\norm{\bv}_{L^2_{(\mathcal{C})}})+C\norm{\bu_t}_{L^2_{(\mathcal{C})}}.
\end{split}
\end{equation}

We infer from
\begin{equation*}
\begin{split}
\frac{\mathrm{d}}{\mathrm{d}t}\bo_{\bu}\times\bx=\bu_t-\bv_t-(\boldsymbol{\xi_{\bu}})_t,
\end{split}
\end{equation*}
 and $\eqref{xi-u-t-1}$ that
\begin{multline}\label{bo-t}
|(\bo_{\bu})_t|\sim\norm{\frac{\mathrm{d}}{\mathrm{d}t}\bo_{\bu}\times\bx}_{L^2_{(\mathcal{C})}}\\ 
\leq \norm{\bv_t}_{L^2_{(\mathcal{C})}}+C\norm{\bv}_{L^2_{(\mathcal{C})}}(\norm{\nabla_{\bx}\bv}_{L^2_{(\mathcal{C})}}+\norm{\bu}_{L^2_{(\mathcal{C})}}+\norm{\bv}_{L^2_{(\mathcal{C})}})+C\norm{\bu_t}_{L^2_{(\mathcal{C})}}.
\end{multline}
Moreover, we apply  $\eqref{bo-t}$ and $\eqref{xi-u-t-1}$ to get 
\begin{multline}\label{u-b}
\norm{\bu_t}^2_{L^2_{(\mathcal{B})}}\leq  \norm{(\bo_{\bu})_t\times\bx}^2_{L^2_{(\mathcal{B})}}+|\mathrm{vol}(\mathcal{B})|^2|(\boldsymbol{\xi}_{\bu})_t|^2\leq C(|(\bo_{\bu})_t|^2+|(\bxi_{\bu})_t|^2)\\
\leq C\norm{\bv_t}^2_{L^2_{(\mathcal{C})}}+C\norm{\bu_t}^2_{L^2_{(\mathcal{C})}}+C\norm{\bv}^2_{L^2_{(\mathcal{C})}}(\norm{\nabla_{\bx}\bv}^2_{L^2_{(\mathcal{C})}}+\norm{\bu}^2_{L^2_{(\mathcal{C})}}+\norm{\bv}^2_{L^2_{(\mathcal{C})}}).
\end{multline}

Using $\eqref{bound-density-1}$, $\eqref{u-B-C}$, $\eqref{bo-t}$, and $\eqref{u-b}$, we deduce that 
\begin{multline*}
\mathcal{I}_{3,2}\leq \epsilon\|\bu_t\|^2_{L^2_{(\mathcal{S})}}+C\norm{\bu}^2_{L^2_{(\mathcal{S})}}|(\bo_{\bu})_t|^2
\leq \epsilon(\|\bu_t\|^2_{L^2_{(\mathcal{C})}}+\|\bu_t\|^2_{L^2_{(\mathcal{B})}})+C(\norm{\bu}^2_{L^2_{(\mathcal{C})}}+\norm{\bu}^2_{L^2_{(\mathcal{B})}})|(\bo_{\bu})_t|^2\\
\leq \epsilon(\|\bu_t\|^2_{L^2_{(\mathcal{C})}}+\|\bv_t\|^2_{L^2_{(\mathcal{C})}})+C\norm{\bv}^2_{L^2_{(\mathcal{C})}}(\norm{\nabla_{\bx}\bv}^2_{L^2_{(\mathcal{C})}}+\norm{\bu}^2_{L^2_{(\mathcal{C})}}+\norm{\bv}^2_{L^2_{(\mathcal{C})}})\\
+C(\norm{\bu}^2_{L^2_{(\mathcal{C})}}+\norm{\bv}^2_{L^2_{(\mathcal{C})}})(\norm{\bv_t}^2_{L^2_{(\mathcal{C})}}+\norm{\bu_t}^2_{L^2_{(\mathcal{C})}}+\norm{\bv}^2_{L^2_{(\mathcal{C})}}(\norm{\nabla_{\bx}\bv}^2_{L^2_{(\mathcal{C})}}+\norm{\bu}^2_{L^2_{(\mathcal{C})}}+\norm{\bv}^2_{L^2_{(\mathcal{C})}})). 
\end{multline*}

It follows from $\eqref{bo}$ and H\"older's inequality that 
\begin{multline*}
\mathcal{I}_{3,3}\leq \epsilon \norm{\bu_t}^2_{L^2_{(\mathcal{C})}}+C\norm{\varrho_t}^2_{L^2_{(\mathcal{C})}}\norm{\bv}^2_{L^\infty_{(\mathcal{C})}}(\|\nabla_{\bx}\bv\|^2_{L^\infty_{(\mathcal{C})}}+|\bo_{\bu}|^2)\\
\leq \epsilon \norm{\bu_t}^2_{L^2_{(\mathcal{C})}}+C\norm{\varrho_t}^2_{L^2_{(\mathcal{C})}}\norm{\bv}^2_{L^\infty_{(\mathcal{C})}}(\|\nabla_{\bx}\bv\|_{L^\infty_{(\mathcal{C})}}^2+\norm{\bu}^2_{L^2_{(\mathcal{C})}}+\norm{\bv}^2_{L^2_{(\mathcal{C})}}),
\end{multline*}
and 
\begin{multline*}
\mathcal{I}_{3,4}\leq \epsilon \norm{\bu_t}^2_{L^2_{(\mathcal{C})}}+C\norm{\bv_t}^2_{L^2_{(\mathcal{C})}}\norm{\nabla_{\bx}\bu}^2_{L^\infty_{(\mathcal{C})}}\leq \epsilon \norm{\bu_t}^2_{L^2_{(\mathcal{C})}}+C\norm{\bv_t}^2_{L^2_{(\mathcal{C})}}(\norm{\nabla_{\bx}\bv}^2_{L^\infty_{(\mathcal{C})}}+|\bo_{\bu}|^2)\\
\leq \epsilon \norm{\bu_t}^2_{L^2_{(\mathcal{C})}}+C\norm{\bv_t}^2_{L^2_{(\mathcal{C})}}(\norm{\bv}^2_{H^3_{(\mathcal{C})}}+\norm{\bu}^2_{L^2_{(\mathcal{C})}}). 
\end{multline*}

By $\eqref{bo-t}$ and H\"older's inequality, we get 
\begin{multline*}
\mathcal{I}_{3,5}\leq C\int_{\mathcal{C}}\rho|\bv|(|\nabla_{\bx}\bv_t|+|(\bo_{\bu})_t|)|\bu_t|\ddx\\
\leq  \epsilon (\norm{\nabla_{\bx}\bv_t}^2_{L^2_{(\mathcal{C})}}+\norm{\bu_t}^2_{L^2_{(\mathcal{C})}})
+C(\norm{\bu_t}^2_{L^2_{(\mathcal{C})}}\norm{\bv}_{L^\infty_{(\mathcal{C})}}+\norm{\bv}^2_{L^2_{(\mathcal{C})}}|(\bo_{\bu})_t|^2)\\
\leq  \epsilon (\norm{\nabla_{\bx}\bv_t}^2_{L^2_{(\mathcal{C})}}+\norm{\bu_t}^2_{L^2_{(\mathcal{C})}})
+C\norm{\bu_t}^2_{L^2_{(\mathcal{C})}}\norm{\bv}^2_{H^2_{(\mathcal{C})}}\\
\quad +C\norm{\bv}^2_{L^2_{(\mathcal{C})}}(\norm{\bv_t}^2_{L^2_{(\mathcal{C})}}+\norm{\bu_t}^2_{L^2_{(\mathcal{C})}}+\norm{\bv}^2_{L^2_{(\mathcal{C})}}(\norm{\bu}^2_{L^2_{(\mathcal{C})}}+\norm{\bv}^2_{H^1_{(\mathcal{C})}})).
\end{multline*}
Furthermore, it follows from H\"older's  and Sobolev's inequality that 
\begin{multline*}
\mathcal{I}_{3,6}+\mathcal{I}_{3,7}\leq \epsilon \norm{\bu_t}^2_{L^2_{(\mathcal{C})}}+C\norm{\varrho_t}^2_{L^4_{(\mathcal{C})}}
\norm{\nabla_{\bx}\varrho}^2_{L^4_{(\mathcal{C})}}+\norm{\varrho}^2_{L^\infty_{(\mathcal{C})}}\norm{\nabla_{\bx}\varrho_t}^2_{L^2_{(\mathcal{C})}}\\
\leq \epsilon \norm{\bu_t}^2_{L^2_{(\mathcal{C})}}+
C\norm{\varrho_t}^2_{H^1_{(\mathcal{C})}}
\norm{\varrho}^2_{H^2_{(\mathcal{C})}} 
\end{multline*}
and 
\begin{equation*}
\begin{split}
\mathcal{I}_{3,8}&\leq \epsilon\norm{\nabla_{\bx}\bv_t}^2_{L^2_{(\mathcal{C})}}+C\norm{\bb}^2_{L^\infty_{(\mathcal{C})}}\norm{\bb_t}^2_{L^2_{(\mathcal{C})}}.
\end{split}
\end{equation*}

 We finally insert above estimates into $\eqref{l-2-u-t}$ to derive that 
\begin{equation}\label{u-t-l-2-r}
\begin{split}  
&\frac{\mathrm{d}}{\mathrm{d}t}\int_{\mathcal{S}}\rho_{\mathcal{S}}|\bu_t|^2+\int_{\mathcal{C}}S(\nabla_{\bx}\bv_t):\nabla_{\bx}\bv_t-p'(\bar{\rho})\int_{\mathcal{C}}\varrho_t\mathrm{div}_{\bx}\bv_t\\
&\leq \epsilon(\|\bu_t\|^2_{L^2_{(\mathcal{C})}}+\|\bv_t\|^2_{H^1_{(\mathcal{C})}})+C\norm{\varrho_t}^2_{L^2_{(\mathcal{C})}}(\norm{\bv}^2_{L^\infty_{(\mathcal{C})}}+\norm{\bu}^2_{L^2_{(\mathcal{C})}}+\norm{\bv}^2_{L^2_{(\mathcal{C})}})(\norm{{\bu}}^2_{L^2_{(\mathcal{C})}}+\norm{{\bv}}^2_{L^2_{(\mathcal{C})}})\\
&\quad+C\norm{\bv}^2_{L^2_{(\mathcal{C})}}(\norm{\nabla_{\bx}\bv}^2_{L^2_{(\mathcal{C})}}+\norm{\bu}^2_{L^2_{(\mathcal{C})}}+\norm{\bv}^2_{L^2_{(\mathcal{C})}})\\
& \quad +C(\norm{\bu}^2_{L^2_{(\mathcal{C})}}+\norm{\bv}^2_{L^2_{(\mathcal{C})}})(\norm{\bv_t}^2_{L^2_{(\mathcal{C})}}+\norm{\bu_t}^2_{L^2_{(\mathcal{C})}}+\norm{\bv}^2_{L^2_{(\mathcal{C})}}(\norm{\nabla_{\bx}\bv}^2_{L^2_{(\mathcal{C})}}+\norm{\bu}^2_{L^2_{(\mathcal{C})}}+\norm{\bv}^2_{L^2_{(\mathcal{C})}}))
\\
&
\quad+C\norm{\varrho_t}^2_{L^2_{(\mathcal{C})}}\norm{\bv}^2_{L^\infty_{(\mathcal{C})}}(\|\nabla_{\bx}\bv\|_{L^\infty_{(\mathcal{C})}}^2+\norm{\bu}^2_{L^2_{(\mathcal{C})}}+\norm{\bv}^2_{L^2_{(\mathcal{C})}})
\\
&\quad +C\norm{\bv_t}^2_{L^2_{(\mathcal{C})}}\norm{\bv}^2_{H^3_{(\mathcal{C})}}+C\norm{\bv}^2_{H^2_{(\mathcal{C})}}
\norm{\bu_t}^2_{L^2_{(\mathcal{C})}} +C\norm{\varrho_t}^2_{H^1_{(\mathcal{C})}}
\norm{\varrho}^2_{H^2_{(\mathcal{C})}}+C\norm{\bb}^2_{L^\infty_{(\mathcal{C})}}\norm{\bb_t}^2_{L^2_{(\mathcal{C})}}.
\end{split}
\end{equation}

Next, we apply $\partial_t$ on $\eqref{re-system-induction}_1$ and multiply by $\bb_t$. Straightforward calculations lead to 
\begin{multline}\label{b-t-l-2}
\frac{1}{2}\frac{\mathrm{d}}{\mathrm{d}t}\int_{\mathcal{C}}|\bb_t|^2\ddx+C\int_{\mathcal{C}}|\nabla_{\bx}\bb_t|^2\ddx\\
=-\int_{\mathcal{C}}(\bo_{\bu})_t\times\bb\cdot \bb_t\ddx+\int_{\mathcal{C}}((\bo_{\bu})_t\times\bx+(\boldsymbol{\xi}_{\bu})_t)\cdot \nabla_{\bx}\bb\cdot\bb_t\ddx\\
-\int_{\mathcal C}((\bo_{\bu})\times\bx+(\boldsymbol{\xi}_{\bu}))\cdot \nabla_{\bx}\bb_t\cdot\bb_t\ddx\\
-\int_{\mathcal{C}}(\bu\cdot\nabla_{\bx}\bb+\bb\mathrm{div}\bu-\bb\cdot\nabla_{\bx}\bu)_t\cdot\bb_t\ddx
=:\sum_{i=1}^4\mathcal{I}_{4,i}.
\end{multline}

We use $\eqref{xi-u-t-1}$ and $\eqref{bo-t}$ to derive 
\begin{multline*}
\mathcal{I}_{4,1}\leq \epsilon\norm{\bb_t}^2_{L^2_{(\mathcal{C})}}
+C|(\bo_{\bu})_t|^2\norm{\bb}^2_{L^2_{(\mathcal{C})}}\\
\leq \epsilon\norm{\bb_t}^2_{L^2_{(\mathcal{C})}}+C(\norm{\bu_t}^2_{L^2_{(\mathcal{C})}}+\norm{\bv_t}^2_{L^2_{(\mathcal{C})}}+\norm{\bv}^2_{L^2_{(\mathcal{C})}}(\norm{\bu}^2_{L^2_{(\mathcal{C})}}+\norm{\bv}^2_{H^1_{(\mathcal{C})}}))\norm{\bb}^2_{L^2_{(\mathcal{C})}}, 
\end{multline*}
and 
\begin{equation*}
\mathcal I_{4,2}\leq \varepsilon \|\nabla_\bx \bb_t\|^2_{L^2_{(\mathcal C)}} + C \left(\|\bu\|^2_{L^2_{(\mathcal C)}} + \|\bv\|_{H^3_{(\mathcal C)}}^2\right)\|\bb_t\|^2_{L^2_{(\mathcal C)}}.
\end{equation*}
Further,
\begin{multline*}
\mathcal{I}_{4,3}\leq \epsilon\norm{\bb_t}^2_{L^2_{(\mathcal{C})}}+C\norm{\nabla_{\bx}\bb}^2_{L^2_{(\mathcal{C})}}(|(\bo_{\bu})_t|^2+|(\boldsymbol{\xi}_{\bu})_t|^2)\\
\leq \epsilon\norm{\bb_t}^2_{L^2_{(\mathcal{C})}}+C\norm{\nabla_{\bx}\bb}^2_{L^2_{(\mathcal{C})}}(\norm{\bu_t}^2_{L^2_{(\mathcal{C})}}+\norm{\bv_t}^2_{L^2_{(\mathcal{C})}}+\norm{\bv}^2_{L^2_{(\mathcal{C})}}(\norm{\bu}^2_{L^2_{(\mathcal{C})}}+\norm{\bv}^2_{H^1_{(\mathcal{C})}})).
\end{multline*}

Using $\eqref{u-l-infty-2}$, $\eqref{bo}$ and $\eqref{bo-t}$, we have 
\begin{equation*}
\begin{split}  
\mathcal{I}_{4,4}&=-\int_{\mathcal{C}}(\bu_t\cdot\nabla_{\bx}\bb+\bu\cdot\nabla_{\bx}\bb_t+\bb_t\mathrm{div}_{\bx}\bu +\bb\mathrm{div}_{\bx}\bu_t-\bb_t\cdot\nabla_{\bx}\bu-\bb\cdot\nabla_{\bx}\bu_t)\cdot\bb_t\\
&\leq C(\norm{\bu_t}_{L^2_{(\mathcal{C})}}\norm{\nabla_{\bx}\bb}_{L^4_{(\mathcal{C})}}\norm{\bb_t}_{L^4_{(\mathcal{C})}}+\norm{\nabla_{\bx}\bb_t}_{L^2_{(\mathcal{C})}}\norm{\bu}_{L^\infty_{(\mathcal{C})}}\norm{\bb_t}_{L^2_{(\mathcal{C})}}\\
&\quad +\norm{\bb_t}_{L^2_{(\mathcal{C})}}(\norm{\mathrm{div}_{\bx}\bv}_{L^\infty_{(\mathcal{C})}}+|\bo_{\bu}|)\norm{\bb_t}_{L^2_{(\mathcal{C})}}+\norm{\bb}_{L^\infty_{(\mathcal{C})}}(\norm{\nabla_{\bx}\bv_t}_{L^2_{(\mathcal{C})}}+|(\bo_{\bu})_t|)\norm{\bb_t}_{L^2_{(\mathcal{C})}})\\
&\leq\epsilon(\norm{\bu_t}^2_{L^2_{(\mathcal{C})}}+\norm{\nabla_{\bx}\bb_t}^2_{L^2_{(\mathcal{C})}}+\norm{\bb_t}^2_{L^2_{(\mathcal{C})}}+\norm{\nabla_\bx\bv_t}^2_{L^2_{(\mathcal{C})}}+\norm{\bv_t}^2_{L^2_{(\mathcal{C})}})\\
&\quad +C\norm{\bb}_{H^2_{(\mathcal{C})}}^2\norm{\bb_t}_{H^1_{(\mathcal{C})}}^2+C(\norm{\bu}^2_{L^2_{(\mathcal{C})}}+\norm{\bv}^2_{H^3_{(\mathcal{C})}})\norm{\bb_t}^2_{L^2_{(\mathcal{C})}}\\
&\quad +C\norm{\bb}^2_{L^\infty_{(\mathcal{C})}}\norm{\bv}_{L^2_{(\mathcal{C})}}^2(\norm{\nabla_{\bx}\bv}^2_{L^2_{(\mathcal{C})}}+\norm{\bu}^2_{L^2_{(\mathcal{C})}}+\norm{\bv}^2_{L^2_{(\mathcal{C})}}).
\end{split}
\end{equation*}

We plug the above estimates into $\eqref{b-t-l-2}$ to deduce that 
\begin{multline}\label{b-t-l-2-r}
\frac{1}{2}\frac{\mathrm{d}}{\mathrm{d}t}\int_{\mathcal{C}}|\bb_t|^2\ddx+C\int_{\mathcal{C}}|\nabla_{\bx}\bb_t|^2\ddx\\
\leq \epsilon(\norm{\bu_t}^2_{L^2_{(\mathcal{C})}}+\norm{\nabla_{\bx}\bb_t}^2_{L^2_{(\mathcal{C})}}+\norm{\bb_t}^2_{L^2_{(\mathcal{C})}}+\norm{\nabla_\bx\bv_t}^2_{L^2_{(\mathcal{C})}}+\norm{\bv_t}^2_{L^2_{(\mathcal{C})}})\\
+C(\norm{\bu_t}^2_{L^2_{(\mathcal{C})}}+\norm{\bv_t}^2_{L^2_{(\mathcal{C})}}+\norm{\bv}^2_{L^2_{(\mathcal{C})}}(\norm{\bu}^2_{L^2_{(\mathcal{C})}}+\norm{\bv}^2_{H^1_{(\mathcal{C})}}))\norm{\bb}^2_{H^1_{(\mathcal{C})}}\\
+C\norm{\bb}_{H^2_{(\mathcal{C})}}^2\norm{\bb_t}_{H^1_{(\mathcal{C})}}^2+C(\norm{\bu}^2_{L^2_{(\mathcal{C})}}+\norm{\bv}^2_{H^3_{(\mathcal{C})}})\norm{\bb_t}^2_{L^2_{(\mathcal{C})}}\\
+C\norm{\bb}^2_{L^\infty_{(\mathcal{C})}}\norm{\bv}_{L^2_{(\mathcal{C})}}^2(\norm{\nabla_{\bx}\bv}^2_{L^2_{(\mathcal{C})}}+\norm{\bu}^2_{L^2_{(\mathcal{C})}}+\norm{\bv}^2_{L^2_{(\mathcal{C})}}).
\end{multline}

Estimates $\eqref{varrho-l-2-t-r}$, $\eqref{u-t-l-2-r}$ and $\eqref{b-t-l-2-r}$ together yield
\begin{multline}\label{ene-l-2-t}
\frac{\mathrm{d}}{\mathrm{d}t}\left(\int_{\mathcal{C}}\varrho_t^2\ddx+\int_{\mathcal{S}}\rho_{\mathcal{S}}|\bu_t|^2\ddx+\int_{\mathcal{C}}|\bb_t|^2\ddx\right)+ C\int_{\mathcal{C}}\left(|\nabla_{\bx}\bv_t|^2+|\nabla_{\bx}\bb_t|^2\right)\ddx\\
\leq \epsilon(\norm{\bv}^2_{H^3_{(\mathcal{C})}}+\norm{\varrho_t}^2_{L^2_{(\mathcal{C})}}+\|\bu_t\|^2_{L^2_{(\mathcal{C})}}+\|\bv_t\|^2_{L^2_{(\mathcal{C})}}+\norm{\bb_t}^2_{L^2_{(\mathcal{C})}})+Re_1,
\end{multline}
where 
\begin{equation*}
\begin{split}
Re_1&:=C(\norm{\nabla_{\bx}\varrho}^4_{H^1_{(\mathcal{C})}}+\norm{\varrho_t}^4_{L^2_{(\mathcal{C})}}+\norm{\varrho}^2_{H^2_{(\mathcal{C})}}\norm{\bv_t}^2_{H^1_{(\mathcal{C})}})\\
&\quad +C\norm{\varrho_t}^2_{L^2_{(\mathcal{C})}}(\norm{\bv}^2_{L^\infty_{(\mathcal{C})}}+\norm{\bu}^2_{L^2_{(\mathcal{C})}}+\norm{\bv}^2_{L^2_{(\mathcal{C})}})(\norm{{\bu}}^2_{L^2_{(\mathcal{C})}}+\norm{{\bv}}^2_{L^2_{(\mathcal{C})}}+\norm{{\nabla_\bx\bv}}^2_{L^2_{(\mathcal{C})}})\\
&\quad+C\norm{\bv}^2_{L^2_{(\mathcal{C})}}(\norm{\nabla_{\bx}\bv}^2_{L^2_{(\mathcal{C})}}+\norm{\bu}^2_{L^2_{(\mathcal{C})}}+\norm{\bv}^2_{L^2_{(\mathcal{C})}})\\
& \quad +C(\norm{\bu}^2_{L^2_{(\mathcal{C})}}+\norm{\bv}^2_{L^2_{(\mathcal{C})}})(\norm{\bv_t}^2_{L^2_{(\mathcal{C})}}+\norm{\bu_t}^2_{L^2_{(\mathcal{C})}}+\norm{\bv}^2_{L^2_{(\mathcal{C})}}(\norm{\nabla_{\bx}\bv}^2_{L^2_{(\mathcal{C})}}+\norm{\bu}^2_{L^2_{(\mathcal{C})}}+\norm{\bv}^2_{L^2_{(\mathcal{C})}}))
\\
& \quad+C\norm{\varrho_t}^2_{L^2_{(\mathcal{C})}}\norm{\bv}^2_{L^\infty_{(\mathcal{C})}}(\|\nabla_{\bx}\bv\|_{L^\infty_{(\mathcal{C})}}^2+\norm{\bu}^2_{L^2_{(\mathcal{C})}}+\norm{\bv}^2_{L^2_{(\mathcal{C})}})
\\
&\quad +C\norm{\bv_t}^2_{L^2_{(\mathcal{C})}}\norm{\bv}^2_{H^3_{(\mathcal{C})}}+C\norm{\bv}^2_{H^2_{(\mathcal{C})}}\norm{\bu_t}^2_{L^2_{(\mathcal{C})}}\\
&\quad +C\norm{\varrho_t}^2_{H^1_{(\mathcal{C})}}
\norm{\varrho}^2_{H^2_{(\mathcal{C})}}+C\norm{\bb}^2_{L^\infty_{(\mathcal{C})}}\norm{\bb_t}^2_{L^2_{(\mathcal{C})}}
\\&\quad +C(\norm{\bu_t}^2_{L^2_{(\mathcal{C})}}+\norm{\bv_t}^2_{L^2_{(\mathcal{C})}}+\norm{\bv}^2_{L^2_{(\mathcal{C})}}(\norm{\bu}^2_{L^2_{(\mathcal{C})}}+\norm{\bv}^2_{H^1_{(\mathcal{C})}}))\norm{\bb}^2_{H^1_{(\mathcal{C})}}\\
&\quad +C\norm{\bb}_{H^2_{(\mathcal{C})}}^2\norm{\bb_t}_{H^1_{(\mathcal{C})}}^2+C(\norm{\bu}^2_{L^2_{(\mathcal{C})}}+\norm{\bv}^2_{H^3_{(\mathcal{C})}})\norm{\bb_t}^2_{L^2_{(\mathcal{C})}}\\
&\quad +C\norm{\bb}^2_{L^\infty_{(\mathcal{C})}}\norm{\bv}_{L^2_{(\mathcal{C})}}^2(\norm{\nabla_{\bx}\bv}^2_{L^2_{(\mathcal{C})}}+\norm{\bu}^2_{L^2_{(\mathcal{C})}}+\norm{\bv}^2_{L^2_{(\mathcal{C})}}).
\end{split}
\end{equation*}

We pick $\boldsymbol{\phi}=\bu_t$ in $\eqref{re-system-momentum}$ and use $\eqref{bo}$ to get 
\begin{equation*}
\begin{split}
\norm{\bu_t}^2_{L^2_{(\mathcal{S})}}&\leq C(\norm{\nabla_{\bx}\bv}_{L^2_{(\mathcal{C})}}\norm{\nabla_{\bx}\bv_t}_{L^2_{(\mathcal{C})}}+\norm{\varrho}_{L^2_{(\mathcal{C})}}\norm{\nabla_{\bx}\bv_t}_{L^2_{(\mathcal{C})}}+|\bo_{\bu}|\norm{\bu_t}_{L^2_{(\mathcal{S})}}\norm{\bu}_{L^2_{(\mathcal{S})}}\\
&\quad+\norm{\bv}_{L^2_{(\mathcal{C})}}\norm{\nabla_{\bx}\bv}_{L^\infty_{(\mathcal{C})}}\norm{\bu_t}_{L^2_{(\mathcal{C})}}+(\norm{\bv}_{L^2_{(\mathcal{C})}}+\norm{\bu}_{L^2_{(\mathcal{C})}})\norm{\bv}_{L^2_{(\mathcal{C})}}\norm{\bu_t}_{L^2_{(\mathcal{C})}}\\
&\quad+\norm{\varrho}_{L^4_{(\mathcal{C})}}\norm{\nabla_{\bx}\varrho}_{L^4_{(\mathcal{C})}}\norm{\bu_t}_{L^2_{(\mathcal{C})}}+\norm{\bb}_{L^\infty_{(\mathcal{C})}}\norm{\bb}_{L^2_{(\mathcal{C})}}\norm{\nabla_{\bx}\bv_t}_{L^2_{(\mathcal{C})}}).
\end{split}
\end{equation*}
This together with $\eqref{bo}$ and $\eqref{u-B-C}$ yield 
\begin{multline}\label{l-2-t-u-new}
\norm{\bu_t}^2_{L^2_{(\mathcal{S})}}\leq \epsilon(\norm{\nabla_{\bx}\bv}^2_{L^2_{(\mathcal{C})}}+\norm{\varrho}^2_{L^2_{(\mathcal{C})}})+C(\norm{\nabla_{\bx}\bv_t}^2_{L^2_{(\mathcal{C})}}+\norm{\bu}^4_{L^2_{(\mathcal{C})}}+\norm{\bv}^4_{L^2_{(\mathcal{C})}}+\norm{\varrho}^4_{H^2_{(\mathcal{C})}}\\
+\norm{\bv}^2_{L^2_{(\mathcal{C})}}\norm{\nabla_{\bx}\bv}^2_{L^\infty_{(\mathcal{C})}}+\norm{\bb}^2_{H^2_{(\mathcal{C})}}\norm{\bb}^2_{L^2_{(\mathcal{C})}}).
\end{multline}

We multiply $\eqref{re-system-induction}_1$ by $\bb_t$, then integrate the resulting equation over $\mathcal{C}$ to get 
\begin{multline*}
\norm{\bb_t}^2_{L^2_{(\mathcal{C})}}\leq C(|\bo_{\bu}|\norm{\bb}_{L^2_{(\mathcal{C})}}\norm{\bb_t}_{L^2_{(\mathcal{C})}}+\norm{\bv}_{L^\infty_{(\mathcal{C})}}\norm{\nabla_{\bx}\bb}_{L^2_{(\mathcal{C})}}\norm{\bb_t}_{L^2_{(\mathcal{C})}}\\
+(\norm{\nabla_{\bx}\bv}_{L^2_{(\mathcal{C})}}+|\bo_{\bu}|)\norm{\bb}_{L^\infty_{(\mathcal{C})}}\norm{\bb_t}_{L^2_{(\mathcal{C})}}+\norm{\nabla_{\bx}\bb}_{L^2_{(\mathcal{C})}}\norm{\nabla_{\bx}\bb_t}_{L^2_{(\mathcal{C})}}),
\end{multline*}
which together with $\eqref{bo}$ leads to
\begin{multline}\label{l-2-t-b-new}
\norm{\bb_t}^2_{L^2_{(\mathcal{C})}}\leq \epsilon \norm{\nabla_{\bx}\bb}^2_{L^2_{(\mathcal{C})}}
+C(\norm{\nabla_{\bx}\bb_t}^2_{L^2_{(\mathcal{C})}}\\
+(\norm{\bu}^2_{L^2_{(\mathcal{C})}}+\norm{\bv}^2_{L^2_{(\mathcal{C})}})\norm{\bb}^2_{L^2_{(\mathcal{C})}}+\norm{\bv}_{L^\infty_{(\mathcal{C})}}^2\norm{\nabla_{\bx}\bb}^2_{L^2_{(\mathcal{C})}}\\
+
(\norm{\nabla_{\bx}\bv}^2_{L^2_{(\mathcal{C})}}+\norm{\bu}^2_{L^2_{(\mathcal{C})}}+\norm{\bv}^2_{L^2_{(\mathcal{C})}})\norm{\bb}^2_{L^\infty_{(\mathcal{C})}}).
\end{multline}

\subsection{Further Estimates}

In order to utilize the existing results of $a$ $priori$ estimates for the compressible N-S equations in \cite{V-1983-ASNSP}, we rewrite original system as follows.

$\bullet$ {\bf Continuity equation:}
\begin{equation}\label{varrho-new}
\begin{split}
\varrho_t+\bv\cdot\nabla_{\bx}\varrho+\bar{\rho}\mathrm{div}_{\bx}\bv=f_0, \quad  f_0:=-\varrho\mathrm{div}_{\bx}\bv.
\end{split}
\end{equation}

$\bullet$ {\bf Momentum equations:}
\begin{equation*}
\begin{split}
&\bv_t-\frac{1}{\bar{\rho}}\mathrm{div}_{\bx}S(\nabla_{\bx}\bv)+\frac{p'(\bar{\rho})}{\bar{\rho}}\nabla_{\bx}\varrho=\bm{f},\\
&\bm{f}:=-\frac{\varrho}{\bar{\rho}\rho}\mathrm{div}_{\bx}S(\nabla_{\bx}\bv)+\left(\frac{p'(\bar{\rho})}{\bar{\rho}}-\frac{p'({\rho})}{{\rho}}\right)\nabla_{\bx}\varrho-(\boldsymbol{\xi}_{\bu})_t-(\bo_{\bu})_t\times\bx\\
&\quad\quad -\bo_{\bu}\times\bu-\bv\cdot\nabla_{\bx}\bu+\frac{1}{\rho}\mathrm{div}_{\bx}\left(\bb\otimes\bb-\frac{1}{2}|\bb|^2\mathbb{I}\right).
\end{split}
\end{equation*}

$\bullet$ {\bf Induction equation:}
\begin{equation}\label{ind.equation}
\begin{split}
&\bb_t+\nabla_{\bx}\times(\nabla_{\bx}\times\bb)=\bm{f}_1,\\
&\bm{f}_1:=-(\bo_{\bu}\times\bb)+(\bo_{\bu}\times\bx+\boldsymbol{\xi}_{\bu})\cdot\nabla_{\bx}\bb-\bu\cdot\nabla_{\bx}\bb-\bb\mathrm{div}_{\bx}\bu+\bb\cdot\nabla_{\bx}\bu,\\
&\mathrm{with}\quad  \mathrm{div}_{\bx}\bb=0,\\
&\mathrm{and\hspace{0.1cm}boundary\hspace{0.1cm} condition\hspace{0.1cm}} \bb\cdot\bn=0,(\nabla_{\bx}\times\bb)\times\bn=0.
\end{split}
\end{equation}

 {\bf(First-order estimates of $\bv$ and $\bb$):}

We borrow the existing results from \cite[Lemma 4.1]{V-1983-ASNSP}. 
It holds that 
\begin{equation}\label{first-order-v}
\begin{split}
\frac{\mathrm{d}}{\mathrm{d}t}(\norm{\mathcal{D}\bv}^2_{L^2_{(\mathcal{C})}}+\norm{\mathrm{div}_{\bx}\bv}^2_{L^2_{(\mathcal{C})}})+C\norm{\nabla^2_{\bx}\bv}^2_{L^2_{(\mathcal{C})}}\leq C\norm{\varrho}^2_{H^1_{(\mathcal{C})}}
+\norm{\bm{f}}^2_{L^2_{(\mathcal{C})}},
\end{split}
\end{equation}
and (\cite[Proof of Lemma 2.2]{FaYu})
\begin{equation}\label{first-order-b}
\begin{split}
\frac{\mathrm{d}}{\mathrm{d}t}\norm{\nabla_{\bx}\times\bb}^2_{L^2_{(\mathcal{C})}}+C\norm{\nabla_{\bx}\times(\nabla_{\bx}\times\bb)}^2_{L^2_{(\mathcal{C})}}\leq C\norm{\bm{f}_1}^2_{L^2_{(\mathcal{C})}}.
\end{split}
\end{equation}

We take operator $\nabla_{\bx}$ on $\eqref{varrho-new}$, multiply it by $\nabla_{\bx}\varrho$ and integrate the resulting equation over $\mathcal{C}$ to get
\begin{equation}\label{first-order-rho}
\begin{split}
\frac{\mathrm{d}}{\mathrm{d}t}\norm{\nabla_{\bx}\varrho}^2_{H^1_{(\mathcal{C})}}\leq \epsilon\norm{\bv}^2_{H^3_{(\mathcal{C})}}+C
(\norm{\mathrm{div}_{\bx}\bv}^2_{H^2_{(\mathcal{C})}}+\norm{\varrho}^2_{H^2_{(\mathcal{C})}}+\norm{\varrho}^4_{H^2_{(\mathcal{C})}}). 
\end{split}
\end{equation}
Furthermore, it holds that 
\begin{equation}\label{first-order-rho-t}
\begin{split}
\norm{\varrho_t}^2_{H^1_{(\mathcal{C})}}\leq C(\norm{\bv}^2_{H^2_{(\mathcal{C})}}+\norm{\bv}^2_{H^2_{(\mathcal{C})}}\norm{\varrho}^2_{H^2_{(\mathcal{C})}}+\norm{f_0}^2_{H^1_{(\mathcal{C})}}). 
\end{split}
\end{equation}

{\bf(Stoke's type estimates and elliptic equation with Navier-slip boundary):}

Based on the result in \cite[Lemma 4.3]{V-1983-ASNSP} for Stokes's equation and \cite{ADN-64} for elliptic equation with Navier-slip boundary, we derive that
\begin{multline}\label{combine-first}
\norm{\varrho}^2_{H^1_{(\mathcal{C})}}+\norm{\bv}^2_{H^2_{(\mathcal{C})}}+\norm{\bb}^2_{H^2_{(\mathcal{C})}}\\
\leq C(\norm{\mathrm{div}_{\bx}\bv}^2_{H^1_{(\mathcal{C})}}+\norm{\bv_t}^2_{L^2_{(\mathcal{C})}}+\norm{\bb}^2_{L^2_{(\mathcal{C})}}+\norm{\bb_t}^2_{L^2_{(\mathcal{C})}}+\norm{\bm{f}}^2_{L^2_{(\mathcal{C})}}+\norm{\bm{f}_1}^2_{L^2_{(\mathcal{C})}}),
\end{multline}
and
\begin{multline}\label{second-order}
\norm{\varrho}^2_{H^2_{(\mathcal{C})}}+\norm{\bv}^2_{H^3_{(\mathcal{C})}}+\norm{\bb}^2_{H^3_{(\mathcal{C})}}\\
\leq C(\norm{\mathrm{div}_{\bx}\bv}^2_{H^2_{(\mathcal{C})}}+\norm{\nabla_{\bx}\bv_t}^2_{L^2_{(\mathcal{C})}}+\norm{\bb}^2_{L^2_{(\mathcal{C})}}+\norm{\bb_t}^2_{H^1_{(\mathcal{C})}}+\norm{\bm{f}}^2_{H^1_{(\mathcal{C})}}+\norm{\bm{f}_1}^2_{H^1_{(\mathcal{C})}}).
\end{multline}

By using $(4.40)$ in \cite{V-1983-ASNSP}, we deduce that 
\begin{multline}\label{first-order-t}
\frac{\mathrm{d}}{\mathrm{d}t}([\bv]_{1,2}^2+\norm{\varrho}^2_{H^1_{(\mathcal{C})}})+\norm{\mathrm{div}_{\bx}\bv}^2_{H^1_{(\mathcal{C})}}\leq \epsilon(\norm{\varrho}^2_{H^1_{(\mathcal{C})}}+\norm{\bv}^2_{H^3_{(\mathcal{C})}})+C(\norm{\bv}^2_{H^1_{(\mathcal{C})}}+\norm{\bv_t}^2_{L^2_{(\mathcal{C})}}\\
+\norm{\varrho}^4_{H^1_{(\mathcal{C})}}+\norm{\bm{f}}^2_{L^2_{(\mathcal{C})}}+\norm{f_0}^2_{H^1_{(\mathcal{C})}}+\norm{\bm{f}_1}^2_{L^2_{(\mathcal{C})}}), 
\end{multline}
where $[\cdot]_{1,2}$ denotes the sum of $L^2$-norms of the interior and tangential derivatives of order 1.

Furthermore, the following holds.
\begin{lemma}
It holds that
\begin{multline}\label{second-order-t}
\frac{\mathrm{d}}{\mathrm{d}t}([\bv]^2_{2,2}+\norm{\varrho}^2_{H^2_{(\mathcal{C})}}+[\bb]^2_{2,2})+\norm{\mathrm{div}_{\bx}\bv}^2_{H^2_{(\mathcal{C})}}\\
\leq \epsilon(\norm{\varrho}^2_{H^2_{(\mathcal{C})}}+\norm{\bv}^2_{H^3_{(\mathcal{C})}})+C(\norm{\bv}^2_{H^2_{(\mathcal{C})}}+\norm{\bv_t}^2_{H^1_{(\mathcal{C})}}\\
 +\norm{\varrho}^2_{H^1_{(\mathcal{C})}}+\norm{\varrho}^4_{H^2_{(\mathcal{C})}}+\norm{\bb}^2_{H^2_{(\mathcal{C})}}+\norm{\bb_t}^2_{H^1_{(\mathcal{C})}}+\norm{\bm{f}}^2_{H^1_{(\mathcal{C})}}+\norm{f_0}^2_{H^2_{(\mathcal{C})}}+\norm{\bm{f}_1}^2_{H^1_{(\mathcal{C})}}).
\end{multline}
Here $[\cdot]_{2,2}$ denotes the sum of $\|\cdot\|_{H^1_{\mathcal C}}$ and $L^2$ norms of the interior and tangential derivatives of order 2. It holds that $\frac 1C\|\cdot\|_{H^1_{(\mathcal C)}}\leq [\cdot]_{2,2}\leq C\|\cdot\|_{H^2_{(\mathcal C)}}$.\footnote{Hereinafter, we also use $\|\cdot\|$ for an arbitrary norm. In particular, the derivative of norms with respect to time cannot be replaced by the derivatives of another equivalent norm. Nevertheless, we use this simplification as it does not affect a validity of result and it improves the readability of the paper. }
\end{lemma}
\begin{proof}
The estimates concerning $\bv$ and $\varrho$ can be found in \cite[Chapter 4]{V-1983-ASNSP}. We focus here only on the magnetic equation.

For every $\delta>0$ there is $n\in \mathbb N$ and sets $\Omega_i,\, i\in \{0,\ldots,n\}$ such that $\mathcal C\subset \bigcup_{i=0}^n \Omega_i$ and $\Omega_0\subset\subset \Omega$. Further, for every $i\in \{1,\ldots,n\}$ there is a mapping $\Phi_i:\mathbb R^3\to \mathbb R^3$ and $R_i>0$ such that $\Phi_i^{-1}(V_i\cap \mathcal C) = B_{R_i}^+$ and $\Phi_i^{-1}(V_i\cap \partial\mathcal C) = B_{R_i}^{2}\times\{0\}$ where $V_i$ is an open set such that $\Omega_i\subset \overline \Omega_i\subset V_i$, $B_{R_i}^+$ is an upper half ball centered at zero, and $B_{R_i}^2$ is circle in two dimensions. Moreover, $\Phi_i$ can be chosen such that 
$$
\nabla_{\bx} \Phi_i = \mathcal R_i + \mathfrak R_i
$$
where $\mathcal R_i$ is an orthogonal matrix and $\|\mathfrak R_i\|_\infty<\delta$.

To get an interior estimate of $\bb$ we apply $\partial_j\partial_k$ on \eqref{ind.equation} and multiply it by $\chi_0^2\partial_j\partial_k \bb$ where $\chi_0$ is a smooth function with zero trace on $\partial \mathcal C$ which is identically equal to one on $\Omega_0$. An appropriate calculation and \eqref{first-order-b} yield
$$
\frac 12 \frac{d}{dt} \int_{\mathcal C}|\chi_0\partial_j\partial_k\bb|^2\ {\rm d}x + \int_{\mathcal C} |\chi_0\nabla_{\bx} \partial_j\partial_k \bb|^2\, {\rm d}x\leq \|{\bm f}_1\|_{H^1_{(\mathcal C)}}^2.
$$

Fix $i\in\{1,\ldots,n\}$. As \eqref{ind.equation} is invariant with respect to rotations, we assume $\mathcal R_i = \mathbb I$. The unknown $\hat \bb  = \bb\circ \Phi_i$ satisfies
\begin{equation}\label{narovnane}
\begin{split}
\hat \bb_t + \nabla_{\bx}\times(\nabla_{\bx}\times \hat \bb)  + \mathfrak C\nabla_{\bx}^2\hat \bb & = \hat {\bm f_1}\ \mbox{ on }V_i\\
\diver_{\bx} \hat \bb + \operatorname{Tr}\left(\nabla_{\bx} \hat \bb \mathfrak R\right) & = 0\ \mbox{ on }V_i\\
\hat \bb \cdot\left({\bf e}_3(\mathbb I + \mathfrak R)\right) &= 0 \mbox{ on }\mathbb R^2\times \{0\}\\
(\nabla_{\bx} \times \hat \bb)\times ({\bf e}_3(\mathbb I + \mathfrak R)) + (\mathfrak D\nabla_{\bx} \hat \bb)\times({\bf e}_3(\mathbb I + \mathfrak R)) & = 0 \mbox{ on }\mathbb R^2\times\{0\}
\end{split}
\end{equation}
for an appropriately chosen $\hat {\bm f_1}$. Here $\mathfrak C$ and $\mathfrak D$ denotes $\mathcal C^3$ tensors whose $L^\infty$ norm is controlled by $\delta$. Take a smooth function $\chi$ such that $\chi\equiv 1$ on $\Phi_i(\Omega_i)$ and it is zero on $\partial B_{R_i}^+\setminus B_{R_i}^2\times\{0\}$. We differentiate \eqref{narovnane}$_1$ with respect to $\partial_j\partial_k$, multiply by $\chi^2\partial_j
\partial_k \hat \bb$, $j,k\in \{1,2\}$ and integrate over $V_i$ to get
\begin{equation}
\frac 12 \frac d{dt}\int_{V_i}|\chi\partial_j\partial_k \hat \bb|^2\, {\rm d}x + \int_{V_i} \nabla_{\bx} \times (\nabla_{\bx}\times\partial_j\partial_k \hat \bb) \chi^2\partial_j\partial_k \hat \bb\, {\rm d}x 
 \leq \delta\int_{V_i} |\nabla_{\bx} \chi \partial_j\partial_k \hat\bb|\, {\rm d}x +  l.o.t.\, ,\label{eq.indukce.1}
\end{equation}
hereinafter, $l.o.t.$ abbreviates lower order terms which might be estimated in means of \eqref{first-order-b}, a constant $C(\|\mathfrak R\|_{W^{3,\infty}})$ and $\|{\bm f}_1\|_{H^1_{(\mathcal C)}}^2$. To handle the second term on the left-hand side, we establish two symbols, namely
$$
\mathbb S_1(v) = \begin{pmatrix}0 & v_3 & -v_2\\ -v_3 & 0 & v_1 \\ v_2 & -v_1 & 0\end{pmatrix},\mbox{ and }\mathbb S_2(v) = \begin{pmatrix}0 & (\mathfrak R v)_3 & -(\mathfrak R v)_2\\ -(\mathfrak R v)_3 & 0 & (\mathfrak R v)_1\\ (\mathfrak R v)_2 & -(\mathfrak R v)_1 & 0 \end{pmatrix}.
$$
Since $\nabla_{\bx} \times \bb = \mathbb S_1(\partial) \bb$ we get 
\begin{multline*}
\int_{V_i} \mathbb S_1(\partial) (\nabla_{\bx} \times\partial_j\partial_k \hat \bb) \chi^2\partial_j\partial_k \hat \bb\, {\rm d}x  =  \\
\int_{V_i} \left(\mathbb S_1(\partial) + \mathbb S_2(\partial)\right) \partial_j\partial_k \left(\nabla_{\bx} \times\hat \bb + \mathfrak D\nabla_{\bx} \hat \bb\right) \chi^2\partial_j\partial_k \hat \bb\, {\rm d}x \\
 - \int_{V_i} \left(\mathbb S_1(\partial) + \mathbb S_2(\partial)\right) (\partial_j\partial_k \mathfrak D\nabla_{\bx} \hat \bb) \chi^2\partial_j\partial_k \hat \bb\, {\rm d}x 
 - \int_{V_i} \mathbb S_2(\partial) (\nabla_{\bx} \times\partial_j\partial_k \hat \bb) \chi^2\partial_j\partial_k \hat \bb\, {\rm d}x\\
 = 
-\int_{V_i}  \partial_j\partial_k \left(\nabla_{\bx} \times\hat \bb + \mathfrak D\nabla_{\bx} \hat \bb\right) \left(\mathbb S_1(\partial) + \mathbb S_2(\partial)\right)^T\left(\chi^2\partial_j\partial_k \hat \bb\right)\, {\rm d}x  + l.o.t.\\
 - \int_{V_i} \left(\mathbb S_1(\partial) + \mathbb S_2(\partial)\right) (\partial_j\partial_k \mathfrak D\nabla_{\bx} \hat \bb) \chi^2\partial_j\partial_k \hat \bb\, {\rm d}x 
 - \int_{V_i} \mathbb S_2(\partial) (\nabla_{\bx} \times\partial_j\partial_k \hat \bb) \chi^2\partial_j\partial_k \hat \bb\, {\rm d}x\\
 = 
\int_{V_i}  \left(\partial_j\partial_k \nabla_{\bx} \times\hat \bb\right) \left(\nabla_{\bx} \times\left(\chi^2\partial_j\partial_k \hat \bb\right)\right)\, {\rm d}x\\
-\int_{V_i}  \partial_j\partial_k \mathfrak D\nabla_{\bx} \hat \bb \left(\mathbb S_1(\partial) + \mathbb S_2(\partial)\right)^T\left(\chi^2\partial_j\partial_k \hat \bb\right)\, {\rm d}x\\
-\int_{V_i}  \partial_j\partial_k \nabla_{\bx} \times\hat \bb \mathbb S_2(\partial)^T\left(\chi^2\partial_j\partial_k \hat \bb\right)\, {\rm d}x
  + l.o.t.\\
 - \int_{V_i} \left(\mathbb S_1(\partial) + \mathbb S_2(\partial)\right) (\partial_j\partial_k \mathfrak D\nabla_{\bx} \hat \bb) \chi^2\partial_j\partial_k \hat \bb\, {\rm d}x 
 - \int_{V_i} \mathbb S_2(\partial) (\nabla_{\bx} \times\partial_j\partial_k \hat \bb) \chi^2\partial_j\partial_k \hat \bb\, {\rm d}x\\
\geq \int_{V_i} |\nabla_{\bx}  \chi \partial_j\partial_k \hat b|\, {\rm d}x - \delta\int_{V_i}|\nabla_{\bx}  \chi \partial_j\partial_k \hat b| - l.o.t.
\end{multline*}
Assuming $\delta $ is small enough we obtain from \eqref{eq.indukce.1}
$$
\frac 12\frac d{dt}\int_{V_i}|\chi\partial_j\partial_k \hat \bb|^2\, {\rm d}x + \int_{V_i} |\nabla_{\bx}\chi  \partial_j\partial_k \hat \bb|^2\, {\rm d}x \leq C\|\bm f_1\|_{H^1_{(\mathcal C)}}.
$$
whenever $j,k\in \{1,2\}$.
\end{proof}

In the spirit of $\eqref{xi-u}$, $\eqref{bo}$, $\eqref{xi-u-t-1}$ and  $\eqref{bo-t}$,  Sobolev's inequality and direct calculation lead to the following set of estimates
\begin{equation*}
\begin{split}
\norm{f_0}_{L^2_{(\mathcal{C})}}\leq C\norm{\varrho}_{H^2_{(\mathcal{C})}}\norm{\mathrm{div}_{\bx}\bv}_{L^2_{(\mathcal{C})}},
\end{split}
\end{equation*} 

\begin{equation*}
\norm{\nabla_{\bx}f_0}_{L^2_{(\mathcal{C})}}\leq C\left(\norm{\nabla_{\bx}\varrho\mathrm{div}_{\bx}\bv}_{L^2_{(\mathcal{C})}}+\norm{\varrho\nabla_{\bx}\mathrm{div}_{\bx}\bv}_{L^2_{(\mathcal{C})}}\right)\leq C\norm{\varrho}_{H^2_{(\mathcal{C})}}\norm{\bv}_{H^2_{(\mathcal{C})}},
\end{equation*}

\begin{equation*}
\begin{split}
\norm{\nabla^2_{\bx}f_0}_{L^2_{(\mathcal{C})}}&\leq C\left(\norm{\nabla^2_{\bx}\varrho\mathrm{div}_{\bx}\bv}_{L^2_{(\mathcal{C})}}+\norm{\nabla_{\bx}\varrho\nabla_{\bx}\mathrm{div}_{\bx}\bv}_{L^2_{(\mathcal{C})}}+\norm{\varrho\nabla^2_{\bx}\mathrm{div}_{\bx}\bv}_{L^2_{(\mathcal{C})}}\right)\leq C\norm{\varrho}_{H^2_{(\mathcal{C})}}\norm{\nabla_{\bx}\bv}_{H^2_{(\mathcal{C})}},
\end{split}
\end{equation*}

\begin{equation*}
\begin{split}
\norm{f_0}^2_{H^2_{(\mathcal{C})}}\leq C\norm{\varrho}^2_{H^2_{(\mathcal{C})}}\norm{\bv}^2_{H^3_{(\mathcal{C})}},
\end{split}
\end{equation*}

\begin{equation*}
\begin{split}
\norm{\bm{f}}_{L^2_{(\mathcal{C})}}&\leq C(\norm{\varrho|\nabla^2_{\bx}\bv|}_{L^2_{(\mathcal{C})}}+\norm{\varrho|\nabla_{\bx}\varrho|}_{L^2_{(\mathcal{C})}}+|(\boldsymbol{\xi}_{\bu})_t|+|(\bo_{\bu})_t|+|\bo_{\bu}|\norm{\bu}_{L^2_{(\mathcal{C})}}+|\bo_{\bu}|\norm{\bv}_{L^2_{(\mathcal{C})}}\\
&\quad+\norm{|\bv||\nabla_{\bx}\bv|}_{L^2_{(\mathcal{C})}}+\norm{\bb}_{L^\infty_{(\mathcal{C})}}\norm{\nabla_{\bx}\bb}_{L^2_{(\mathcal{C})}})\\
&\leq C(\norm{\varrho}_{H^1_{(\mathcal{C})}}\norm{\nabla^2_{\bx}\bv}_{H^1_{(\mathcal{C})}}+\norm{\varrho}_{H^1_{(\mathcal{C})}}\norm{\nabla_{\bx}\varrho}_{H^1_{(\mathcal{C})}}+\norm{\bu_t}_{L^2_{(\mathcal{C})}}+\norm{\bv_t}_{L^2_{(\mathcal{C})}}\\
&\quad+\norm{\bv}_{H^2_{(\mathcal{C})}}^2+\norm{\bu}_{L^2_{(\mathcal{C})}}^2+\norm{\bb}_{L^\infty_{(\mathcal{C})}}\norm{\nabla_{\bx}\bb}_{L^2_{(\mathcal{C})}}),
\end{split}
\end{equation*}

\begin{equation*}
\begin{split}
\norm{\nabla_{\bx}\bm{f}}_{L^2_{(\mathcal{C})}}&\leq C(\norm{|\nabla_{\bx}\varrho||\nabla^2_{\bx}\bv|}_{L^2_{(\mathcal{C})}}+\norm{|\varrho||\nabla^3_{\bx}\bv|}_{L^2_{(\mathcal{C})}}+\norm{|\nabla_{\bx}\varrho||\nabla_{\bx}\varrho|}_{L^2_{(\mathcal{C})}}+\norm{\varrho|\nabla^2_{\bx}\varrho|}_{L^2_{(\mathcal{C})}}\\
&\quad +|(\bo_{\bu})_t|+|\bo_{\bu}|\norm{\nabla_{\bx}\bv}_{L^2_{(\mathcal{C})}}+|\bo_{\bu}|^2+\norm{|\nabla_{\bx}\bv||\nabla_{\bx}\bv|}_{L^2_{(\mathcal{C})}}\\
&\quad +\norm{\nabla_{\bx}\bv}_{L^2_{(\mathcal{C})}}|\bo_{\bu}|+\norm{|\bv||\nabla^2_{\bx}\bv|}_{L^2_{(\mathcal{C})}}+\norm{\bb}_{H^2_{(\mathcal{C})}}^2+\norm{\varrho}_{H^2_{(\mathcal{C})}}\norm{\bb}_{H^2_{(\mathcal{C})}}^2)\\
&\leq C(\norm{\varrho}_{H^2_{(\mathcal{C})}}\norm{\bv}_{H^3_{(\mathcal{C})}}
+\norm{\varrho}^2_{H^2_{(\mathcal{C})}}+\norm{\bu_t}_{L^2_{(\mathcal{C})}}+\norm{\bv_t}_{L^2_{(\mathcal{C})}}\\
&\quad+\norm{\bu}^2_{L^2_{(\mathcal{C})}}+\norm{\bv}^2_{H^2_{(\mathcal{C})}}+\norm{\bb}^2_{H^2_{(\mathcal{C})}}+\norm{\varrho}_{H^2_{(\mathcal{C})}}\norm{\bb}_{H^2_{(\mathcal{C})}}^2),
\end{split}
\end{equation*}

\begin{equation*}
\begin{split}
\norm{\bm{f}}^2_{H^1_{(\mathcal{C})}}&\leq C(\norm{\varrho}^2_{H^2_{(\mathcal{C})}}\norm{\bv}^2_{H^3_{(\mathcal{C})}}+\norm{\varrho}^4_{H^2_{(\mathcal{C})}}+\norm{\bu_t}^2_{L^2_{(\mathcal{C})}}\\
&\quad+\norm{\bv_t}^2_{L^2_{(\mathcal{C})}}+\norm{\bu}^4_{L^2_{(\mathcal{C})}}+\norm{\bv}^4_{H^2_{(\mathcal{C})}}+\norm{\bb}^4_{H^2_{(\mathcal{C})}}+\norm{\varrho}_{H^2_{(\mathcal{C})}}^2\norm{\bb}_{H^2_{(\mathcal{C})}}^4),
\end{split}
\end{equation*}

\begin{equation*}
\begin{split}
\norm{\bm{f}_1}_{L^2_{(\mathcal{C})}}&\leq C(|\bo_{\bu}|\norm{\bb}_{L^2_{(\mathcal{C})}}+(|\bo_{\bu}|+|\boldsymbol{\xi}_{\bu}|)\norm{\nabla_{\bx}\bb}_{L^2_{(\mathcal{C})}}\\
&\quad+\norm{\bb}_{L^\infty_{(\mathcal{C})}}\norm{\nabla_{\bx}\bu}_{L^2_{(\mathcal{C})}}+\norm{\bu}_{L^\infty_{(\mathcal{C})}}\norm{\nabla_{\bx}\bb}_{L^2_{(\mathcal{C})}})\\
&\leq C((\norm{\bu}_{L^2_{(\mathcal{C})}}+\norm{\bv}_{L^2_{(\mathcal{C})}})\norm{\bb}_{H^1_{(\mathcal{C})}}+\norm{\bb}_{H^2_{(\mathcal{C})}}\norm{\nabla_{\bx}\bu}_{L^2_{(\mathcal{C})}}+\norm{\bu}_{H^2_{(\mathcal{C})}}\norm{\nabla_{\bx}\bb}_{L^2_{(\mathcal{C})}}),
\end{split}
\end{equation*}

\begin{equation*}
\begin{split}
\norm{\nabla_{\bx}\bm{f}_1}_{L^2_{(\mathcal{C})}}&\leq C (|\bo_{\bu}|\norm{\nabla_{\bx}\bb}_{L^2_{(\mathcal{C})}}+(|\boldsymbol{\xi}_{\bu}|+|\bo_{\bu}|)\norm{\nabla^2_{\bx}\bb}_{L^2_{(\mathcal{C})}}\\
&\quad+\norm{\nabla_{\bx}^2\bb}_{L^2_{(\mathcal{C})}}\norm{\bu}_{L^\infty_{(\mathcal{C})}}+\norm{|\nabla_{\bx}\bu||\nabla_{\bx}\bb|}_{L^2_{(\mathcal{C})}}+\norm{\nabla^2_{\bx}\bu}_{L^2_{(\mathcal{C})}}\norm{\bb}_{L^\infty_{(\mathcal{C})}})\\
&\leq C((\norm{\bu}_{L^2_{\mathcal{C}}}+\norm{\bv}_{L^2_{(\mathcal{C})}})\norm{\bb}_{H^2_{(\mathcal{C})}}+\norm{\bb}_{H^2_{(\mathcal{C})}}\norm{\bu}_{H^2_{(\mathcal{C})}}),
\end{split}
\end{equation*}

\begin{equation*}
\begin{split}
\norm{\bm{f}_1}^2_{H^1_{(\mathcal{C})}}\leq C((\norm{\bu}^2_{L^2_{(\mathcal{C})}}+\norm{\bv}^2_{L^2_{(\mathcal{C})}})\norm{\bb}^2_{H^2_{(\mathcal{C})}}+\norm{\bb}^2_{H^2_{(\mathcal{C})}}\norm{\bu}^2_{H^2_{(\mathcal{C})}}).
\end{split}
\end{equation*}

Thanks to $\bu_t|_{\mathcal{B}}=((\bo_{\bu})_t\times{\bx}+(\bxi_{\bu})_t)|_{\mathcal{B}}$, we get 
\begin{equation}\label{w-u-t}
\begin{split}
|(\bo_{\bu})_t|\sim\norm{(\bo_{\bu})_t\times \bx}_{L^2_{(\mathcal{B})}}\leq C\norm{\bu_t}_{L^2_{(\mathcal{B})}}+C|(\bxi_{\bu})_t|, 
\end{split}
\end{equation}
which together with $\eqref{xi-u-t-1}$, $\eqref{w-u-t}$ and $\bv=\bu-\bo_{\bu}\times \bx-\boldsymbol{\xi}_{\bu}$, 
yield
\begin{equation}\label{v-t-u-t}
\begin{split}
\norm{\bv_t}_{L^2_{(\mathcal{C})}}&\leq C(\norm{\bu_t}_{L^2_{(\mathcal{C})}}+\norm{(\bo_{\bu})_t \times \bx}_{L^2_{(\mathcal{C})}}+|(\boldsymbol{\xi}_{\bu})_t|)\\
&\leq C(\norm{\bu_t}_{L^2_{(\mathcal{C})}}+|(\bo_{\bu})_t|+|(\boldsymbol{\xi}_{\bu})_t|)\\
&\leq C(\norm{\bu_t}_{L^2_{(\mathcal{C})}}+\norm{\bu_t}_{L^2_{(\mathcal{B})}}+\norm{\bv}_{L^2_{(\mathcal{C})}}(\norm{\nabla_{\bx}\bv}_{L^2_{(\mathcal{C})}}+\norm{\bv}_{L^2_{(\mathcal{C})}}+\norm{\bu}_{L^2_{(\mathcal{C})}}))\\
&\leq C(\norm{\bu_t}_{L^2_{(\mathcal{S})}}
+\norm{\bv}_{L^2_{(\mathcal{C})}}(\norm{\nabla_{\bx}\bv}_{L^2_{(\mathcal{C})}}+\norm{\bv}_{L^2_{(\mathcal{C})}}+\norm{\bu}_{L^2_{(\mathcal{C})}}))
\end{split}
\end{equation}

Combining $\eqref{first-order-v}$, $\eqref{first-order-b}$, $\eqref{first-order-rho}$, $\eqref{first-order-rho-t}$, $\eqref{second-order}$, $\eqref{zero-order}$ and $\eqref{ene-l-2-t}$,  such that the terms $\norm{\nabla_\bx\bv_t}_{L^2_{(\mathcal{C})}}+\norm{\nabla_\bx\bb_t}_{L^2_{(\mathcal{C})}}$ appearing on the right-hand side of $\eqref{second-order}$ can be absorbed by left-hand terms in  $\eqref{ene-l-2-t}$, the right-hand side term $\norm{\varrho}^2_{H^1_{(\mathcal{C})}}$ in $\eqref{first-order-v}$, the right-hand side term $\norm{\varrho}^2_{H^2_{(\mathcal{C})}}$ in $\eqref{first-order-rho}$ and the right-hand side term $\norm{\bv}^2_{H^2_{(\mathcal{C})}}$ in $\eqref{first-order-rho-t}$  can be  absorbed by the left-hand side terms in $\eqref{second-order}$, we then use $\eqref{v-t-u-t}$ to get 
\begin{equation}\label{three-order}
\begin{split}
&\frac{\mathrm{d}}{\mathrm{d}t}(\norm{\mathcal{D}\bv}_{L^2_{(\mathcal{C})}}^2+\norm{\mathrm{div}_{\bx}\bv}_{L^2_{(\mathcal{C})}}^2+\norm{\varrho}_{H^1_{(\mathcal{C})}}^2+\norm{\varrho_t}_{L^2_{(\mathcal{C})}}^2+
\norm{\sqrt{\rho_{\mathcal{S}}}\bu}_{L^2_{(\mathcal{S})}}^2+\norm{\sqrt{\rho_{\mathcal{S}}}\bu_t}_{L^2_{(\mathcal{S})}}^2\\
&
\hspace{0.5cm}+\norm{\bb}_{L^2_{(\mathcal{C})}}^2+\norm{\nabla_{\bx}\times\bb}_{L^2_{(\mathcal{C})}}^2+\norm{\bb_t}_{L^2_{(\mathcal{C})}}^2)\\
&+\norm{\bv}_{H^3_{(\mathcal{C})}}^2+\norm{\bb}_{H^3_{(\mathcal{C})}}^2+\norm{\varrho}_{H^2_{(\mathcal{C})}}^2+\norm{\varrho_t}_{H^1_{\mathcal{(C)}}}^2+\norm{\nabla_\bx\bv_t}_{L^2_{(\mathcal{C)}}}^2+\norm{\nabla_\bx\bb_t}_{L^2_{(\mathcal{C)}}}^2\\
&\leq \epsilon(\norm{\bv}^2_{H^3_{(\mathcal{C})}}+\norm{\bu}^2_{L^2_{(\mathcal{C})}}+\norm{\bu_t}^2_{L^2_{(\mathcal{S})}}+\norm{\varrho_t}^2_{L^2_{(\mathcal{C})}}+\norm{\bb_t}^2_{L^2_{(\mathcal{C})}})\\
 &+C(\norm{\mathrm{div}_{\bx}\bv}^2_{H^2_{(\mathcal{C})}}+\norm{\bb}^2_{L^2_{(\mathcal{C})}}+\norm{\bb_t}^2_{L^2_{(\mathcal{C})}})+Re_1+Re_2,
\end{split}
\end{equation}
where 
\begin{equation*}
\begin{split}
Re_2&=:C(\norm{\varrho}^4_{L^2_{(\mathcal{C})}}+\norm{{\nabla_{\bx}\varrho}}^2_{L^2_{(\mathcal{C})}}\norm{{\varrho}}^2_{L^\infty_{(\mathcal{C})}}+\norm{\bb}^2_{H^2_{(\mathcal{C})}}\norm{\bu}^2_{L^2_{(\mathcal{C})}}\\
&+\norm{\bb}^2_{H^2_{(\mathcal{C})}}\norm{\bb}^2_{L^2_{(\mathcal{C})}}+\norm{\varrho}^4_{H^2_{(\mathcal{C})}}+\norm{\bv}^2_{H^2_{(\mathcal{C})}}\norm{\varrho}^2_{H^2_{(\mathcal{C})}}+\norm{\bm{f}}^2_{H^1_{(\mathcal{C})}}+\norm{f_0}^2_{H^1_{(\mathcal{C})}}+\norm{\bm{f}_1}^2_{H^1_{(\mathcal{C})}}\\
&+\norm{\bv}^2_{L^2_{(\mathcal{C})}}(\norm{\nabla_{\bx}\bv}^2_{L^2_{(\mathcal{C})}}+\norm{\bv}^2_{L^2_{(\mathcal{C})}}+\norm{\bu}^2_{L^2_{(\mathcal{C})}})).
\end{split}
\end{equation*}

Putting $\eqref{three-order}$, $\eqref{l-2-t-u-new}$ and $\eqref{l-2-t-b-new}$ together, such that the right-hand side term $\norm{\nabla_{\bx}\bv_t}_{L^2_{(\mathcal{C})}}^2$ in $\eqref{l-2-t-u-new}$ and the right-hand side term  $\norm{\nabla_{\bx}\bb_t}_{L^2_{(\mathcal{C})}}^2$ in $\eqref{l-2-t-b-new}$ can be absorbed by the left-hand side terms in $\eqref{three-order}$, then we derive that 
\begin{equation}\label{three-order-1}
\begin{split}
&\frac{\mathrm{d}}{\mathrm{d}t}(\norm{\mathcal{D}\bv}_{L^2_{(\mathcal{C})}}^2+\norm{\mathrm{div}_{\bx}\bv}_{L^2_{(\mathcal{C})}}^2+\norm{\varrho}_{H^2_{(\mathcal{C})}}^2+\norm{\varrho_t}_{L^2_{(\mathcal{C})}}^2+
\norm{\sqrt{\rho_{\mathcal{S}}}\bu}_{L^2_{(\mathcal{S})}}^2+\norm{\sqrt{\rho_{\mathcal{S}}}\bu_t}_{L^2_{(\mathcal{S})}}^2\\
&\hspace{0.5cm}+\norm{\bb}_{L^2_{(\mathcal{C})}}^2+\norm{\nabla_{\bx}\times\bb}_{L^2_{(\mathcal{C})}}^2+\norm{\bb_t}_{L^2_{(\mathcal{C})}}^2)\\
&+\norm{\bv}_{H^3_{(\mathcal{C})}}^2+\norm{\bb}_{H^3_{(\mathcal{C})}}^2+\norm{\varrho}_{H^2_{(\mathcal{C})}}^2+\norm{\varrho_t}_{H^1_{\mathcal{(C)}}}^2+\norm{\nabla_x\bv_t}_{L^2_{(\mathcal{C)}}}^2+\norm{\bu_t}_{L^2_{(\mathcal{S)}}}^2+\norm{\bb_t}_{H^1_{(\mathcal{C)}}}^2\\
&\leq \epsilon(\norm{\bv}^2_{H^3_{(\mathcal{C})}}+\norm{\nabla_{\bx}\bb}^2_{L^2_{(\mathcal{C})}}+\norm{\bu}^2_{L^2_{(\mathcal{C})}}+\norm{\bu_t}^2_{L^2_{(\mathcal{S})}}+\norm{\varrho}^2_{L^2_{(\mathcal{C})}}+\norm{\varrho_t}^2_{L^2_{(\mathcal{C})}}+\norm{\bb_t}^2_{L^2_{(\mathcal{C})}})\\
 &+C(\norm{\mathrm{div}_{\bx}\bv}^2_{H^2_{(\mathcal{C})}}+\norm{\bb}^2_{L^2_{(\mathcal{C})}}+\norm{\bb_t}^2_{L^2_{(\mathcal{C})}})+Re_1+Re_2+Re_3,
\end{split}
\end{equation}
where 
\begin{equation*}
\begin{split}
Re_3&=:C(\norm{\bu}^4_{L^2_{(\mathcal{C})}}+\norm{\bv}^4_{L^2_{(\mathcal{C})}}+\norm{\varrho}^4_{H^2_{(\mathcal{C})}} +\norm{\bv}^2_{L^2_{(\mathcal{C})}}\norm{\nabla_{\bx}\bv}^2_{L^\infty_{(\mathcal{C})}}+\norm{\bb}^2_{H^2_{(\mathcal{C})}}\norm{\bb}^2_{L^2_{(\mathcal{C})}}\\
&+(\norm{\bu}^2_{L^2_{(\mathcal{C})}}+\norm{\bv}^2_{L^2_{(\mathcal{C})}})\norm{\bb}^2_{L^2_{(\mathcal{C})}}+\norm{\bv}_{L^\infty_{(\mathcal{C})}}^2\norm{\nabla_{\bx}\bb}^2_{L^2_{(\mathcal{C})}}+
(\norm{\nabla_{\bx}\bv}^2_{L^2_{(\mathcal{C})}}+\norm{\bu}^2_{L^2_{(\mathcal{C})}}+\norm{\bv}^2_{L^2_{(\mathcal{C})}})\norm{\bb}^2_{L^\infty_{(\mathcal{C})}})).
\end{split}
\end{equation*}

To eliminate the right-hand side term $\norm{\mathrm{div}_{\bx}\bv}^2_{H^2_{(\mathcal{C})}}$ in $\eqref{three-order-1}$,  we combine $\eqref{three-order-1}$ with $\eqref{second-order-t}$, and use $\eqref{combine-first}$ to absorb the right-hand side terms $\norm{\bv}^2_{H^2_{(\mathcal{C})}}$, $\norm{\bb}^2_{H^2_{(\mathcal{C})}}$ and $\|\varrho\|^2_{H^1_{(\mathcal{C})}}$ in $\eqref{second-order-t}$, then the above procedures arise the term $\norm{\mathrm{div}_{\bx}\bv}^2_{H^1_{(\mathcal{C})}}$. We use $\eqref{first-order-t}$ to absorb the term $\norm{\mathrm{div}_{\bx}\bv}^2_{H^1_{(\mathcal{C})}}$ arising from the right-hand side terms in above estimates, then put the resulting estimate, $\eqref{zero-order}$, $\eqref{ene-l-2-t}$, $\eqref{l-2-t-b-new}$, and $\eqref{v-t-u-t}$ together to absorb the right-hand side terms  $\norm{\bv}^2_{H^1_{(\mathcal{C})}}$, $\norm{\bb}^2_{H^1_{(\mathcal{C})}}$, $\norm{\bv_t}^2_{H^1_{(\mathcal{C})}}$ and $\norm{\bb_t}^2_{H^1_{(\mathcal{C})}}$. (Here we have used the Poincare inequality $\norm{\bv}_{L^2_{(\mathcal{C})}}\leq C\norm{\nabla_{\bx}\bv}_{L^2_{(\mathcal{C})}}$, and the Gagliardo-Nirenberg inequality $\norm{\bb}_{L^2_{(\mathcal{C})}}\leq C\norm{\bb}_{L^6_{(\mathcal{C})}}\leq C\norm{\nabla_{\bx}\bb}_{L^2_{(\mathcal{C})}}$ in \cite[Lemma A.1]{CHPS-23} and \cite{Nir-1959}.) Lastly, the smallness of $\epsilon$ helps us to get that 
\begin{equation}\label{three-order-2}
\begin{split}
&\frac{\mathrm{d}}{\mathrm{d}t}\mathsf{E}+\norm{\bv}_{H^3_{(\mathcal{C})}}^2+\norm{\bb}_{H^3_{(\mathcal{C})}}^2+\norm{\varrho}_{H^2_{(\mathcal{C})}}^2+\norm{\varrho_t}_{H^1_{\mathcal{(C)}}}^2+\norm{\bv_t}_{H^1_{(\mathcal{C)}}}^2+\norm{\bu_t}_{L^2_{(\mathcal{S)}}}^2+\norm{\bb_t}_{H^1_{(\mathcal{C)}}}^2\\
&\leq \epsilon\norm{\bu}^2_{L^2_{(\mathcal{C})}}+Re_1+Re_2+Re_3+Re_4,
\end{split}
\end{equation}
where 
\begin{equation}\label{define-e}
\begin{split}
&\mathsf{E}:=([\bv]_{2,2}^2+\norm{\varrho}_{H^2_{(\mathcal{C})}}^2+\norm{\varrho_t}_{L^2_{(\mathcal{C})}}^2+
\norm{\sqrt{\rho_{\mathcal{S}}}\bu}_{L^2_{(\mathcal{S})}}^2+\norm{\sqrt{\rho_{\mathcal{S}}}\bu_t}_{L^2_{(\mathcal{S})}}^2+[\bb]_{2,2}^2+\norm{\bb_t}_{L^2_{(\mathcal{C})}}^2), 
\end{split}
\end{equation}

\begin{equation*}
\begin{split}
Re_4:=C\left(\norm{\varrho}^4_{H^2_{(\mathcal{C})}}+\norm{\bm{f}}^2_{H^1_{(\mathcal{C})}}+\norm{f_0}^2_{H^2_{(\mathcal{C})}}+\norm{\bm{f}_1}^2_{H^1_{(\mathcal{C})}}\right).
\end{split}
\end{equation*}
Here, we note that the remaining terms may contain some same terms in $Re_i$ with $i=1,2,3,4$.

The dissipative terms in $\eqref{three-order-2}$ can not absorb $\epsilon \norm{\bu}^2_{L^2_{(\mathcal{C})}}$, therefore we add term $|\bxi_{\bu}|^2+|\bo_{\bu}|^2$ to the both sides of $\eqref{three-order-2}$. 

By $\norm{\bu}^2_{L^2_{(\mathcal{C})}}\leq C(\norm{\bv}^2_{L^2_{(\mathcal{C})}}+|\bxi_{\bu}|^2+|\bo_{\bu}|^2)$ we get 
\begin{equation}\label{E-D-re}
\begin{split}
&\frac{\mathrm{d}}{\mathrm{d}t}\mathsf{E}+\mathsf{D}\leq Re_1+Re_2+Re_3+Re_4+C\left(|\bxi_{\bu}|^2+|\bo_{\bu}|^2\right),
\end{split}
\end{equation}
where
\begin{equation*}
\begin{split}
&\mathsf{D}:=\norm{\bv}_{H^3_{(\mathcal{C})}}^2+\norm{\bb}_{H^3_{(\mathcal{C})}}^2+\norm{\varrho}_{H^2_{(\mathcal{C})}}^2+\norm{\varrho_t}_{H^1_{\mathcal{(C)}}}^2+\norm{\bv_t}_{H^1_{(\mathcal{C)}}}^2+\norm{\bu_t}_{L^2_{(\mathcal{S)}}}^2+\norm{\bb_t}_{H^1_{(\mathcal{C)}}}^2+|\bxi_{\bu}|^2+|\bo_{\bu}|^2. 
\end{split}
\end{equation*}

Direct calculations lead to 
\begin{equation*}
	\begin{split}
		Re_1+Re_2+Re_3+Re_4\leq C\mathsf{D}(\mathsf{E}+\mathsf{E}^2), 
\end{split}
\end{equation*}
and \eqref{E-D-re} gives that 
\begin{equation}\label{E-D}
\begin{split}
&\frac{\mathrm{d}}{\mathrm{d}t}\mathsf{E}+\mathsf{D}\leq \mathsf{D}(\mathsf{E}+\mathsf{E}^2)+|\bxi_{\bu}|^2+|\bo_{\bu}|^2.
\end{split}
\end{equation}
Next, the definitions of $\mathsf D$ and $\mathsf E$ yield
\begin{equation}\label{E-D-1}
\begin{split}
\mathsf{E}\leq C\mathsf{D}.
\end{split}
\end{equation}

Thanks to Sobolev's embedding theorem, we could choose a small positive constant $\epsilon_1$ satisfying $\mathsf{E}(t)\leq \epsilon_1$ such that 
\begin{equation}\label{bound-density}
\begin{split}
\frac{1}{2}\bar{\rho}\leq \rho(t, \bx)\leq \frac{3}{2}\bar{\rho}.
\end{split}
\end{equation}

The energy estimate
\begin{multline*}
\int_{\mathcal{C}}\left(\frac{1}{2}\rho|\bu|^2+\frac{1}{2}|\bb|^2+\frac{1}{\gamma-1}p(\rho)\right)\ddx+\int_0^t\int_{\mathcal{C}}(S(\nabla_{\bx}\bv):\nabla_{\bx}\bv+|\nabla_{\bx}\times\bb|^2)\ddx\\
=\int_{\mathcal{C}}\left(\frac{1}{2}\rho_0|\bu_0|^2+\frac{1}{2}|\bb_0|^2+\frac{1}{\gamma-1}p(\rho_0)\right)\ddx\\
+\int_0^t\left[\bo\cdot\int_{\partial{\mathcal{C}}}\bx\times \left( T(\bu, p(\rho)) + \bb\otimes\bb - \frac 12 |\bb|^2\mathbb I\right)\cdot \boldsymbol{n}\right]\ddx\\
+\int_0^t\left[\bxi\cdot\int_{\partial{\mathcal{C}}} \left( T(\bu, p(\rho)) + \bb\otimes\bb - \frac 12 |\bb|^2\mathbb I\right)\cdot \boldsymbol{n}\right]\ddx 
\end{multline*}
leads to
\begin{equation}\label{zero-f}
\begin{split}
&\mathsf{F}+\int_0^t\int_{\mathcal{C}}(S(\nabla_{\bx}\bv):\nabla_{\bx}\bv+|\nabla_{\bx}\times\bb|^2)\ddx=\mathsf{F}(0),
\end{split}
\end{equation}
where 
\begin{equation}\label{F}
\begin{split}
\mathsf{F}=:\int_{\mathcal{C}}\left(\frac{1}{2}\rho|\bu|^2+\frac{1}{2}|\bb|^2+\frac{1}{\gamma-1}p(\rho)\right)\ddx+\frac{1}{2}(\bo_{\bu}\cdot \boldsymbol{I}_c\cdot \bo_{\bu}+m_{\mathcal{B}}|\bxi_{\bu}|^2),
\end{split}
\end{equation}
therefore, $\eqref{zero-f}$ gives a uniform in time bound 
\begin{equation}\label{uniform-bound}
\begin{split}
\sup_{0\leq t\leq T}(|\bo_{\bu}|^2+|\bxi_{\bu}|^2)\leq C \mathsf{F}(0).
\end{split}
\end{equation}

\begin{lemma}\label{E}
 For given $\mathsf{E}$ in $\eqref{define-e}$ and $\mathsf{F}$ in $\eqref{F}$,  $t\in [0, T)$ there is a sufficiently small constant $\epsilon_0\in(0, 1)$, such that if 
 \begin{equation}\label{assume}
\begin{split}
\mathsf{E}(0)+\mathsf{F}^{\frac{1}{2}}(0)\leq \epsilon_0, 
 \end{split}
\end{equation}
then one has 
 \begin{equation}\label{result}
\begin{split}
\sup_{0\leq t\leq T}\mathsf{E}(t)\leq \epsilon_0.
 \end{split}
 \end{equation}
\end{lemma}
\begin{proof}
Suppose that $\eqref{result}$ is not true, and assume that 
 \begin{equation*}
\begin{split}
\tilde{t}=\inf \{t\in [0, T], \mathsf{E}(t)> \epsilon_0\}, 
 \end{split}
 \end{equation*}
 then it gives that $\mathsf{E}(\tilde{t})=\epsilon_0$. Then, according to $\eqref{E-D}$, $\eqref{E-D-1}$ and $\eqref{uniform-bound}$, we get 
  \begin{equation*}
\begin{split}
\frac{\mathrm{d}}{\mathrm{d}t}\mathsf{E}(\tilde{t})\leq -C\epsilon_0(1-\epsilon_0-\epsilon_0^2)+C\epsilon_0^2,
  \end{split}
 \end{equation*}
subsequently we use the sufficiently small $\epsilon_0$ to get $\frac{\mathrm{d}}{\mathrm{d}t}\mathsf{E}(\tilde{t})< 0$. It gives a contradiction, so \eqref{result} holds.
 
\end{proof}

Here we use the boot-strap argument to conclude the proof of Theorem $\ref{global-existence-strong-solution}$. Thanks to the local existence result in Theorem $\ref{local-existence}$, there exists a time $T^*>0$, and $(\rho, \bu, \bb)$ satisfies $\eqref{regularity}$. Due to the smallness of initial data $\eqref{assume}$, Lemma $\ref{E}$ implies that  $\eqref{result}$ holds for $T=T^*$, then $\eqref{bound-density}$ gives that 
\begin{equation*}
\begin{split}
\frac{\bar{\rho}}{2}\leq \rho(T^*, \bx)\leq \frac{3}{2}\bar{\rho}, \quad \bx \in \mathcal{C},
  \end{split}
\end{equation*}
and 
\begin{equation*}
\begin{split}
\norm{\bu(T^*)}_{L^2_{(\mathcal{S})}}^2+\norm{\bv(T^*)}_{H^2_{(\mathcal{C})}}+\norm{\varrho(T^*)}_{H^2_{(\mathcal{C})}}+\norm{\bb(T^*)}_{H^2_{(\mathcal{C})}}\leq C,
  \end{split}
\end{equation*}
thus we could apply again the local  existence result in Theorem $\ref{local-existence}$, then we could establish the existence of solution in $[T^*, 2T^*]$. This argument can be repeated in $[0, nT^*]$ for $n\in\mathbb{N}^+$. Eventually, we complete the proof of Theorem $\ref{global-existence-strong-solution}$.

\section{Existence of Weak Solutions to Coupled System}

In this section, we mainly show the existence of the finite energy weak solutions to the coupled system without the hypothesis of the smallness of initial data.

We present the definition of the finite energy weak solution to $\eqref{mhd-x}$-$\eqref{com-x}$.
\begin{definition}\label{weak-solution}
A finite energy weak solution to \eqref{mhd-x}--\eqref{com-x} is $(\rho, \bu, \boldsymbol{\xi_{\bu}},\boldsymbol{\omega}_{\bu}, \bb)$ satisfying
\begin{itemize}
	\item 
	$\rho \geq 0$ in $(0,T)\times\mathcal{S} $, \quad $\rho \in L^\infty(0,T;L^\gamma _{(\mathcal{C})})$,\\
	$\bu \in L^2(0,T;W^{1,2}_{(\mathcal{S})})\cap L^\infty(0,T;L^2_{(\mathcal{S})})$,\quad $\boldsymbol{\xi_{\bu}}$, $\boldsymbol{\omega_{\bu}}$ $\in$ $C([0,T])$,\\
	$\bb\in L^2(0,T;H^1_{(\mathcal{C})})\cap L^\infty(0,T;L^2_{(\mathcal{C})})$, \quad $\mathrm{div}_{\bx}\bb=0$ in $ \mathcal{D}'_{(\mathcal{C})}$;

		\item For any test function $\phi \in C_c^\infty([0,T)\times \overline{\mathcal{C}})$, the density $\rho$ satisfies 
		\begin{equation}\label{weak-density}
		\int_0^T\int_{\mathcal{S}}(\rho\phi_t+\rho\bv\cdot\nabla_{\bx}\phi)\dxdt=-\int_{\mathcal{S}}\rho_0\phi_0\ddx, 
		\end{equation}
furthermore, a renormalized equation of continuity also holds
			\begin{equation*}\label{renormalized-weak-density}
	\int_0^T\int_{\mathcal{S}}(b(\rho)\phi_t+b(\rho)\bv\cdot\nabla_{\bx}\phi-(b'(\rho)\rho-b(\rho))\mathrm{div}_{\bx}\phi)\dxdt=-\int_{\mathcal{S}}b(\rho_0)\phi_0\ddx 
	\end{equation*}
	for any $b \in C^1(\mathbb{R}) $ satisfying $b'(z)=0$ for all $z\in \mathbb{R}^+$ large enough;
			
			\item $\rho\bu \in C_{weak}(0,T;L^{\frac{2\gamma}{\gamma+1}}_{(\mathcal{S})})$ and the momentum equation hold in a weak sense
				\begin{multline}\label{weak-momentum}
			\int_0^T\int_{\mathcal{S}}(-\rho \bu\cdot\boldsymbol{\varphi}_t-\rho\bv\otimes\bu:\nabla_{\bx}\boldsymbol{\varphi}+\rho\boldsymbol{\omega_{\bu}}\times \bu\cdot\boldsymbol{\varphi} -p(\rho)\mathrm{div}_{\bx}\boldsymbol{\varphi}\\
			+S(\nabla_{\bx}\bu):\nabla_{\bx}\boldsymbol{\varphi}-((\nabla_{\bx}\times\bb)\times\bb)\cdot\boldsymbol{\varphi})\dxdt=\int_{\mathcal{S}}\bm{q}\cdot\boldsymbol{\varphi}_0\ddx
		\end{multline}
			for any $\boldsymbol{\varphi} \in C_{c}^\infty([0,T); \mathcal{V}_{(\mathcal{S})})$, and
					\begin{equation*}\label{weak-xi-u}
					m_{\mathcal{B}}\boldsymbol{\xi}_{\bu}(t)=-\int_{\mathcal{C}}\rho\bu\ddx,\quad   \forall t\in [0,T] ;
				\end{equation*}		
				
		\item The induction equation		
					\begin{multline}\label{weak-magnetic}
					\int_0^T\int_{\mathcal{C}}(-\bb\cdot\boldsymbol{\eta}_t+(\boldsymbol{\omega}\times \bb)\cdot\boldsymbol{\eta}-(\boldsymbol{\omega}\times \bx+\boldsymbol{\xi})\cdot\nabla_{\bx}\bb\cdot\boldsymbol{\eta}\\
					-(\bu\times \bb)\cdot (\nabla_{\bx}\times\boldsymbol{\eta})+(\nabla_{\bx}\times \bb):(\nabla_{\bx}\times\boldsymbol{\eta}))\dxdt=\int_{\mathcal{C}}\bb_0\cdot\boldsymbol{\eta}_0\ddx
									\end{multline}		
holds for any $\boldsymbol{\eta}\in C_c^\infty([0,T)\times \overline{\mathcal{C}})$ with $\boldsymbol{\eta}\cdot \bn = 0$ on $\partial \mathcal C$.
%
	\item The energy inequality 
		\begin{multline}\label{energy-inequality}
		\int_{\mathcal{S}}\left(\frac{1}{2}\rho|\bu|^2+\frac{a}{\gamma-1}\rho^\gamma\right)(t)\ddx+\int_{\mathcal{C}}\frac{1}{2}|\bb|^2(t)\ddx+\int_0^t\int_{\mathcal{S}}S(\nabla_{\bx}\bu):\nabla_{\bx}\bu\dxdt\\
		+\int_0^t\int_{\mathcal{C}}|\nabla_{\bx}\times \bb|^2\dxdt
		\leq \int_{\mathcal{S}}\left(\frac{1}{2}\frac{\bm{q}^2}{\rho_0}+\frac{a}{\gamma-1}\rho_0^\gamma\right)\ddx+\int_{\mathcal{C}}\frac{1}{2}|\bb_0|^2\ddx
		\end{multline}		
		holds for almost all $t\in (0,T)$.
	\end{itemize}

\end{definition}

\subsection{Faedo-Galerkin approximation}

We complete the proof of the existence of weak solution to $\eqref{mhd-x}$-$\eqref{com-x}$ by means of a two-level approximation. Precisely, we put an artificial viscosity term to the continuity equation $\eqref{mhd-x}_1$, and add term $\delta\rho^\beta$ to the momentum equation to get higher integrability of the density. 
\begin{equation}\label{mhd-x-re}
\left.
\begin{split}
\rho_t+ \mathrm{div}_{\bx} (\rho\bv)=\epsilon \Delta_{\bx}\rho  &\\
(\rho\bu)_t+\mathrm{div}_{\bx}(\rho\bv\otimes \bu)+\rho\boldsymbol{\omega}\times \bu+\nabla_{\bx}(p(\rho)+\delta\rho^\beta) +\epsilon\nabla_{\bx}\rho\cdot\nabla_{\bx}\bu&\\=\mathrm{div}_{\bx}S(\nabla_{\bx}\bu) +(\nabla_{\bx}\times\bb)\times\bb &\\
\bb_t+(\bo\times \bb)-(\bo\times \bx+\bxi)\cdot\nabla_{\bx}\bb-\nabla_{\bx}\times(\bu\times\bb)=-\nabla_{\bx}\times(\nabla_{\bx}\times\bb), \quad \mathrm{div}_{\bx}\bb=0&
\end{split}
\right\}
\end{equation}
\hspace{9cm}$\forall (t, \bx) \in (0,\infty)\times \mathcal{C},$
\begin{equation}\label{boundary-x-re}
\left.
\begin{split}
\frac{\partial \rho}{\partial \boldsymbol{n}}=0 &\\
\bu={\bo}(t)\times\bx+{\bxi}(t)&\\
\bb\cdot \bn =0,\quad (\nabla_\bx\times \bb)\times \bn = 0&
\end{split}
\right\}\forall(t, \bx) \in (0,\infty)\times\partial\mathcal{C},
\end{equation}

\begin{equation}\label{com-x-re}
\left.
\begin{split}
\boldsymbol{I}_c\cdot \frac{\mathrm{d}}{\mathrm{d}t}\bo+\bo\times(\boldsymbol{I}_c\cdot\bo)=-\int_{\partial\mathcal{C}}\bx\times \left(T(\bu, p(\rho))+\bb\otimes\bb-\frac{1}{2}|\bb|^2\mathbb{I}\right)\cdot \boldsymbol{n}\ddS&\\
m_{\mathcal{B}}\boldsymbol{\xi}(t)=-\int_{\mathcal{C}}\rho\bu\ddx&
\end{split}
\right\}  \forall t\in(0,\infty).
\end{equation}

The reason why we use $\eqref{mhd-x-re}$-$\eqref{com-x-re}$ as an approximation system of $\eqref{mhd-x}$-$\eqref{com-x}$ is that we can utilize $\mathrm{div}_{\bx}\bb=0$ to obtain
\begin{equation*}
\begin{split}
\bu\cdot\nabla_{\bx}\bb+\bb\mathrm{div}_{\bx}\bu-\bb\cdot\nabla_{\bx}\bu&=-\nabla_{\bx}\times(\bu\times\bb),\\
\mathrm{div}_{\bx}\left(\bb\otimes\bb-\frac{1}{2}|\bb|^2\mathbb{I}\right)&=(\nabla_{\bx}\times \bb)\times\bb.
\end{split}
\end{equation*}

Based on the Lemma 3.5 in \cite{GMN-2020-IJNM} and Proposition 4.1 in \cite{HW-2010-ARMA}, we establish the existence of finite energy weak solutions to $\eqref{mhd-x-re}$-$\eqref{com-x-re}$. The result is given as follows.
\begin{lemma}\label{existence of approximate}
	Assume that $\epsilon >0$, $\delta>0$, $\beta>\max\{\gamma, 4\}$ and $\gamma>\frac{3}{2}$, moreover, 
	\begin{equation*}\label{initial-data-approximate}
	\begin{split}
	&\rho_0\in W^{1,\infty}_{(\mathcal{C})}, \hspace{0.2cm} 0<\underline{\rho}\leq \rho \leq \overline{\rho}, \hspace{0.2cm}  \nabla_{\bx}\rho_0\cdot \boldsymbol{n}=0 \hspace{0.2cm} on \hspace{0.1cm} \partial \mathcal{C},\\
	&\boldsymbol{q}=(\rho\bu)_0\in C^2_{(\mathcal{C})},\hspace{0.2cm}  with \int_{S}\bm{q}\ddx=0,\\
	&\bb_0 \in L^{2}_{(\mathcal{C})}, \quad \int_{\mathcal C} \bb_0\cdot\nabla\varphi\ddx = 0\quad \forall \varphi \in W^{1,2}_{(\mathcal C)},\quad \bb_0\cdot\bn = 0.
	\end{split}
	\end{equation*}
	Then there exists  at least one weak solution $(\rho, \bu, \bb)$ to $\eqref{mhd-x-re}$, $\eqref{boundary-x-re}$, $\eqref{com-x-re}$. 

	Furthermore, following items hold.
		\begin{itemize}
			\item	$\rho_t, \Delta_{\bx}\rho \hspace{0.1cm}in \hspace{0.1cm} L^r((0,T)\times\mathcal{C})$, for some $r>1$.  \\
			Equation $\eqref{mhd-x-re}_1$ holds almost everywhere on $(0,T)\times \mathcal{C}$.\\
			$\delta\int_0^T\int_{\mathcal{C}}\rho^{\beta+1}\leq C$, \hspace{0.1cm}$\epsilon \int_0^T\int_{\mathcal{C}}|\nabla_{\bx}\rho|^2\leq C$.
		
			\item $\rho \bu \in C_{weak}(0,T;L^{\frac{2\gamma}{\gamma+1}}_{(\mathcal{S})})$, and  $\eqref{mhd-x-re}_2$, $\eqref{mhd-x-re}_3$, $\mathrm{div}_{\bx}\bb=0$ hold in the sense of distribution.
			
			\item The energy inequality is established
			\begin{multline}\label{energy-inequality-epsilon}
			\int_{\mathcal{S}}\left(\frac{1}{2}\rho|\bu|^2+\frac{a}{\gamma-1}\rho^\gamma+\frac{\delta}{\beta-1}\rho^\beta\right)(t)\ddx+\int_{\mathcal{C}}\frac{1}{2}|\bb|^2(t)\ddx\\
			+\int_0^t\int_{\mathcal{S}}S(\nabla_{\bx}\bu):\nabla_{\bx}\bu\dxdt+\int_0^t\int_{\mathcal{C}}|\nabla_{\bx}\times \bb|^2\dxdt\\
			+\epsilon \int_0^t\int_{\mathcal{C}}(a\gamma\rho^{\gamma-2}+\delta\beta\rho^{\beta-2})|\nabla_{\bx}\rho|^2\dxdt\\
			\leq \int_{\mathcal{S}}\left(\frac{1}{2}\frac{\bq^2}{\rho_0}+\frac{a}{\gamma-1}\rho_0^\gamma+\frac{b}{\beta-1}\rho_0^\beta\right)\ddx+\int_{\mathcal{C}}\frac{1}{2}|\bb_0|^2\ddx.
			\end{multline}			
		\end{itemize}
\end{lemma}
\begin{proof}
Here we provide proof for the sake of completeness. As it does not contain any new ideas, we omit some details -- for more, we refer to \cite{FNP-2001-JMFM} and \cite{HW-2010-ARMA}.
 
The solution is constructed through the Faedo-Galerkin approximations. We investigate an $n-$dimensional space $X_n\subset \left\{\bv\in W^{1,2}_{(\mathcal C)},\, \bv|_{\partial \mathcal C} = 0\right\}$ which is spanned by $n$ eigenvalues of the Dirichlet-Laplace operator. Next, we consider a $6+n$ dimensional space 
$$Y_n = \mathbb{R}^6\oplus X_n$$
equipped with a scalar product
$$
\langle \bm{\varphi},\bm{\psi}\rangle = \int_{\mathcal B} \left(\boldsymbol{\omega}_{\bm\varphi}\times \bx + \boldsymbol{\xi}_{\bm\varphi}\right)\cdot\left(\boldsymbol{\omega}_{\bm\psi}\times \bx + \boldsymbol{\xi}_{\bm\psi}\right)\ {\rm d}\bx + \int_{\mathcal C} \nabla_{\bx} {\bm\varphi}_{\mathcal C}: \nabla_{\bx} \bm{\psi}_{\mathcal C}\ {\rm d}\bx
$$
where we use the fact that every element $\bm{\varphi}$ of $Y_n$ can be decomposed into $\boldsymbol{\omega}_{\bm{\varphi}}\in\mathbb R^3$, $\boldsymbol{\xi}_{\bm{\varphi}}\in \mathbb R^3$ and $\bm{\varphi}_{\mathcal C}\in X_n$.

Let $\boldsymbol u\in C([0,T];Y_n)$ (as usual, we use $\boldsymbol v$ for $\boldsymbol u|_{\mathcal C}$, $\boldsymbol \omega$ for $\boldsymbol{\omega}_{\boldsymbol u}$ and $\boldsymbol \xi$ for $\boldsymbol{ \xi}_{\boldsymbol u}$). There is an operator $\mathcal{S}:C([0,T];X_n)\to C([0,T]; C^{2+\nu}_{\mathcal C})$ such that $\rho = \mathcal{S}(\boldsymbol v)$ solves  \eqref{mhd-x-re}$_1$ equipped with the boundary condition \eqref{boundary-x-re}$_1$ and initial condition $\rho(0,\cdot) = \rho_0(\cdot)$ (see for example \cite[Theorem 10.22 \& 10.23]{FN-2009-Birkhauser}). Similarly, given $\boldsymbol{u}\in C([0,T];Y_n)$ there is a solution $\boldsymbol b = \mathcal{H}(\boldsymbol u)$ to \eqref{mhd-x-re}$_3$ equipped with the boundary condition \eqref{boundary-x-re}$_3$ and with initial condition $\boldsymbol b(0,\cdot) = \boldsymbol b_0(\cdot)$ (here we refer to \cite[Lemma 3.1 \& 3.2]{HW-2010-ARMA}).

Next, we define $\mathcal M_\varrho:Y_n\to Y_n$ as 
$$
\mathcal M_\varrho(\boldsymbol u) = \boldsymbol w \Leftrightarrow \int_{\mathcal S}\varrho\boldsymbol u\cdot \boldsymbol \psi \ddx = \int_{\mathcal S}\boldsymbol w\cdot \boldsymbol\psi\ddx\quad \forall \boldsymbol \psi\in Y_n.
$$

For more about this mapping, we refer to \cite{FNP-2001-JMFM}. Finally, given a function $\boldsymbol u_n\in C([0,T]; Y_n)$ we set a new function $\boldsymbol w_n\in C([0,T]; Y_n)$ defined as follows
\begin{equation*}
\boldsymbol w_n(t)  = \mathcal M^{-1}_{\rho_n(t)} \left(\boldsymbol q + \int_0^t \mathcal N(\rho_n(s), \boldsymbol b_n(s), \boldsymbol u_n(s)\ {\rm d}s\right)
\end{equation*}
where $\boldsymbol q\in Y_n$ such that $\int_{\mathcal S} \boldsymbol q\cdot \boldsymbol \psi\ddx = \int_{\mathcal S} (\rho\boldsymbol u)_0\cdot \boldsymbol \psi \ddx$ for all $\boldsymbol \psi \in Y_n$, $\rho_n = \mathcal{S}(\boldsymbol v_n)$ and $\boldsymbol b_n = \mathcal{H}(\boldsymbol u_n)$ and $\mathcal N$ is a such function that (here and in what follows we use $\boldsymbol \varphi$ also for a function $\boldsymbol \varphi_{\mathcal C} + \boldsymbol \omega_{\boldsymbol \varphi}\times \bx + \boldsymbol \xi_{\boldsymbol \varphi}$ for every $\boldsymbol \varphi \in Y_n$)
\begin{multline*}
\langle\mathcal N(\rho,\boldsymbol b, \boldsymbol u),\boldsymbol \psi\rangle =
\int_{\mathcal C}
\left(\mathrm{div}_{\bx}S(\nabla_{\bx}\bu) +(\nabla_{\bx}\times\bb)\times\bb\right.\\
\left.-
\mathrm{div}_{\bx}(\rho\bv\otimes \bu)-\rho\boldsymbol{\omega}\times \bu-\nabla_{\bx}(p(\rho)-\delta\rho^\beta) -\epsilon\nabla_{\bx}\rho\cdot\nabla_{\bx}\bu\right)\cdot \boldsymbol \psi \ddx.
\end{multline*}

The (above constructed)  mapping $\mathcal T:C([0,T^*];Y_n)\to C([0,T^*];Y_n)$ is a contraction assuming $T^*$ is sufficiently small and, consequently, there is a fixed-point, i.e. a function $\boldsymbol u_n\in C([0,T^*];Y_n )$ such that
\begin{multline}
\label{very.first}
\int_{\mathcal S} \rho_n \boldsymbol u_n\ \cdot \boldsymbol \psi{\rm d}\bx = \int_{\mathcal S} (\rho\boldsymbol u)_0\cdot \boldsymbol \psi\ {\rm d}\bx + 
\int_0^t\int_{\mathcal C}\left(
\mathrm{div}_{\bx}S(\nabla_{\bx}\bu) +(\nabla_{\bx}\times\bb)\times\bb\right.\\
\left.-
\mathrm{div}_{\bx}(\rho\bv\otimes \bu)-\rho\boldsymbol{\omega}\times \bu-\nabla_{\bx}(p(\rho)-\delta\rho^\beta) -\epsilon\nabla_{\bx}\rho\cdot\nabla_{\bx}\bu\right)\cdot \boldsymbol \psi \ {\rm d}\bx\mathrm{d}s,
\end{multline}
for every $\bm{\psi}\in Y_n$. We take the derivative with respect to $t$ of \eqref{very.first} and use $\boldsymbol \psi = \boldsymbol u$ to get (with the help of the continuity)
\begin{multline}\label{ene.equality}
\frac 12 \partial_t 
\int_{\mathcal S}\rho_n|\boldsymbol u_n|^2\ {\rm d}\bx+ \partial_t \int_{\mathcal S} \left(\frac a{\gamma-1}\rho^\gamma + \frac \delta{\beta-1}\rho^\beta\right)\ {\rm d}\bx + \partial_t \int_{\mathcal C} \frac 12 |\boldsymbol b|^2\ {\rm d}\bx+ \int_{\mathcal C} S(\nabla_{\bx} \boldsymbol u):\nabla_{\bx}\boldsymbol u\ {\rm d}\bx\\
+ \int_{\mathcal C} |\nabla_{\bx}\times \boldsymbol b|^2\ {\rm d}\bx+ \epsilon \int_{\mathcal C} (a\gamma\rho^{\gamma-2} + \delta \beta \rho^{\beta-2})|\nabla_{\bx} \rho|^2 \ {\rm d}\bx\leq 0.
\end{multline}
This yields the boundedness of solutions to \eqref{very.first} and, in turn, we get the existence of the solution on the whole time interval $[0,T]$ (recall that all norms on a finitely-dimensional space are equivalent).

It remains to pass with $n\to \infty$. We deduce from the energy equality \eqref{ene.equality} the following set of estimates
\begin{equation*}
\begin{split}
\sup_{t\in [0,T]} \|\rho_n\|_{L^\beta_{(\mathcal C)}}&\leq C,\\
\epsilon \|\nabla_{\bx} \rho_n\|_{L^2(0,T;L^2_{(\mathcal C)})}& \leq C,\\
\sup_{t\in [0,T]} \|\rho_n |\bm{u}|^2\|_{L^1_{(\mathcal S)}} & \leq C,\\
\|\boldsymbol v_n\|_{L^2(0,T;W^{1,2}_{(\mathcal C)})} & \leq C,\\
\|\boldsymbol b_n\|_{L^2(0,T;W^{1,2}_{(\mathcal C)})} & \leq C,\\
\|\boldsymbol b_n\|_{L^\infty(0,T;L^2_{(\mathcal C)})} & \leq C,
\end{split}
\end{equation*}
where $C$ is a constant independent of $n$. Thus, up to a subsequence, 
\begin{equation*}
\begin{split}
\rho_n&\to \rho \ \mbox{weakly}^*\ \mbox{in }L^\infty(0,T;L^\beta_{(\mathcal C)})\\
\boldsymbol v_n&\to \boldsymbol v \ \mbox{weakly in }L^2(0,T;W^{1,2}_{(\mathcal C)})\\
\boldsymbol b_n&\to \boldsymbol b \ \mbox{weakly in }L^2(0,T;W^{1,2}_{(\mathcal C)}).
\end{split}
\end{equation*}

The Aubin-Lions lemma yields $\rho_n\to \rho$ strongly in, say, $L^4(0,T;L^4_{(\mathcal C)})$. This gives
$$
\rho_n\boldsymbol u_n \to \rho \boldsymbol u\ \mbox{weakly}^*\ \mbox{in }L^\infty(0,T,L^{2\beta/(\beta+1)}_{(\mathcal C)}).
$$
Note that this also yields
$$
\boldsymbol \omega_n\to \boldsymbol \omega,\ \boldsymbol \xi_n\to \boldsymbol \xi\ \mbox{in }C([0,T]).
$$
Furthermore, \cite[Lemma 2.4]{FNP-2001-JMFM} yields the existence of $r>1$ and $q>1$ such that
$$
\|\partial_t \rho_n\|_{L^r(0,T;L^r_{(\mathcal C)})} + \|\Delta_{\bx}\rho_n\|_{L^r(0,T;L^r_{(\mathcal C)})} + \|\nabla_{\bx} \rho_n\|_{L^q(0,T;L^2_{(\mathcal C)})}\leq C.
$$
Consequently, 
$$
\partial_t \rho + {\rm div}_{\bx} (\rho \boldsymbol v) = \epsilon \Delta_{\bx}\rho
$$
almost everywhere. 

Next, we test the continuity equations by $\rho_n$, respectively $\rho$ to get
$$
\|\rho_n(t)\|_{L^2_{(\mathcal C)}}^2 + 2\epsilon \int_0^t\|\nabla_{\bx} \rho_n\|_{L^2_{(\mathcal C)}}^2\ {\rm d}s = -\int_0^t\int_\Omega {\rm div}_{\bx} \boldsymbol v_n |\rho_n|^2\ {\rm d}\bx{\rm d}s + \|\rho_0\|_{L^2_{(\mathcal C)}}^2,
$$
and
$$
\|\rho(t)\|_{L^2_{(\mathcal C)}}^2 + 2\epsilon \int_0^t\|\nabla_{\bx} \rho\|_{L^2_{(\mathcal C)}}^2\ {\rm d}s = -\int_0^t\int_\Omega {\rm div}_{\bx} \boldsymbol v |\rho|^2\ {\rm d}\bx{\rm d}s + \|\rho_0\|_{L^2_{(\mathcal C)}}^2.
$$

We deduce that $\nabla_{\bx}\rho_n \to \nabla_{\bx}\rho$ strongly in $L^2(0,T;L^2_{(\mathcal C)})$. Thus, $(\nabla_{\bx}\rho_n\cdot \nabla_{\bx})\boldsymbol u_n\to (\nabla_{\bx}\rho \cdot \nabla_{\bx}) \boldsymbol u $. Next, \eqref{mhd-x-re}$_3$ implies that $\partial_t \boldsymbol b_n\in L^2(0,T;(W^{1,3/2}_{(\mathcal C)})^*)$ with a bound independent of $n$. Therefore (by use of \cite[Lemma 6.2]{NS-2004-Oxford}) $\bb_n\to \bb$ in $C(0,T;L^6_{weak})$ and thus $$\int_0^T\int_{\mathcal C}(\nabla_{\bx}\times \boldsymbol b_n)\times \boldsymbol b_n \cdot \bm{\varphi}\ {\rm d}\bx{\rm d}t \to 
\int_0^T\int_{\mathcal C}(\nabla_{\bx}\times \boldsymbol b)\times \boldsymbol b \cdot \bm{\varphi}\ {\rm d}\bx{\rm d}t.$$

Consequently, the limit functions satisfy \eqref{weak-momentum} and also \eqref{weak-magnetic}. The proof of the energy inequality follows from the weak-lower semi-continuity of the convex functions. 
\end{proof}

\subsection{Artificial viscosity coefficient $\epsilon \rightarrow 0$}

Let $\rho_\epsilon$, $\bu_\epsilon$ and $\bb_\epsilon$ be the solution constructed in the previous section. According to Lemma $\ref{existence of approximate}$, we obtain the following estimates 
\begin{multline}\label{uniform in epsilon}
\sup_{t\in(0,T)}\norm{\sqrt{\rho_{\epsilon}}\bu_{\epsilon}}_{L^2_{(\mathcal{S})}}+\norm{\bu_{\epsilon}}_{L^2(0,T;W^{1,2}_{(\mathcal{S})})}+\norm{\bb_{\epsilon}}_{L^2(0,T:H^1_{(\mathcal{C})})}\\
+\norm{\rho_{\epsilon}}_{L^{\beta}((0,T)\times\mathcal{C})}+\norm{\boldsymbol{\xi}_{\bu_{\epsilon}}}_{L^\infty_{(0,T)}}+\norm{\boldsymbol{\omega}_{\bu_{\epsilon}}}_{L^\infty_{(0,T)}}\leq C,
\end{multline}
where constant $C$ is independent of $\epsilon$. 

We use classical Bogovskii's operator $\mathcal{B}$ defined in \cite{B-1980-TSSLS} to improve the estimates of the density component. Recall $\mathcal B$ is the right inverse to $\rm{div}$ such that $\mathcal{B}:\left\{f \in L^p(\tilde K)\right\}\to W_0^{1,p}(\tilde K)$ is a bounded linear operator for any $1<p<+\infty.$
Moreover, if $f \in L^p(\tilde K)$ can be written in the form $f=\mathrm{div}_x \bf{g}$ for a certain ${\bf{g}} \in L^r(\tilde K), {\bf{g}}\cdot \bm{n}|_{\partial\tilde K}=0$, then it holds that $\|\mathcal{B}[f]\|_{L^r(\tilde K)}\leq C(r)\|{\bf{g}}\|_{L^r(\tilde K)}.$

Choose $\psi(t)\in \mathcal{D}(0,T)$ and set $\overline{m}=\frac{1}{\mathrm{vol}(\mathcal{C})}\int_{\mathcal{C}}\rho_{\epsilon}$. The test function $\psi(t)\mathcal{B}(\rho_{\epsilon}-\overline{m})$ in $\eqref{mhd-x-re}_2$ yields
\begin{equation}\label{high-estimate-pressure}
\begin{split}
&\int_0^T\int_{\mathcal{C}}\psi(a\rho_{\epsilon}^{\gamma+1}+\delta\rho_{\epsilon}^{\beta+1})\dxdt\\
&= \overline{m}\int_0^T\int_{\mathcal{C}}\psi(a\rho_{\epsilon}^\gamma+\delta\rho_{\epsilon}^\beta)\dxdt+\int_0^T\int_{\mathcal{C}}\psi_t\rho_{\epsilon}\bu_{\epsilon}\cdot \mathcal{B}(\rho_{\epsilon}-\overline{m})\dxdt\\
&\hspace{0.5cm}+\int_0^T\int_{\mathcal{C}}\psi\rho_{\epsilon}\bu_{\epsilon}\cdot \mathcal{B}(\mathrm{div}_{\bx}(\rho_{\epsilon}\bu_{\epsilon}))\dxdt-\epsilon \int_0^T\int_{\mathcal{C}}\psi\rho_{\epsilon}\bu_{\epsilon}\cdot\mathcal{B}(\Delta_{\bx}\rho_{\epsilon})\dxdt\\
&\hspace{0.5cm}+\int_0^T\int_{\mathcal{C}}\psi\rho_{\epsilon}\bv_{\epsilon}\otimes \bu_{\epsilon}:\nabla_{\bx}\mathcal{B}(\rho_{\epsilon}-\overline{m})\dxdt+\epsilon \int_0^T\int_{\mathcal{C}}\psi\nabla_{\bx}\rho_{\epsilon}\cdot\nabla_{\bx}\bu_{\epsilon}\cdot \mathcal{B}(\rho_{\epsilon}-\overline{m})\dxdt\\
&\hspace{0.5cm}+\int_0^T\int_{\mathcal{C}}\psi \rho_{\epsilon}\boldsymbol{\omega}_{\bu_{\epsilon}}\times \bu_{\epsilon}\cdot \mathcal{B}(\rho_{\epsilon}-\overline{m})\dxdt+\int_0^T\int_{\mathcal{C}}\psi(\lambda-\frac{2}{3}\mu)\mathrm{div}_{\bx}\bu_{\epsilon}(\rho_{\epsilon}-\overline{m})\dxdt\\
&\hspace{0.5cm}+\int_0^T\int_{\mathcal{C}}2\mu \mathcal{D}(\bu_{\epsilon}):\mathcal{D}(\mathcal{B}(\rho_{\epsilon}-\overline{m}))\dxdt+\int_0^T\int_{\mathcal{C}}\psi((\nabla_{\bx}\times\bb_{\epsilon})\times\bb_{\epsilon})\cdot\mathcal{B}(\rho_{\epsilon}-\overline{m})\dxdt\\
&=:\sum_{i=1}^{10}\mathcal{J}_{1,i}.
\end{split}
\end{equation}
We estimate $\mathcal{J}_{1,i}$ with $i=1,...,10$ term by term.

The properties of Bogovskii's operator, $\eqref{uniform in epsilon}$ and H\"older's inequality lead to 
\begin{equation*}
\begin{split}
&|\mathcal{J}_{1,2}|\leq \int_0^T|\psi_t|\norm{\sqrt{\rho_{\epsilon}}}_{L^2_{(\mathcal{C})}}\norm{\sqrt{\rho_{\epsilon}}\bu_{\epsilon}}_{L^2_{(\mathcal{C})}}\norm{\mathcal{B}(\rho_{\epsilon}-\overline{m})}_{L^\infty_{(\mathcal{C})}}\dt\leq C\int_0^T|\psi_t|\dt\leq C,\\
&|\mathcal{J}_{1,3}|\leq \int_0^T\norm{\rho_{\epsilon}}^2_{L^3_{(\mathcal{C})}}\norm{\bu_{\epsilon}}^2_{L^6_{(\mathcal{C})}}\dt\leq C,\\
&|\mathcal{J}_{1,4}|\leq \epsilon \int_0^T\norm{\rho_{\epsilon}}_{L^3_{(\mathcal{C})}}\norm{\bu_{\epsilon}}_{L^6_{(\mathcal{C})}}\norm{\nabla_{\bx}\rho}_{L^2_{(\mathcal{C})}}\dt\leq C,\\
&|\mathcal{J}_{1,5}|
 \leq \int_0^T\norm{\rho_{\epsilon}}^2_{L^3_{(\mathcal{C})}}\norm{\bu_{\epsilon}}_{L^6_{(\mathcal{C})}}\norm{\bv_{\epsilon}}_{L^6_{(\mathcal{C})}}\dt \leq C,\\
&|\mathcal{J}_{1,6}|\leq C\sqrt{\epsilon}\sqrt{\epsilon}\norm{\nabla_{\bx}\rho_{\epsilon}}_{L^2((0,T)\times \mathcal{C})}\norm{\bu_{\epsilon}}_{L^2((0,T)\times \mathcal{C})}\norm{\mathcal{B}(\rho_{\epsilon}-\overline{m})}_{L^\infty((0,T)\times \mathcal{C})}\leq C\sqrt{\epsilon},\\
&|\mathcal{J}_{1,7}| \leq C\norm{\boldsymbol{\omega}_{\bu_\epsilon}}_{L^\infty_{(0,T)}}\norm{\rho_{\epsilon}}_{L^2((0,T)\times\mathcal{C})}\norm{\bu_{\epsilon}}_{L^2((0,T)\times\mathcal{C})}\norm{\mathcal{B}(\rho_{\epsilon}-\overline{m})}_{L^\infty((0,T)\times\mathcal{C})}\leq C,\\
&|\mathcal{J}_{1,8}+\mathcal{J}_{1,9}| \leq C\norm{\nabla_{\bx}\bu_{\epsilon}}_{L^2((0,T)\times\mathcal{C})}(\norm{\rho_{\epsilon}}_{L^2((0,T)\times\mathcal{C})}+\norm{\mathcal{B}(\rho_{\epsilon}-\overline{m})}_{L^2(0,T;H^1_{(\mathcal{C})})} \leq C, \\
&|\mathcal{J}_{1,10}|\leq C\norm{\nabla_{\bx}\times\bb_{\epsilon}}_{L^2((0,T)\times\mathcal{C})}\norm{\bb_{\epsilon}}_{L^\infty(0,T;L^2_{(\mathcal{C})})}\norm{\mathcal{B}(\rho_{\epsilon}-\overline{m})}_{L^\infty((0,T)\times\mathcal{C})}\leq C.
\end{split}
\end{equation*}
We plug these estimates into $\eqref{high-estimate-pressure}$ to get 
\begin{equation}\label{high-rho}
\begin{split}
\norm{\rho_{\epsilon}}_{L^{\gamma+1}((0,T)\times\mathcal{C})}+\norm{\rho_{\epsilon}}_{L^{\beta+1}((0,T)\times\mathcal{C})}\leq C,
\end{split}
\end{equation}
where $C$ is independent of $\epsilon$.

By $\epsilon\norm{\nabla_{\bx}\rho_{\epsilon}}^2_{L^2((0,T)\times\mathcal{C})}\leq C$ in Lemma $\ref{existence of approximate}$, $\eqref{energy-inequality-epsilon}$, and $\eqref{high-rho}$, we deduce that 
\begin{equation*}
\begin{split}
&\epsilon \nabla_{\bx}{\rho_{\epsilon}}\cdot \nabla_{\bx}\bu_{\epsilon}\rightarrow 0 \ \mbox{in }L^1{((0,T)\times \mathcal{C})},\\
&\epsilon \Delta_{\bx}\rho_{\epsilon}\rightarrow 0 \hspace{0.1cm} in  \hspace{0.1cm} L^2(0,T;H^{-1}_{(\mathcal{C})}),\\
&\rho_{\epsilon}\rightarrow \rho \hspace{0.1cm}in  \hspace{0.1cm} C_{weak}(0,T;L^\beta_{(\mathcal{C})}),\\
&\bu_{\epsilon}\rightarrow \bu  \hspace{0.1cm} weakly  \ \mbox{in }L^2(0,T;H^1_{(\mathcal{C})}),\\
&\bb_{\epsilon}\rightarrow \bb  \hspace{0.1cm} weakly  \ \mbox{in }L^2(0,T;H^1_{(\mathcal{C})})\hspace{0.1cm}with\hspace{0.1cm} \mathrm{div}_{\bx}\bb=0,\\
&\bb_{\epsilon}\rightarrow \bb  \hspace{0.1cm} strongly  \ \mbox{in }L^2(0,T;L^2_{(\mathcal{C})}),\\
&p(\rho_{\epsilon})+\delta \rho_{\epsilon}^\beta\rightarrow \overline{p(\rho)+\delta \rho^\beta}\hspace{0.1cm} weakly \ \mbox{in } L^{\frac{\beta+1}{\beta}}((0,T)\times \mathcal{C}).
\end{split}
\end{equation*}

 By similar procedures as in  \cite[Section 7.9]{NS-2004-Oxford} and  \cite[Section 3.2]{GMN-2020-IJNM}, we obtain the equi-continuity of $\rho_{\epsilon}\bu_{\epsilon} \in W^{-1,\frac{\beta}{\beta+1}}_{(\mathcal{S})}$, and
\begin{equation*}
\begin{split}
\rho_{\epsilon}\bu_{\epsilon}\rightarrow \rho\bu \ \mbox{in } C_{weak}(0,T;L^{\frac{2\beta}{\beta+1}}_{\mathcal{(S)}}).
\end{split}
\end{equation*}
Consequently, we derive that 
\begin{equation*}
\begin{split}
&\boldsymbol{\omega}_{\bu_{\epsilon}}\rightarrow \boldsymbol{\omega}_{\bu}\hspace{0.1cm}strongly \ \mbox{in }L^\infty(0,T),\\
&\boldsymbol{\xi}_{\bu_{\epsilon}}\rightarrow \boldsymbol{\xi}_{\bu}\hspace{0.1cm}strongly \ \mbox{in }L^\infty(0,T).
\end{split}
\end{equation*}
Moreover, 
\begin{equation*}
\begin{split}
\rho_{\epsilon}\bu_{\epsilon}\rightarrow \rho\bu \hspace{0.1cm}weakly^*\hspace{0.1cm}in \hspace{0.1cm}L^\infty(0,T;L^{\frac{2\beta}{\beta+1}}_{(\mathcal{S})}).
\end{split}
\end{equation*}
It holds that 
\begin{equation*}
\begin{split}
\rho_{\epsilon}\bu_{\epsilon}\rightarrow \rho\bu \hspace{0.1cm}strongly  \hspace{0.1cm}in  \hspace{0.1cm}L^p(0,T;W^{-1,2}_{(\mathcal{S})}), \hspace{0.1cm} for \hspace{0.1cm} p\geq 1,
\end{split}
\end{equation*}
thus, 
\begin{equation*}
\begin{split}
\rho_{\epsilon}\bv_{\epsilon}\otimes\bu_{\epsilon}\rightarrow \rho\bv\otimes\bu \hspace{0.1cm}in  \hspace{0.1cm} (C(0,T;\mathcal{V}_{(\mathcal{S})}))^*.
\end{split}
\end{equation*}
Also 
\begin{equation*}
\begin{split}
(\nabla_{\bx}\times\bb_{\epsilon})\times\bb_{\epsilon}\rightarrow (\nabla_{\bx}\times\bb)\times\bb\ \mbox{in } \mathcal{D}'(0,T;\mathcal{C}), 
\end{split}
\end{equation*}
and
\begin{equation*}
\begin{split}
\nabla_{\bx}\times(\bu_{\epsilon}\times\bb_{\epsilon})\rightarrow\nabla_{\bx}\times(\bu\times\bb)\ \mbox{in } \mathcal{D}'(0,T;\mathcal{C}).
\end{split}
\end{equation*}

Therefore, we obtain that $(\rho, \bu, \bb)$ satisfies the following system
\begin{equation*}\label{mhd-x-re-viscosity}
\left.
\begin{split}
\rho_t+ \mathrm{div}_{\bx} (\rho\bv)=0, \hspace{0.1cm} in\hspace{0.1cm} \mathcal{C}&\\
(\rho\bu)_t+\mathrm{div}_{\bx}(\rho\bv\otimes \bu)+\rho\boldsymbol{\omega}\times \bu+\nabla_{\bx}\overline{(p(\rho)+\delta\rho^\beta)} &\\=\mathrm{div}_{\bx}S(\nabla_{\bx}\bu) +(\nabla_{\bx}\times\bb)\times\bb,\hspace{0.1cm}in\hspace{0.1cm} \mathcal{C} &\\
\bb_t+(\bo\times \bb)-(\bo\times \bx+\bxi)\cdot\nabla_{\bx}\bb-\nabla_{\bx}\times(\bu\times\bb)=-\nabla_{\bx}\times(\nabla_{\bx}\times\bb), \quad \mathrm{div}_{\bx}\bb=0, \hspace{0.1cm} in\hspace{0.1cm} \mathcal{C}&\\
\bu=\boldsymbol{\omega_{\bu}}\times \bx +\boldsymbol{\xi_{\bu}}, \hspace{0.1cm}on \hspace{0.1cm}\partial{\mathcal{C}}&\\
\bb\cdot \bn = 0,\quad (\nabla_\bx\times \bb)\times \bn = 0 \hspace{0.1cm}on \hspace{0.1cm}\partial{\mathcal{C}}&
\end{split}
\right\}
\end{equation*}
\begin{equation*}\label{com-x-re-1}
\left.
\begin{split}
\boldsymbol{I}_c\cdot \frac{\mathrm{d}}{\mathrm{d}t}\bo+\bo\times(\boldsymbol{I}_c\cdot\bo)=-\int_{\partial\mathcal{C}}\bx\times \left(T(\bu, p(\rho)) + \bb\otimes\bb - \frac 12 |\bb|^2\mathbb I\right)\cdot \boldsymbol{n}\ddS&\\
m_{\mathcal{B}}\boldsymbol{\xi}(t)=-\int_{\mathcal{C}}\rho\bu\ddx&
\end{split}
\right\}  \forall t\in(0,\infty).
\end{equation*}

The above convergences allows us to deduce the convergence of the effective viscous flux. In particular \cite[Lemma 3.2]{FNP-2001-JMFM}
\begin{lemma}
It holds that
$$
\int_0^T\int_{\mathcal C} (p(\rho_\epsilon) + \delta \rho_\epsilon^\beta - (\lambda + 2\mu)\mathrm{div}_x \bv_\epsilon)\rho_\epsilon \varphi\dxdt \to \int_0^T\int_{\mathcal C} \left(\overline{p(\rho) + \delta\rho^\beta} - (\lambda + 2\mu)\mathrm{div}_x \bv\right)\rho\varphi\dxdt
$$
for all $\varphi \in C^\infty_c((0,T)\times \mathcal C)$. 
\end{lemma}
Since $\rho_\epsilon$, $\bv_\epsilon$, $\rho$, and $\bv$ are regular enough, they fulfill the renormalized continuity equation, we then take $b(\rho) = \rho\log\rho$ in order to obtain
$$
\int_0^\tau \int_{\mathcal C} \rho_\epsilon \mathrm{div}_x \bv_\epsilon \dxdt \leq \int_{\mathcal C} \rho_0\log \rho_0\ddx - \int_{\mathcal C} \rho_\epsilon(\tau,\cdot)\log \rho_\epsilon(\tau,\cdot)\ddx,
$$
and
$$
\int_0^\tau \int_{\mathcal C} \rho \mathrm{div}_x \bv \dxdt = \int_{\mathcal C} \rho_0\log \rho_0\ddx - \int_{\mathcal C} \rho(\tau,\cdot)\log \rho(\tau,\cdot)\ddx.
$$

Let $\Phi_n$ be a sequence of smooth compactly supported positive functions such that $\Phi_n\to 1$ in $L^q$ for some sufficiently large $q$. We have 
\begin{multline*}
\int_{\mathcal C}[\rho(\tau,\cdot)\log \rho(\tau,\cdot) - \rho_\epsilon (\tau,\cdot)\log \rho_\epsilon(\tau,\cdot)]\ddx\geq \int_0^\tau \int_{\mathcal C}(\rho_\epsilon \mathrm{div}_{\bx} \bv_\epsilon - \rho \mathrm{div}_{\bx}\bv )\dxdt \\
 = \int_0^\tau \Phi_n\left(\rho_\epsilon \mathrm{div}_{\bx} \bv_\epsilon - \frac 1{\lambda + 2\mu}(p(\rho_\epsilon) + \delta \rho_\epsilon^\beta)\rho_\epsilon + \frac 1{\lambda + 2\mu}(p(\rho_\epsilon) + \delta \rho_\epsilon^\beta)\rho_\epsilon - \rho \mathrm{div}_{\bx}\bv\right)\dxdt\\
 + \int_0^\tau \int_{\mathcal C} (1-\Phi_n)(\rho_\epsilon \mathrm{div}_{\bx} \bv_\epsilon - \rho \mathrm{div}_{\bx} \bv)\dxdt,
\end{multline*}
and we use the effective viscous flux lemma to pass to a  limit
\begin{multline*}
\int_{\mathcal C} \rho(\tau,\cdot)\log \rho(\tau,\cdot) - \overline{\rho(\tau,\cdot)\log \rho(\tau,\cdot)}\ddx\\
\geq \frac 1{\lambda + 2\mu}\int_0^\tau\int_{\mathcal C} \left(\overline{(p(\rho) - \delta \rho^\beta)\rho} - \overline{(p(\rho) - \delta \rho^\beta)}\rho\right)\Phi_n\dxdt+ O(\|1-\Phi\|_q)\\
\geq 0 + O(\|1-\Phi\|_q),
\end{multline*}
where the last inequality holds due to the monotonicity of the mapping $\rho\mapsto p(\rho) + \delta \rho^\beta$. 
Thus,
$$
\int_{\mathcal C} \rho(\tau,\cdot)\log \rho(\tau,\cdot) - \overline{\rho(\tau,\cdot)\log \rho(\tau,\cdot)}\ddx\geq 0,
$$
and since $\rho\mapsto \rho \log \rho$ is convex, we get $\rho_\epsilon \to \rho$ almost everywhere in $(0,T)\times \mathcal C$.

To sum up, we get the following proposition.
\begin{proposition}\label{prop-delta}
	Suppose that all conditions in Lemma $\ref{existence of approximate}$ are satisfied, and $\beta>\max\{4, \frac{6\gamma}{2\gamma-3}\}$, then there exists let least one weak solution $(\rho, \bu, \bb)$ to the problem
	\begin{equation}\label{mhd-x-re-viscosity-1}
	\left.
	\begin{split}
	\rho_t+ \mathrm{div}_{\bx} (\rho\bv)=0, \hspace{0.1cm} in\hspace{0.1cm} \mathcal{C}&\\
	(\rho\bu)_t+\mathrm{div}_{\bx}(\rho\bv\otimes \bu)+\rho\boldsymbol{\omega}\times \bu+\nabla_{\bx}(p(\rho)+\delta\rho^\beta)&\\=\mathrm{div}_{\bx}S(\nabla_{\bx}\bu) +(\nabla_{\bx}\times\bb)\times\bb,\hspace{0.1cm}in\hspace{0.1cm} \mathcal{C} &\\
	\bb_t+(\bo\times \bb)-(\bo\times \bx+\bxi)\cdot\nabla_{\bx}\bb-\nabla_{\bx}\times(\bu\times\bb)=-\nabla_{\bx}\times(\nabla_{\bx}\times\bb), \quad \mathrm{div}_{\bx}\bb=0, \hspace{0.1cm} in\hspace{0.1cm} \mathcal{C}&\\
	\bu=\boldsymbol{\omega_{\bu}}\times \bx +\boldsymbol{\xi_{\bu}}, \hspace{0.1cm}on \hspace{0.1cm}\partial{\mathcal{C}}&\\
\bb\cdot \bn = 0,\quad (\nabla_\bx\times \bb)\times \bn = 0 \hspace{0.1cm}on \hspace{0.1cm}\partial{\mathcal{C}}&
	\end{split}
	\right\}
	\end{equation} 
	\begin{equation}\label{com-x-re-2}
	\left.
	\begin{split}
	\boldsymbol{I}_c\cdot \frac{\mathrm{d}}{\mathrm{d}t}\bo+\bo\times(\boldsymbol{I}_c\cdot\bo)=-\int_{\partial\mathcal{C}}\bx\times \left(T(\bu, p(\rho)) + \bb\otimes\bb - \frac 12 |\bb|^2\mathbb I\right)\cdot \boldsymbol{n}\ddS&\\
	m_{\mathcal{B}}\boldsymbol{\xi}(t)=-\int_{\mathcal{C}}\rho\bu\ddx&
	\end{split}
	\right\}  \forall t\in(0,\infty).
	\end{equation}

	Furthermore, $\rho\in L^{\beta+1}((0,T)\times\mathcal{C})$, and the equation $\eqref{mhd-x-re-viscosity-1}_1$ holds in the sense of renormalized solutions. The triple $(\rho,\bu,\bb)$ also satisfies the energy inequality 
\begin{multline}\label{regularity-delta}
\int_{\mathcal{S}}\left(\frac{1}{2}\rho|\bu|^2+\frac{a}{\gamma-1}\rho^\gamma+\frac{\delta}{\beta-1}\rho^\beta\right)(t)\ddx+\int_{\mathcal{C}}\frac{1}{2}|\bb|^2(t)\ddx\\
			+\int_0^t\int_{\mathcal{S}}S(\nabla_{\bx}\bu):\nabla_{\bx}\bu\dxdt+\int_0^t\int_{\mathcal{C}}|\nabla_{\bx}\times \bb|^2\dxdt\\
			\leq \int_{\mathcal{S}}\left(\frac{1}{2}\frac{\bq^2}{\rho_0}+\frac{a}{\gamma-1}\rho_0^\gamma+\frac{b}{\beta-1}\rho_0^\beta\right)\ddx+\int_{\mathcal{C}}\frac{1}{2}|\bb_0|^2\ddx.
\end{multline}
for almost all $t\in (0,T)$.
	\end{proposition}

\subsection{Artificial pressure viscosity $\delta \rightarrow 0$}
 %
We prescribe initial data $(\rho_{0,\delta},q_{\delta}, \bb_{0,\delta})$ in the same manner as in \cite[Section 4]{FNP-2001-JMFM}, in particular we assume that
\begin{equation}\label{initial-regular-data}
\begin{split}
\rho_{0,\delta}\rightarrow \rho_0\ \mbox{in } L^\gamma_{(\mathcal{C})},\  \bm{q}_{\delta}\rightarrow \bm{q}\ \mbox{in }L^1_{(\mathcal{C})},\  \bb_{0,\delta}\rightarrow\bb_0\ \mbox{in }L^2_{(\mathcal{S})}\quad \mbox{ as }\delta \to 0. 
\end{split}
\end{equation}
Then we consider the sequence of approximate solutions $(\rho_{\delta},\bb_{\delta}, \bb_{\delta})$ to the system $\eqref{mhd-x-re-viscosity-1}$-$\eqref{com-x-re-2}$  with the initial data $(\rho_{0,\delta},q_{\delta}, \bb_{0,\delta})$, its solvability is established in the previous section.  

Similarly to the Lemma 4.1 in \cite{FNP-2001-JMFM}, we again use Bogovskii's operator to get a higher integrability of the density. 
There exists $\theta>0$, such that 
\begin{equation}\label{high-density}
\begin{split}
\int_0^T\int_{\mathcal{C}}\left(a\rho_{\delta}^{\gamma+\theta}+\delta \rho_{\delta}^{\beta+\theta}\right)\dxdt\leq C,
\end{split}
\end{equation}
where constant $C$ is independent of $\delta$.

By virtue of  Proposition $\ref{prop-delta}$
\begin{equation}\label{limit-delta}
\begin{split}
&\rho_{\delta}\rightarrow \rho \ \mbox{in } C_{weak}(0,T;L^\gamma_{(\mathcal{C})}),\\
&\bu_{\delta}\rightarrow \bu \ \mbox{weakly} \ \mbox{in } L^2(0,T;W^{1,2}_{(\mathcal{S})}),\\
&\bb_{\delta}\rightarrow \bb \ \mbox{weakly}^* \ \mbox{in }L^2(0,T;H^1_{(\mathcal{C})})\cap L^\infty(0,T;L^2_{(\mathcal{C})}),\\
& \mathrm{div}_{\bx}\bb=0   \ \mbox{in } \mathcal{D}'((0,T)\times \mathcal{C}),
\end{split}
\end{equation}
and, $\eqref{high-density}$ gives that 
\begin{equation*}
\begin{split}
&\rho_{\delta}^{\gamma}\rightarrow \overline{\rho^\gamma} \ \mbox{weakly} \ \mbox{in } L^{\frac{\gamma+\theta}{\gamma}}((0,T)\times \mathcal{C}),\\
&\delta \rho^\delta\rightarrow 0  \ \mbox{in } L^1((0,T)\times \mathcal{C}).
\end{split}
\end{equation*}

By means of the momentum equations in $\eqref{mhd-x-re-viscosity-1}$ and $\eqref{regularity-delta}$, we get 
\begin{equation*}
\begin{split}
\rho_{\delta}\bu_{\delta}\rightarrow \rho \bu  \ \mbox{in }  C_{weak}(0,T;L^\frac{2\gamma}{\gamma+1}_{(\mathcal{S})}), 
\end{split}
\end{equation*}
and 
\begin{equation*}
\begin{split}
\boldsymbol{\omega}_{\bu_{\delta}}\rightarrow \boldsymbol{\omega}_{\bu},\hspace{0.1cm} \boldsymbol{\xi}_{\bu_{\delta}}\rightarrow \boldsymbol{\xi}_{\bu}, \ \mbox{in }  L^\infty(0,T).
\end{split}
\end{equation*}
Furthermore, 
\begin{equation*}
\begin{split}
\rho_{\delta}\bv_{\delta}\times\bu_{\delta}\rightarrow \rho \bv \times \bu \ \mbox{in } (C((0,T)\times \mathcal{V}_{(\mathcal{S})}))^*.
\end{split}
\end{equation*}
Similarly, it also holds that 
\begin{equation*}
\begin{split}
\bb_{\delta}\rightarrow \bb \ \mbox{in }  C_{weak}(0,T;L^2_{(\mathcal{C})}).
\end{split}
\end{equation*}

Additionally, the limits $(\rho, \bu, \bb )$ satisfy the initial conditions of $\eqref{initial-regular-data}$ in the weak sense. 
By $\eqref{limit-delta}_2$, $\eqref{limit-delta}_3$, $\eqref{limit-delta}_4$ and the fact $H^1_{(\mathcal{C})}\hookrightarrow L^2_{(\mathcal{C})}$, we derive that 
\begin{equation*}
\begin{split}
&\nabla_{\bx}\times (\bu_{\delta}\times \bb_{\delta})\rightarrow \nabla_{\bx}\times (\bu\times \bb) \ \mbox{in } \mathcal{D}'((0,T)\times \mathcal{C}),\\
&(\nabla_{\bx}\times \bb_{\delta })\times {\bb_{\delta}}\rightarrow (\nabla_{\bx}\times \bb)\times {\bb}
\ \mbox{in } \mathcal{D}'((0,T)\times \mathcal{C}).
\end{split}
\end{equation*}

As a result, we use the above convergences to derive that $(\rho, \bu, \bb)$ satisfies 
\begin{equation*}\label{mhd-x-re-delta}
\left.
\begin{split}
\rho_t+ \mathrm{div}_{\bx} (\rho\bv)=0, \hspace{0.1cm} in\hspace{0.1cm} \mathcal{C}&\\
(\rho\bu)_t+\mathrm{div}_{\bx}(\rho\bv\otimes \bu)+\rho\boldsymbol{\omega}\times \bu+\nabla_{\bx}\overline{p(\rho)} &\\=\mathrm{div}_{\bx}S(\nabla_{\bx}\bu) +(\nabla_{\bx}\times\bb)\times\bb,\hspace{0.1cm}in\hspace{0.1cm} \mathcal{C} &\\
\bb_t+(\bo\times \bb)-(\bo\times \bx+\bxi)\cdot\nabla_{\bx}\bb-\nabla_{\bx}\times(\bu\times\bb)=-\nabla_{\bx}\times(\nabla_{\bx}\times\bb), \quad \mathrm{div}_{\bx}\bb=0, \hspace{0.1cm} in\hspace{0.1cm} \mathcal{C}&\\
\bu=\boldsymbol{\omega_{\bu}}\times \bx +\boldsymbol{\xi_{\bu}} \hspace{0.1cm}on \hspace{0.1cm}\partial{\mathcal{C}}&\\
\bb\cdot \bn = 0,\quad (\nabla_\bx \times \bb)\times \bn = 0 \hspace{0.1cm}on \hspace{0.1cm}\partial{\mathcal{C}}&
\end{split}
\right\}.
\end{equation*}

 The only thing left is to show  $a\rho^\gamma=a \overline{\rho^\gamma}$. We first define a family of smooth concave functions as follows
 $$
 T_k(z) = \left\{
 \begin{array}{r} z\hspace{0.3cm}\mbox{for} \hspace{0.1cm}z\leq k,\\
 2k\hspace{0.3cm} \mbox{for}\hspace{0.1cm} z\geq 3k.
 \end{array}
 \right.
 $$
 Since 
 $$ 
 \sup_{k\geq 1} \left(\lim \sup_{\delta \to 0} \int_Q |T_k(\rho_\delta) - T_k(\rho)|^q\dxdt\right) <\infty,
 $$
 the functions $\rho,\bu$ satisfy the renormalized continuity equation according to \cite[Lemma 3.8]{FN-2009-Birkhauser}. Therefore
$$
\partial_t L_k(\rho_\delta) + \mathrm{div}_{\bx}(L_k(\rho_\delta)\bv_\delta) + T_k(\rho_\delta)\mathrm{div}_{\bx}(\bv_\delta) = 0,
$$
and
$$
\partial_t L_k(\rho) + \mathrm{div}_{\bx}(L_k(\rho)\bv) + T_k(\rho)\mathrm{div}_{\bx}(\bv) = 0,
$$
where
$$L_k(\rho) = \left\{\begin{array}{r}
\rho \log\rho\quad\mbox{for }0\leq \rho< k, \\
\rho\log k + \rho\int_k^\rho\frac{1}{s^2}T_k(s)\ {\rm d}s\quad\mbox{for }\rho\geq  k.
\end{array}
\right.
$$
Thus, we deduce that
\begin{multline*}
\int_{\mathcal C}[L_k(\rho(\tau,\cdot)) - L_k(\rho_\delta(\tau,\cdot))]\ {\rm d}\bx \geq \int_0^\tau\int_{\mathcal C} [T_k(\rho_\delta)\mathrm{div}_{\bx} \bv_\delta - T_k(\rho) \mathrm{div}_{\bx} \bv]\ {\rm d}\bx{\rm d}t\\
 = \int_0^\tau\int_{\mathcal C}\Phi_n \left(T_k(\rho_\delta)\mathrm{div}_\bx \bv_\delta - \frac{1}{\lambda + 2\mu}p(\rho_\delta)T_k(\rho_\delta) + \frac{1}{\lambda + 2\mu}p(\rho_\delta)T_k(\rho_\delta) - \rho\mathrm{div}_\bx\bv\right)\ {\rm d}\bx{\rm d}t \\
  + \int_0^\tau \int_{\mathcal C} (1-\Phi_n)(T_k(\rho_\delta)\mathrm{div}_{\bx} \bv_\delta - T_k(\rho)\mathrm{div}_{\bx} \bv)\ {\rm d}\bx{\rm d}t,
\end{multline*}
and we use the monotonicity of pressure and the effective viscous flux lemma (see for example \cite[Lemma 4.2]{FNP-2001-JMFM}) to deduce that
$$
\int_{\mathcal C}[L_k(\rho(\tau,\cdot)) - L_k(\rho_\delta(\tau,\cdot))]\ddx \geq 0.
$$

We send $k\to \infty$ and the convexity of $\rho\log\rho$ yields
$$
\rho_\delta\to \rho \quad \mbox{almost everywhere in }(0,T)\times \mathcal C.
$$

Eventually, this completes the proof of Theorem $\ref{weak-solution-existence}$.

\section {Weak-strong Uniqueness Property}
In this section, we aim to show that the strong solution established in Theorem $\ref{global-existence-strong-solution}$ coincides with the weak solution constructed in Theorem $\ref{weak-solution-existence}$. 

Firstly, we construct the relative energy functional 
\begin{equation*}\label{relative-entropy-functional}
\begin{split}
\mathcal{E}(\rho, \bu, \bb|r, \bU,\bB ):=\frac{1}{2}\rho|\bu-\bU|^2+\frac{1}{\gamma-1}(p(\rho)-p'(r)(\rho-r)-p(r))+\frac{1}{2}|\bb-\bB|^2, 
\end{split}
\end{equation*}
where $(\rho,\bu, \bb)$ is a weak solution to $\eqref{mhd-x}$-$\eqref{com-x}$ and $(r, \bU,\bB)$ are sufficiently smooth functions.
Then the following lemma is established.
\begin{lemma}
Suppose  that all assumptions in Theorem $\ref{weak-solution-existence}$ are satisfied, and $(\rho, \bu, \boldsymbol{\omega}_{\bu}, \boldsymbol{\xi}_{\bu}, \bb)$ is a weak solution to $\eqref{mhd-x}$-$\eqref{com-x}$. Then the following relative energy inequality
\begin{multline}\label{relative-entropy-inequality}
\left[\int_{\mathcal{S}}\left(\frac{1}{2}\rho|\bu-\bU|^2+\frac{1}{\gamma-1}(p(\rho)-p'(r)(\rho-r)-p(r))\right)(s)\ddx+\int_{\mathcal{C}}\frac{1}{2}|\bb-\bB|^2(s)\ddx\right]_{s=0}^{s=t}\\
+\int_0^t\int_\mathcal{S}(S(\nabla_{\bx}\bu)-S(\nabla_{\bx}\bU)):(\nabla_{\bx}\bu-\nabla_{\bx}\bU)\dxdt\\
+\int_0^t\int_\mathcal{C}(\nabla_{\bx}\times\bb-\nabla_{\bx}\times\bB):(\nabla_{\bx}\times\bb-\nabla_{\bx}\times\bB)\dxdt\\
\leq \int_0^t\int _\mathcal{S}\Big(\rho(\bU_t+\bv\cdot\nabla_{\bx}\bU)\cdot(\bU-\bu)-(p(\rho)-p(r))\mathrm{div}_{\bx}\bV
+\rho\boldsymbol{\omega}_{\bu}\times\bu \cdot \bU\\
+S(\nabla_{\bx}\bU):(\nabla_{\bx}\bU-\nabla_{\bx}\bu)+(r-\rho)\left(\frac{a\gamma}{\gamma-1}r^{\gamma-1}\right)_t+\nabla_{\bx}\left(\frac{a\gamma}{\gamma-1}r^{\gamma-1}\right)\cdot(r\bV-\rho\bv)\\
-((\nabla_{\bx}\times\bb)\times\bb)\cdot\bU\Big)\dxdt+\int_0^t\int_{\mathcal{C}}\Big(\bB_t\cdot({\bB-\bb})+(\boldsymbol{\omega}\times \bb)\cdot\bB+(\boldsymbol{\omega}\times \bx+\boldsymbol{\xi})\cdot\nabla_{\bx}\bb\cdot\bB\\
-(\bu\times \bb)\cdot (\nabla_{\bx}\times\bB)+\nabla_{\bx}\times\bB:(\nabla_{\bx}\times\bB-\nabla_{\bx}\times\bb)\Big)\dxdt,
\end{multline}
holds for almost all $t\in(0,T)$. Here $ r>0, r \in C^\infty{((0,T)\times\mathcal{S})},  \bU\in \mathcal{V}_{\mathcal{S}}, \bB \in C^\infty{((0,T)\times\mathcal{C})}$ and (for simplicity of notation) we assume $\varrho(0,\cdot) = \varrho_0$ and $(\varrho\bu)(0,\cdot) = \bq$.
\end{lemma}

\begin{proof}
We choose a test function $\varphi=\bU$ in $\eqref{weak-momentum}$ to obtain 
\begin{multline}\label{weak-momentum-1}
\left[\int_\mathcal{S}\rho\bu\cdot\bU(s)\ddx\right]_{s=0}^{s=t}-\int_0^t\int_{\mathcal{S}}\Big(\rho \bu\cdot\bU_t+\rho\bv\times\bu:\nabla_{\bx}\bU-\rho\boldsymbol{\omega_{\bu}}\times\bu\cdot\bU +p(\rho)\mathrm{div}_{\bx}\bU\\
-S(\nabla_{\bx}\bu):\nabla_{\bx}\bU+((\nabla_{\bx}\times\bb)\times\bb)\cdot\bU\Big)\dxdt=0,
\end{multline}
and pick a test function $\eta=\bB$ in $\eqref{weak-magnetic}$ to derive 
	\begin{multline}\label{weak-magnetic-1}
\left[\int_\mathcal{C}\bb\cdot\bB(s)\ddx\right]_{s=0}^{s=t}-\int_0^t\int_{\mathcal{C}}\Big(\bb\cdot\bB_t-(\boldsymbol{\omega}\times \bb)\cdot\bB+(\boldsymbol{\omega}\times \bx+\boldsymbol{\xi})\cdot\nabla_{\bx}\bb\cdot\bB\\
+(\bu\times \bb)\cdot (\nabla_{\bx}\times\bB)-(\nabla_{\bx}\times\bb):(\nabla_{\bx}\times\bB)\Big)\dxdt=0.
\end{multline}		
Furthermore, we select $\phi=\frac{1}{2}\bU^2$ in $\eqref{weak-density}$ to derive that 
	\begin{equation}\label{weak-density-1}
\begin{split}
\left[\frac{1}{2}\int_{\mathcal{S}}\rho\left|\bU\right|^2(s)\ddx\right]_{s=0}^{s=t}-\int_0^t\int_{\mathcal{S}}(\rho\bU\cdot\bU_t+\rho\bv\cdot\nabla_{\bx}\bU\cdot\bU)\dxdt=0
\end{split}       
           \end{equation}		                                                     

We combine $\eqref{energy-inequality}$, $\eqref{weak-momentum-1}$, $\eqref{weak-magnetic-1}$ and $\eqref{weak-density-1}$ to get 
\begin{multline*}
\left[\int_{\mathcal{S}}\left(\frac{1}{2}\rho|\bu-\bU|^2+\frac{1}{\gamma-1}p(\rho)\right)(s)\ddx+\int_{\mathcal{C}}\frac{1}{2}|\bb-\bB|^2(s)\ddx\right]_{s=0}^{s=t}\\
+\int_0^t\int_\mathcal{S}(S(\nabla_{\bx}\bu)-S(\nabla_{\bx}\bU)):(\nabla_{\bx}\bu-\nabla_{\bx}\bU)\dxdt\\
+\int_0^t\int_\mathcal{C}(\nabla_{\bx}\times\bb-\nabla_{\bx}\times\bB):(\nabla_{\bx}\times\bb-\nabla_{\bx}\times\bB)\dxdt\\
\leq \int_0^t\int _\mathcal{S}\Big(\rho(\bU_t+\bv\cdot\nabla_{\bx}\bU)\cdot(\bU-\bu)-p(\rho)\mathrm{div}_{\bx}\bV
+\rho\boldsymbol{\omega}_{\bu}\times\bu \cdot \bU\\
+S(\nabla_{\bx}\bU):(\nabla_{\bx}\bU-\nabla_{\bx}\bu)-((\nabla_{\bx}\times\bb)\times\bb)\cdot\bU\Big)\dxdt\\
+\int_0^t\int_{\mathcal{C}}\Big(\bB_t\cdot({\bB-\bb})+(\boldsymbol{\omega}\times \bb)\cdot\bB-(\boldsymbol{\omega}\times \bx+\boldsymbol{\xi})\cdot\nabla_{\bx}\bb\cdot\bB\\
-(\bu\times \bb)\cdot (\nabla_{\bx}\times\bB)+\nabla_{\bx}\times\bB:(\nabla_{\bx}\times\bB-\nabla_{\bx}\times\bb)\Big)\dxdt.
\end{multline*}

Moreover, straightforward calculation and continuity equation give that 
\begin{equation*}
\begin{split}
\left[\int_{\mathcal{S}}\frac{1}{\gamma-1}\left(p'(r)r-p(r)\right)(s)\ddx\right]_{s=0}^{s=t}=\left[\int_{\mathcal{S}}ar^\gamma(s)\ddx\right]_{s=0}^{s=t}=\int_0^t\int_{\mathcal{S}}r\left(\frac{a\gamma}{\gamma-1}r^{\gamma-1}\right)_t\dxdt,
\end{split}
\end{equation*}
 and 
\begin{multline*}
-\left[\int_{\mathcal{S}}\frac{1}{\gamma-1}p'(r)\rho(s)\ddx\right]_{s=0}^{s=t}=\left[-\int_{\mathcal{S}}\frac{a\gamma}{\gamma-1}r^{\gamma-1}\rho(s)\ddx\right]_{s=0}^{s=t} \\
=\int_0^t\int_{\mathcal{S}}-\rho\left(\frac{a\gamma}{\gamma-1}r^{\gamma-1}\right)_t-\rho\nabla_{\bx}\left(\frac{a\gamma}{\gamma-1}r^{\gamma-1}\right)\cdot\bv\dxdt.
\end{multline*}

Integration by parts leads to 
\begin{equation*}
\begin{split}
\int_{\mathcal{S}}r\nabla_{\bx}\left(\frac{a\gamma}{\gamma-1}r^{\gamma-1}\right)\cdot \bV\ddx=\int_{\mathcal{S}}\nabla_{\bx}p(r)\cdot\bV\ddx=-\int_{\mathcal{S}}p(r)\mathrm{div}_{\bx}\bV,
\end{split}
\end{equation*}
which gives the demanded estimate.
\end{proof}

The proof of Theorem \ref{w.s.strong.thm} follows.
\begin{proof} We consider the relative entropy inequality where $r$ and $\bU$ is the strong solution.
	It follows from $\eqref{relative-entropy-inequality}$ that 
	\begin{multline}\label{relative-entropy-inequality-2}
	\left[\int_{\mathcal{S}}\left(\frac{1}{2}\rho|\bu-\bU|^2+\frac{1}{\gamma-1}(p(\rho)-p'(r)(\rho-r)-p(r))\right)(s)\ddx+\int_{\mathcal{C}}\frac{1}{2}|\bb-\bB|^2(s)\ddx\right]_{s=0}^{s=t}\\
	+\int_0^t\int_\mathcal{S}(S(\nabla_{\bx}\bu)-S(\nabla_{\bx}\bU)):(\nabla_{\bx}\bu-\nabla_{\bx}\bU)\dxdt\\
	+\int_0^t\int_\mathcal{C}(\nabla_{\bx}\times\bb-\nabla_{\bx}\times\bB):(\nabla_{\bx}\times\bb-\nabla_{\bx}\times\bB)\dxdt\\
	\leq \int_0^t\int_\mathcal{S}\rho(\bU_t+\bv\cdot\nabla_{\bx}\bU)\cdot(\bU-\bu)\dxdt-\int_0^t\int _\mathcal{S}(p(\rho)-p(r))\mathrm{div}_{\bx}\bV\dxdt
	+\int_0^t\int _\mathcal{S}\rho\boldsymbol{\omega}_{\bu}\times\bu \cdot \bU\dxdt\\
	+\int_0^t\int_\mathcal{S}S(\nabla_{\bx}\bU):(\nabla_{\bx}\bU-\nabla_{\bx}\bu)\dxdt+\int_0^t\int_\mathcal{S}(r-\rho)\left(\frac{a\gamma}{\gamma-1}r^{\gamma-1}\right)_t\dxdt\\
	+\int_0^t\int_\mathcal{S}\nabla_{\bx}\left(\frac{a\gamma}{\gamma-1}r^{\gamma-1}\right)\cdot(r\bV-\rho\bv)\dxdt\\
	-\int_0^t\int_\mathcal{S}((\nabla_{\bx}\times\bb)\times\bb)\cdot\bU\dxdt
	+\int_0^t\int_{\mathcal{C}}\Big(\bB_t\cdot({\bB-\bb})+(\boldsymbol{\omega}\times \bb)\cdot\bB+(\boldsymbol{\omega}\times \bx+\boldsymbol{\xi})\cdot\nabla_{\bx}\bb\cdot\bB\\
	-(\bu\times \bb)\cdot (\nabla_{\bx}\times\bB)+\nabla_{\bx}\times\bB:(\nabla_{\bx}\times\bB-\nabla_{\bx}\times\bb)\Big)\dxdt=:\sum_{i=1}^8\mathcal{J}_{2,i}.
	\end{multline}
	
	Before we give the specific estimates of $\mathcal{J}_{2,i}$ with $i=1,...,8$, let us introduce essential and residual parts for any function $f$ 
	\begin{equation*}
	\begin{split}
	f=f_{ess}+f_{res}, \hspace{0.2cm}f_{ess}=f{\chi}_{\frac{r_0}{2}\leq \rho\leq 2r_1},\hspace{0.2cm} f_{res}=f(1-{\chi}_{\frac{r_0}{2}\leq \rho\leq 2r_1}), 
\end{split}
\end{equation*}
where $r_0, r_1>0$.
	
Theorem $\ref{global-existence-strong-solution}$ reveals that there exist constants $r_0,r_1\in (0,\infty)$ such that $r_0<r<r_1$. Direct calculation gives that 
	\begin{equation}\label{rho-u-estimate}
\begin{split}
\mathcal{E}(\rho, \bu, \bb|r, \bU,\bB )\geq C(\rho|\bu-\bU|^2+|\rho-r|_{ess}^2+|\rho-r|^\gamma_{res}).
\end{split}
\end{equation}

We divide the following integration into three parts
\begin{equation*}
\begin{split}
&\int_0^t\int_{\mathcal{S}}|\rho-r||\bu-\bU|\dxdt\\
&=\left(\int_0^t\int_{\mathcal{S}\cap\{\frac{r}{2}\leq \rho\leq 2r\}}+\int_0^t\int_{\mathcal{S}\cap\{\rho>2r\}}+\int_0^t\int_{\mathcal{S}\cap\{\rho<\frac{r}{2}\}}\right)|\rho-r||\bu-\bU|\dxdt.
\end{split}
\end{equation*}

Then, direct computation and $\eqref{rho-u-estimate}$ give that 
\begin{multline*}
\int_0^t\int_{\mathcal{S}\cap\{\frac{r}{2}\leq \rho\leq 2r\}}|\rho-r||\bu-\bU|\dxdt
\leq C\int_0^t\int_\mathcal{S}(|\rho-r|_{ess}^2+\rho|\bu-\bU|^2)\dxdt\\
\leq C\int_0^t\int_{\mathcal{S}}\mathcal{E}(\rho,\bu|r,\bU)\dxdt,
\end{multline*}
\begin{multline*}
\int_0^t\int_{\mathcal{S}\cap\{\rho>2r\}}|\rho-r||\bu-\bU|\dxdt
\leq C\int_0^t\int_{\mathcal{S}\cap\{\rho>2r\}}\left(\rho|\bu-\bU|^2+\frac{(\rho-r)^2}{\rho}\right)\dxdt\\
\leq C\int_0^t\int_\mathcal{S}(\rho|\bu-\bU|^2+|\rho-r|^\gamma_{res})\dxdt
\leq C\int_0^t\int_{\mathcal{S}}\mathcal{E}(\rho,\bu|r,\bU)\dxdt,
\end{multline*}
	and 
\begin{multline}\label{less-part-rho}
\int_0^t\int_{\mathcal{S}\cap\{\rho<\frac{r}{2}\}}|\rho-r||\bu-\bU|\dxdt
\leq \epsilon \int_0^t\int_{\mathcal{S}\cap\{\rho<\frac{r}{2}\}}|\bu-\bU|^2+C\int_0^t\int_{\mathcal{S}\cap\{\rho<\frac{r}{2}\}}|\rho-r|^2\dxdt\\
\leq \epsilon \int_0^t\int_\mathcal{S}|\bu-\bU|^2\dxdt+C\int_0^t\int_\mathcal{S}|\rho-r|^\gamma\dxdt\\
\leq \epsilon \int_0^t\int_\mathcal{S}|\bu-\bU|^2\dxdt+C\int_0^t\int_\mathcal{S}\mathcal{E}(\rho,\bu|r,\bU)\dxdt.
\end{multline}

We are now in a position to estimate the first term on the right-hand side of $\eqref{less-part-rho}$.
Recalling the definitions $\bu=\bv+\boldsymbol{\omega_{\bu}}\times\bx+\boldsymbol{\xi}_{\bu}$ and  $\bU=\bV+\boldsymbol{\Omega}_{\bU}\times\bx+\boldsymbol{\Xi}_{\bU}$, it  holds that 
		\begin{equation}\label{bu-bU}
	\begin{split}		
	\int_{\mathcal{S}}|\bu-\bU|^2\dxdt\leq \int_{\mathcal{S}}\left(|\bv-\bV|^2+|(\boldsymbol{\omega}_{\bu}-\boldsymbol{\Omega}_{\bU})\times \bx +(\boldsymbol{\xi}_{\bu}-\boldsymbol{\Xi}_{\bU})|^2\right)\dxdt,
\end{split}
\end{equation}
	and 
	\begin{equation*}
	\begin{split}
	\int_{\mathcal{S}}		S(\nabla_{\bx}\bU-\nabla_{\bx}\bu):(\nabla_{\bx}\bU-\nabla_{\bx}\bu)\dxdt\geq C\int_{\mathcal{S}}|\nabla_{\bx}(\bv-\bV)|^2,
\end{split}
\end{equation*}
which together with  Poincar\'e's inequality yield 		
	\begin{equation}\label{bv-bV}
	\begin{split}
	\int_{\mathcal{S}}|\bv-\bV|^2\ddx\leq C	\int_{\mathcal{S}}		S(\nabla_{\bx}\bU-\nabla_{\bx}\bu):(\nabla_{\bx}\bU-\nabla_{\bx}\bu)\ddx. 
\end{split}	
\end{equation}
	
	Furthermore, the straightforward calculation gives that 
		\begin{equation*}
	\begin{split}
	|\boldsymbol{\omega}_{\bu}-\boldsymbol{\Omega}_{\bU}|^2\sim \int_{\mathcal{S}}|(\boldsymbol{\omega}_{\bu}-\boldsymbol{\Omega}_{\bU})\times \bx|^2\ddx\sim\int_{\mathcal{B}}|(\boldsymbol{\omega}_{\bu}-\boldsymbol{\Omega}_{\bU})\times \bx|^2\ddx, 
	\end{split}	
\end{equation*}
which implies that 
\begin{multline}\label{omega-Omega}
|\boldsymbol{\omega}_{\bu}-\boldsymbol{\Omega}_{\bU}|^2
\leq C\int_{\mathcal{S}}|(\boldsymbol{\omega}_{\bu}-\boldsymbol{\Omega}_{\bU})\times \bx|^2\ddx\\
\leq C\int_{\mathcal{B}}|(\boldsymbol{\omega}_{\bu}-\boldsymbol{\Omega}_{\bU})\times \bx|^2\ddx
\leq C\int_{\mathcal{S}}(\rho|\bu-\bU|^2+|\boldsymbol{\xi}_{\bu}-\boldsymbol{\Xi}_{\bU}|^2)\ddx,
\end{multline}
	and
\begin{multline}\label{omega-Omega-xi-Xi}
\int_{\mathcal{S}}|(\boldsymbol{\omega}_{\bu}-\boldsymbol{\Omega}_{\bU})\times\bx+(\boldsymbol{\xi}_{\bu}-\boldsymbol{\Xi}_{\bU})|^2\dxdt
\leq C\int_{\mathcal{B}}\left(|(\boldsymbol{\omega}_{\bu}-\boldsymbol{\Omega}_{\bU})\times\bx|^2+|\boldsymbol{\xi}_{\bu}-\boldsymbol{\Xi}_{\bU}|^2\right)\dxdt\\
\leq C\int_{\mathcal{S}}\left(\rho|\bu-\bU|^2+|\boldsymbol{\xi}_{\bu}-\boldsymbol{\Xi}_{\bU}|^2\right)\dxdt
\leq C\int_{\mathcal{S}}\left(\mathcal{E}(\rho, \bu|r, \bU)+|\boldsymbol{\xi}_{\bu}-\boldsymbol{\Xi}_{\bU}|^2\right)\dxdt.
\end{multline}

We apply Schwartz's inequality and $\eqref{rho-u-estimate}$ to get that 
\begin{multline}\label{xi-Xi}
\int_{\mathcal{S}}\rho|\boldsymbol{\xi}_{\bu}-\boldsymbol{\Xi}_{\bU}|^2\ddx \leq C|\boldsymbol{\xi}_{\bu}-\boldsymbol{\Xi}_{\bU}|^2
=C\left|\frac{1}{m_{\mathcal{B}}}\int_{\mathcal{C}}\left(\rho\bu-r\bU\right)\ddx\right|^2\\
\leq C\left|\frac{1}{m_{\mathcal{B}}}\int_{\mathcal{C}}\left(\rho(\bu-\bU)+(\rho-r)\bU\right)\ddx\right|^2
\leq C\left|\int_{\mathcal{C}}\sqrt{\rho}\sqrt{\rho}(\bu-\bU)\ddx\right|^2+C\left|\int_{\mathcal{C}}(\rho-r)_{ess}\bU\ddx\right|^2\\
+C\left|\int_{\mathcal{C}\cap{\{\rho>2r\}}}\sqrt{(\rho-r)}\sqrt{(\rho-r)}\bU\ddx\right|^2+C\left|\int_{\mathcal{C}\cap{\{\rho<\frac{r}{2}\}}}\sqrt{(r-\rho)}\sqrt{(r-\rho)}\bU\ddx\right|^2\\
\leq C\int_{\mathcal{S}}\mathcal{E}(\rho,\bu|r,\bU)\ddx.
\end{multline}

The above estimates yield
\begin{multline}\label{rho-bu-r-BU}
\int_0^t\int_{\mathcal{S}}|r-\rho||\bu-\bU|\dxdt\\
\leq C\int_0^t\int_\mathcal{S}\mathcal{E}(\rho,\bu|r, \bU)\dxdt+\epsilon \int_0^t\int_{\mathcal{S}}S(\nabla_{\bx}(\bU-\bu)):(\nabla_{\bx}\bU-\nabla_{\bx}\bu)\dxdt.
\end{multline}

Since $(r, \bU,\bB)$  is a strong solution to the momentum equations, we derive that 
\begin{multline*}
	\mathcal{J}_{2,1}=\int_0^t\int_\mathcal{S}\frac{\rho-r}{r}(r\bU_t+r\bV\cdot\nabla_{\bx}\bU)\cdot(\bU-\bu)\dxdt
	+\int_0^t\int_{\mathcal{S}}p(r)\mathrm{div}_{\bx}(\bU-\bu)\dxdt\\
	-\int_0^t\int_{\mathcal{S}}S(\nabla_{\bx}\bU):(\nabla_{\bx}\bU-\nabla_{\bx}\bu)\dxdt\\
	-\int_0^t\int_{\mathcal{S}}(\bB\otimes\bB-\frac{1}{2}|\bB|^2\mathbb{I}):(\nabla_{\bx}\bU-\nabla_{\bx}\bu)\dxdt-\int_0^t\int_{\mathcal{S}}r\boldsymbol{\Omega}_{\bU}\times\bU\cdot(\bU-\bu)\dxdt\\
	-\int_0^t\int_\mathcal{S}\rho(\bV-\bv)\cdot\nabla_{\bx}\bU\cdot (\bU-\bu)\dxdt
	=:\sum_{i=1}^6	\mathcal{J}_{2,1,i}.
\end{multline*}

It follows from	$\eqref{rho-bu-r-BU}$ that 
\begin{equation*}
	\begin{split}
	\mathcal{J}_{2,1,1}\leq C\int_0^t\int_{\mathcal{S}}\mathcal{E}(\rho, \bu, \bb|r, \bU,\bB )\dxdt+\epsilon \int_0^t\int_{\mathcal{S}}S(\nabla_{\bx}(\bU-\bu)):(\nabla_{\bx}\bU-\nabla_{\bx}\bu)\dxdt.
					\end{split}	
\end{equation*}

H\"older's inequality, $\eqref{omega-Omega-xi-Xi}$, and $\eqref{xi-Xi}$ lead to 	
\begin{multline*}
	\mathcal{J}_{2,1,6}\leq 
	C\int_0^t\int_{\mathcal{S}}(\rho|\bv-\bV|^2+\rho|\bu-\bU|^2)\dxdt\\
	\leq 
	C\int_0^t\int_{\mathcal{S}}\rho|\bv-\bV|^2\dxdt+C\int_0^t\int_{\mathcal{S}}\mathcal{E}(\rho,\bu|r,\bU)\dxdt\\
	\leq C\int_0^t\int_\mathcal{S}\left(\rho|\bu-\bU|^2+\rho|\boldsymbol{\omega}_{\bu}\times\bx-\boldsymbol{\Omega}_{\bU}\times\bx|^2+\rho|\boldsymbol{\xi}_{\bu}-\boldsymbol{\Xi}_{\bU}|^2\right)\dxdt\\
	+C\int_0^t\int_{\mathcal{S}}\mathcal{E}(\rho, \bu, \bb|r, \bU,\bB )\dxdt\leq C\int_0^t\int_{\mathcal{S}}\mathcal{E}(\rho, \bu, \bb|r, \bU,\bB )\dxdt.
\end{multline*}

By $\eqref{rho-bu-r-BU}$, $\eqref{omega-Omega}$, $\eqref{omega-Omega-xi-Xi}$ and $\eqref{xi-Xi}$, we derive that 
\begin{multline*}
	\mathcal{J}_{2,1,5}+\mathcal{J}_{2,3}=\int_0^t\int_{\mathcal{S}}-r\boldsymbol{\Omega}_{\bU}\times \bU\cdot(\bU-\bu)+\rho\boldsymbol{\omega}_{\bu}\times\bu\cdot(\bU-\bu)\dxdt\\
	=\int_0^t\int_{\mathcal{S}}\Big((\rho-r)\boldsymbol{\Omega}_{\bU}\times\bU\cdot(\bU-\bu)+\rho(\boldsymbol{\omega_{\bu}}-\boldsymbol{\Omega_{\bU}})\times \bU\cdot(\bU-\bu)\\
	+\rho\boldsymbol{\omega}_{\bu}\times(\bu-\bU)\cdot(\bU-\bu)\Big)\dxdt\\
 \leq C\int_0^t\int_{\mathcal{S}}\mathcal{E}(\rho,\bu|r,\bU)\dxdt+\epsilon \int_0^t\int_{\mathcal{S}}S(\nabla_{\bx}(\bU-\bu)):(\nabla_{\bx}\bU-\nabla_{\bx}\bu)\dxdt.
\end{multline*}

 $\bB$ satisfies the induction equation and thus
\begin{multline*}
	\mathcal{J}_{2,7}+\mathcal{J}_{2,8}+\mathcal{J}_{2,1,4}=\int_0^t\int_{\mathcal{C}}[(\boldsymbol{\omega}_{\bu}\times\bb)\cdot\bB-(\boldsymbol{\Omega}_{\bU}\times\bB)\cdot(\bB-\bb)]\dxdt\\
	+\int_0^t\int_{\mathcal{C}}[(\boldsymbol{\Omega}_{\bU}\times\bx+\boldsymbol{\Xi}_{\bU})\cdot\nabla_{\bx}\bB\cdot(\bB-\bb)-(\boldsymbol{\omega}_{\bu}\times\bx+\boldsymbol{\xi}_{\bu})\cdot\nabla_{\bx}\bb\cdot\bB]\dxdt\\
	+\int_0^t\int_{\mathcal{C}}\Big(\bU\times\bB)\cdot(\nabla_{\bx}\times\bB-\nabla_{\bx}\times\bb)+(\bU\times\bb)\cdot(\nabla_{\bx}\times\bb)\\
	-(\bu\times\bb)\cdot(\nabla_{\bx}\times\bB)-((\bU-\bu)\times\bB)\cdot(\nabla_{\bx}\times\bB)\Big)\dxdt
	=:\sum_{i=1}^3\mathcal{K}_{i}. 
\end{multline*}
	
 By means of $\eqref{omega-Omega}$, $\eqref{omega-Omega-xi-Xi}$ and $\eqref{xi-Xi}$, we get 
\begin{multline*}
\mathcal{K}_1=\int_0^t\int_{\mathcal{C}}(\boldsymbol{\omega}_{\bu}-\boldsymbol{\Omega}_{\bU})\times(\bb-\bB)\cdot\bB\dxdt
\leq C\left|\boldsymbol{\omega}_{\bu}-\boldsymbol{\Omega}_{\bU}\right|^2+C\int_0^t\int_{\mathcal{C}}|\bb-\bB|^2\dxdt\\
\leq C\int_0^t\int_{\mathcal{S}}\mathcal{E}(\rho, \bu, \bb|r, \bU,\bB )\dxdt,
\end{multline*}
	and
\begin{multline*}
		\mathcal{K}_2=\int_0^t\int_{\mathcal{C}}((\boldsymbol{\Omega}_{\bU}-\boldsymbol{\omega}_{\bu})\times\bx+(\boldsymbol{\Xi}_{\bU}-\boldsymbol{\xi}_{\bu}))\cdot\nabla_{\bx}(\bb-\bB)\cdot\bB\dxdt\\
		=-\int_0^t\int_{\mathcal{C}}((\boldsymbol{\Omega}_{\bU}-\boldsymbol{\omega}_{\bu})\times\bx+(\boldsymbol{\Xi}_{\bU}-\boldsymbol{\xi}_{\bu}))\cdot\nabla_{\bx}\bB\cdot(\bb-\bB)\dxdt\\
		\leq C\int_0^t\int_{\mathcal{C}}(|(\boldsymbol{\Omega}_{\bU}-\boldsymbol{\omega}_{\bu})\times\bx+(\boldsymbol{\Xi}_{\bU}-\boldsymbol{\xi}_{\bu})|^2+|\bb-\bB|^2)\dxdt\\
		\leq C\int_0^t\int_{\mathcal{S}}\mathcal{E}(\rho, \bu, \bb|r, \bU,\bB )\dxdt.
\end{multline*}

By virtue of $\eqref{bu-bU}$, $\eqref{bv-bV}$, $\eqref{omega-Omega-xi-Xi}$ and $\eqref{xi-Xi}$, it implies that 
\begin{multline*}
\mathcal{K}_3=\int_0^t\int_{\mathcal{C}}[\bU\times(\bb-\bB)\cdot(\nabla_{\bx}\times(\bb-\bB))+(\bu-\bU)\times(\bB-\bb)\cdot(\nabla_{\bx}\times\bB)]\dxdt\\
\leq \epsilon \int_0^t\int_{\mathcal{C}}|\nabla_{\bx}\times(\bb-\bB)|^2\dxdt+C\int_0^t\int_{\mathcal{C}}(|\bb-\bB|^2+|\bu-\bU|^2)\dxdt\\
\leq \epsilon \int_0^t\int_{\mathcal{C}}[|\nabla_{\bx}\times(\bb-\bB)|^2+S(\nabla_{\bx}\bU-\nabla_{\bx}\bu):(\nabla_{\bx}\bU-\nabla_{\bx}\bu)]\dxdt\\
+C\int_0^T\int_{\mathcal{S}}\mathcal{E}(\rho, \bu, \bb|r, \bU,\bB )\dxdt.
\end{multline*}

We put above estimates into $\eqref{relative-entropy-inequality-2}$ to obtain that 
\begin{multline}\label{relative-estimates}
\left[\int_{\mathcal{S}}\mathcal{E}(\rho, \bu, \bb|r, \bU,\bB )\ddx\right]_{s=0}^{s=t}+\int_0^t\int_{\mathcal{S}}S(\nabla_{\bx}\bu-\nabla_{\bx}\bU):(\nabla_{\bx}\bu-\nabla_{\bx}\bU)\dxdt\\
+\frac{1}{2}\int_0^t\int_{\mathcal{C}}(\nabla_{\bx}\times\bb-\nabla_{\bx}\times\bB):(\nabla_{\bx}\times\bb-\nabla_{\bx}\times\bB)\dxdt\\
\leq C\int_0^t\int_{\mathcal{S}}\mathcal{E}(\rho, \bu, \bb|r, \bU,\bB )\dxdt\\
+C\underbrace{\int_0^t\int_{\mathcal{S}}\left(p(r)\mathrm{div}_{\bx}(\bU-\bu)-(p(\rho)-p(r))\mathrm{div}_{\bx}\bV+(r-\rho)\frac{p'(r)}{r}r_t+\frac{p'(r)}{r}\nabla_{\bx}r\cdot(r\bV-\rho\bv)\right)\dxdt}_{\mathcal{K}_4}.
\end{multline}

According to $\mathrm{div}_{\bx}(\bU-\bu)=\mathrm{div}_{\bx}(\bV-\bv)$, it holds that 
	\begin{multline*}
\int_0^t\int_{\mathcal{S}}p(r)\mathrm{div}_{\bx}(\bU-\bu)\dxdt=	\int_0^t\int_{\mathcal{S}}p(r)\mathrm{div}_{\bx}(\bV-\bv)\dxdt\\
=-\int_0^t\int_{\mathcal{S}}r\frac{p'(r)}{r}\nabla_{\bx}r\cdot(\bV-\bv)\dxdt, 
\end{multline*}
and we have that 
		\begin{equation*}
	\begin{split}
	\mathcal{K}_4=\int_0^t\int_{\mathcal{S}}\left((r-\rho)\frac{p'(r)}{r}(r_t+\bV\cdot \nabla_{\bx}r)-(p(\rho)-p(r))\mathrm{div}_{\bx}\bV+(\rho-r)(\bV-\bv)\cdot\nabla_{\bx}r\frac{p'(r)}{r}\right)\dxdt.
	 \end{split}	
\end{equation*}

Similarly to $\mathcal{J}_{2,1,1}$, we get 
	\begin{multline}\label{rho-r-bV-bv}
\int_0^t\int_{\mathcal{S}}	(\rho-r)(\bV-\bv)\cdot\nabla_{\bx}r\frac{p'(r)}{r}\dxdt\\
\leq  C\int_0^t\int_{\mathcal{S}}\mathcal{E}(\rho,\bu|r,\bU)\dxdt+\epsilon \int_0^t\int_{\mathcal{S}}S(\nabla_{\bx}(\bu-\bU)):(\nabla_{\bx}\bu-\nabla_{\bx}\bU)\dxdt.
\end{multline}

Since $r$ is a strong solution to the continuity equation, we derive that	
\begin{multline}\label{rest-rho}
	\int_0^t\int_{\mathcal{S}}\left((r-\rho)\frac{p'(r)}{r}(r_t+\bV\cdot \nabla_{\bx}r)-(p(\rho)-p(r))\mathrm{div}_{\bx}\bV\right)\dxdt\\
	=-\int_0^t\int_{\mathcal{S}}\mathrm{div}_{\bx}\bV(p(\rho)-p'(r)(\rho-r)-p(r))\dxdt\leq C\int_0^t\int_{\mathcal{S}}\mathcal{E}(\rho, \bu, \bb|r, \bU,\bB )\dxdt.
	\end{multline}
	
Estimates $\eqref{relative-estimates}$, $\eqref{rho-r-bV-bv}$, and $\eqref{rest-rho}$ together with Gr\"onwall's inequality yield the demanded claim.
\end{proof}

\bibliographystyle{siam}

\end{document}